\documentclass[10pt,reqno]{amsart}

\usepackage{amsfonts,amsmath}
\usepackage[dvips]{epsfig} 
\graphicspath{{./Fig_eps/}{./Eps/}} 
\DeclareGraphicsExtensions{.ps,.eps} 
\usepackage{color} 
\usepackage{fullpage}
\usepackage{latexsym,amsmath} 

\usepackage{bbm}

\usepackage{times}

\numberwithin{equation}{section}

\usepackage{enumerate}

\def\vec#1{\mathbf{#1}}

\newcommand{\thet}{\vartheta}
\newcommand{\nix}{\,\cdot\,}

\newcommand{\spins}{\{0,1,*\}}
\newcommand{\reds}{\cbc{(\rot,\rot),(\rot,\cyan),(\rot,\y),(\cyan,\rot),(\y,\rot)}}

\newcommand{\Pspins}{\mathcal P(\{0,1,*\})}
\newcommand{\Pcols}{\mathcal P(\cols)}
\newcommand{\fent}{f_{\mathrm{ent}}}
\newcommand{\focc}{f_{\mathrm{occ}}}
\newcommand{\Focc}{F_{\mathrm{occ}}}
\newcommand{\Fent}{F_{\mathrm{ent}}}
\newcommand{\focct}{f_{\mathrm{occ},t}}
\newcommand{\Foct}{F_{\mathrm{occ},t}}
\newcommand{\Fvall}{F_{\mathrm{val},\ell}}
\newcommand{\Fval}{F_{\mathrm{val}}}
\newcommand{\fval}{f_{\mathrm{val}}}
\newcommand{\fvall}{f_{\mathrm{val},\ell}}
\newcommand{\fdisc}{f_{\mathrm{disc}}}
\newcommand{\Fdisc}{F_{\mathrm{disc}}}

\newcommand{\phioct}{\varphi_{\mathrm{occ},t}}
\newcommand{\phivall}{\varphi_{\mathrm{val},\ell}}

\newcommand{\Om}{\Omega}

\newcommand{\cols}{\{\rot,\blau,\grun,\y\}}

\newcommand{\cyan}{\mathtt{c}}
\newcommand{\grun}{\mathtt{g}}
\newcommand{\y}{\mathtt{y}}
\newcommand{\purpur}{\mathtt{p}}
\newcommand{\rot}{\mathtt{r}}
\newcommand{\blau}{\mathtt{b}}

\newcommand{\rcond}{r_{k-\mathrm{cond}}}

\newcommand\KL[2]{D_{\mathrm{KL}}\bc{{{#1}\|{#2}}}}
 
\newcommand\wMAJ{w_{\mathrm{maj}}} 
\newcommand\bemph[1]{{\bf\em #1}}

\newcommand\sigmaMAJ{\sigma_{\mathrm{maj}}}

\newcommand\aco[1]{}
\newtheorem{definition}{Definition}[section]
\newtheorem{example}[definition]{Example}
\newtheorem{claim}[definition]{Claim}
\newtheorem{remark}[definition]{Remark}
\newtheorem{theorem}[definition]{Theorem}
\newtheorem{lemma}[definition]{Lemma}
\newtheorem{proposition}[definition]{Proposition}
\newtheorem{corollary}[definition]{Corollary}

\newtheorem{fact}[definition]{Fact}

\newcommand\rk{r_{k\mathrm{-SAT}}}
\newcommand\sign{\mathrm{sign}}

\newcommand\id{\mathrm{id}}

\newcommand\PHI{\vec\Phi}

\newcommand\cA{\mathcal{A}} 
\newcommand\cB{\mathcal{B}} 
\newcommand\cC{\mathcal{C}} 
\newcommand\cD{\mathcal{D}} 
 
\newcommand\cG{\mathcal{G}} 
\newcommand\cE{\mathcal{E}} 
 
\newcommand\cN{\mathcal{N}} 
\newcommand\cQ{\mathcal{Q}} 
 
\newcommand\cS{\mathcal{S}} 
\newcommand\cT{\mathcal{T}} 
\newcommand\cI{\mathcal{I}} 
 
\newcommand\cJ{\mathcal{J}} 
\newcommand\cL{\mathcal{L}} 
\newcommand\cM{\mathcal{M}} 
\newcommand\cO{\mathcal{O}} 
\newcommand\cP{\mathcal{P}} 
\newcommand\cX{\mathcal{X}} 
\newcommand\cY{\mathcal{Y}} 
\newcommand\cV{\mathcal{V}} 
 
\newcommand\cZ{\mathcal{Z}} 

\def\cR{{\mathcal R}}
\def\cC{{\mathcal C}}
\def\cE{{\mathcal E}}

\newcommand\eul{\mathrm{e}} 
\newcommand\eps{\varepsilon} 
\newcommand\ZZ{\mathbf{Z}} 
\newcommand\Var{\mathrm{Var}} 

\DeclareMathOperator{\Erw}{\mathbb E}
\DeclareMathOperator{\pr}{\mathbb P}

\newcommand{\vecone}{\vec{1}}

\newcommand{\Vol}{\mathrm{Vol}}

\newcommand{\Po}{{\rm Po}} 
\newcommand{\Bin}{{\rm Bin}}

\newcommand{\bink}[2] {{{#1}\choose {#2}}}

\newcommand\ra{\rightarrow} 

\newcommand\bc[1]{\left({#1}\right)} 
\newcommand\cbc[1]{\left\{{#1}\right\}} 
\newcommand\bcfr[2]{\bc{\frac{#1}{#2}}} 
 
\newcommand\brk[1]{\left\lbrack{#1}\right\rbrack} 
 
\newcommand\norm[1]{\left\|{#1}\right\|} 
\newcommand\abs[1]{\left|{#1}\right|}

\newcommand\RR{\mathbf{R}} 
\newcommand\RRpos{\RR_{\geq0}} 
\newcommand{\Whp}{W.h.p.} 
\newcommand{\whp}{w.h.p.}

\newcommand{\Erdos}{Erd\H{o}s}
\newcommand{\Renyi}{R\'enyi}

\newcommand\Lem{Lemma}
\newcommand\Prop{Proposition}
\newcommand\Thm{Theorem}
\newcommand\Def{Definition}
\newcommand\Cor{Corollary}
\newcommand\Sec{Section}
\newcommand\Chap{Chapter}

\begin{document}

\title{The asymptotic $k$-SAT threshold}

\author[Amin Coja-Oghlan]{Amin Coja-Oghlan \and Konstantinos Panagiotou}
\thanks{Extended abstracts of this work appeared in the 
	Proceedings of the 45th Annual Symposium on the Theory of Computing (`STOC') 2013, 705-714, and in the
	Proceedings of the 46th Annual Symposium on the Theory of Computing (`STOC') 2014, 804--813.
		  	The research leading to these results has received funding from the European Research Council under the European Union's Seventh Framework
			Programme (FP/2007-2013) / ERC Grant Agreement n.\ 278857--PTCC}
\date{\today} 
 
\address{Amin Coja-Oghlan, {\tt acoghlan@math.uni-frankfurt.de}, Goethe University, Mathematics Institute, 10 Robert Mayer St, Frankfurt 60325, Germany.}

\address{Konstantinos Panagiotou,  {\tt kpanagio@math.lmu.de}, University of Munich, Mathematics Institute, Theresienstr.\ 39, 80333 M\"unchen, Germany}

\maketitle

\begin{abstract}
\noindent
Since the early 2000s physicists have developed an ingenious but non-rigorous formalism called the {\em cavity method}
to put forward precise conjectures on phase transitions in random problems
	[M.~M\'ezard, G.~Parisi, R.~Zecchina: Analytic and algorithmic solution of random satisfiability problems. Science {\bf 297} (2002) 812--815].
The cavity method predicts that the satisfiability threshold in the random $k$-SAT problem is
$\rk=2^k\ln2-\frac12(1+\ln 2)+\eps_k$, with $\lim_{k\ra\infty}\eps_k=0$
	[S.~Mertens, M.~M\'ezard, R.~Zecchina: Threshold values of random $K$-SAT from the cavity method. Random Struct.\ Alg.\ {\bf28} (2006) 340--373].
This paper contains a proof of the conjecture.
\hfill MSC: 60C05, 05C80.

\end{abstract}

\section{Introduction}\label{Sec_intro}

\noindent
For integers $k\geq3,N,M>0$ choose a Boolean formula   
	$\PHI=\PHI_k(N,M)=\PHI_1\wedge\cdots\wedge\PHI_M$ in conjunctive normal form
with clauses $\PHI_i=\PHI_{i1}\vee\cdots\vee\PHI_{ik}$, $\PHI_{ij}\in\{x_1,\neg x_1,\ldots,x_N,\neg x_N\}$
uniformly at random out of all $(2N)^{kM}$ possible such formulas.
Since the early 1990s experimental work has supported the hypothesis
that for any $k\geq3$ there is a {\em sharp threshold for satisfiability}~\cite{Cheeseman,KirkpatrickSelman}.
That is, there exists a number $\rk>0$ 
such that as the formula density $M/N$ passes $\rk$,
the probability of that the random formula $\PHI$ is satisfiable drops from asymptotically $1$ to asymptotically $0$ as $N\ra\infty$. 
An impressive bulk of theoretical work has since been devoted to establishing the existence and location of this threshold $\rk$
as well as the existence of similar ``satisfiability thresholds'' in other random constraint satisfaction problems (see, e.g., \cite{nature} and the
references therein).
In fact, pinning the satisfiability threshold $\rk$ has become one of the best-known benchmark problems in probabilistic combinatorics.

From its early days the random $k$-SAT problem has 
drawn the attention of statistical physicists.
Through the physics lens, random $k$-SAT is an example of a ``disordered system''.
Over the past decades, physicists have developed
a systematic albeit non-rigorous approach to this type of problem called the {\em cavity method}~\cite{MM}.
More specifically, the so-called {\em 1-step replica symmetry breaking} (``1RSB'') instalment of the cavity method, which is centered around the {\em Survey Propagation} message passing procedure~\cite{MPZ}, predicts that
\cite{Mertens}
	\begin{equation}\label{eq1RSBprediction}
	\rk=2^k\ln2-\frac{1+\ln2}2+o_k(1).
	\end{equation}

From the viewpoint of the cavity method as well as from a rigorous perspective, random $k$-SAT is by far the most
challenging problem among the standard examples of random CSPs.
The reason is that there is a fundamental asymmetry between the 
role that the Boolean values `true' and `false' play.
More specifically, consider the thought experiment of first generating a random formula $\PHI$ and then sampling a random satisfying assignment $\sigma$ of $\PHI$.
Then the local ``shape'' of $\PHI$ provides significant clues as to the probability that a given variable $x$ takes the value `true' under the random assignment $\sigma$.
For instance, if $x$ appears many more times positively than negatively in $\PHI$, then we should expect that the probability that $x$ takes the value `true' under $\sigma$
is greater than $1/2$.
This is in contrast to, e.g., the graph coloring problem, where all the colors have the same ``meaning''.
In fact, the probability that a given vertex takes a particular color in a random coloring is just uniform, simply because we can permute the color classes.
Similarly, the $k$-NAESAT (``Not-All-Equal-Satisfiability'') problem, which asks for a satisfying assignment whose
inverse assignment is also satisfying, is perfectly symmetric by its very definition.%
	\footnote{Formally, we could call a random CSP {\em symmetric} if in a random problem instance for each variable
		the marginal distribution over the possible values that the variable can take (`true' or `false' in satisfiability; the colors in graph coloring, etc.)
		converges to the uniform distribution.
		For a more detailed discussion of symmetry see Appendix~\ref{Apx_sym}. }

The inherent asymmetry 
is the reason why the gap between best previous upper and lower bounds
on the $k$-SAT threshold is significantly larger than in other well-studied random problems.
To elaborate, let us say that the random formula $\PHI$ enjoys a property $\cE$ {\em with high probability (\whp)} if $\lim_{N\ra\infty}\pr\brk{\PHI\in\cE}=1$.
Friedgut~\cite{Ehud} established the existence of {\em sharp threshold sequence} $\rk(N)$ for any $k\geq3$.
That is, for any fixed $\eps>0$, $\PHI=\PHI_k(N,M)$ is satisfiable \whp\ if $M/N<(1-\eps)\rk(N)$ and unsatisfiable \whp\ if $M/N>(1+\eps)\rk(N)$.
With respect to the location of $\rk(N)$, a ``first moment'' argument~\cite{KKKS} shows that
	\begin{equation}\label{eqKKKS}
	\limsup_{N\ra\infty}\rk(N)\leq2^k\ln2-\frac{1+\ln2}2+o_k(1).
	\end{equation}
This upper bound coincides with the prediction~(\ref{eq1RSBprediction}).
Furthermore, Achlioptas and Peres~\cite{yuval} used the ``second moment method'' to prove that
	\begin{equation}\label{eqyuval}
	\liminf_{N\ra\infty}\rk(N)\geq2^k\ln2-\frac{k\ln2}2-\bc{1+\frac{\ln 2}2}-o_k(1).
	\end{equation}
Thus, the upper bound~(\ref{eqKKKS}) and the lower bound~(\ref{eqyuval}) differ by $\frac{k\ln 2}2+\frac12+o_k(1)$, a gap that diverges as a function of $k$.
By comparison, in the (symmetric)
random graph $k$-coloring problem, the gap between the best lower and upper bounds is about $2\ln2-1\approx0.39$, i.e., 
	a small absolute constant~\cite{Danny}.
Moreover, in random $k$-NAESAT the best upper and lower bounds differ by a mere $\eps_k=2^{-(1+o_k(1))k}$, a term
that decays exponentially in terms of $k$~\cite{Catching}.
In the present paper we prove a corresponding result for the (asymmetric) random $k$-SAT problem.

\begin{theorem}\label{Thm_main}
There exists $\eps_k=o_k(1)$ such that
	\begin{equation}\label{eqThm_main}
	2^k\ln2-\frac{1+\ln2}2-\eps_k\leq\liminf_{N\ra\infty}r_k(N)\leq\limsup_{N\ra\infty}r_k(N)\leq2^k\ln2-\frac{1+\ln2}2+\eps_k.
	\end{equation}
\end{theorem}

\noindent
In fact, the proof of \Thm~\ref{Thm_main} shows that~(\ref{eqThm_main}) holds with {$\eps_k=2^{-k/2 + o(k)}$}.

\Thm~\ref{Thm_main} establishes~(\ref{eq1RSBprediction}) rigorously.
The proof is based on a novel
type of second moment argument that directly incorporates several insights from the cavity method as well as parts of the Survey Propagation calculations.
For instance, while in prior work~\cite{nae,yuval} the second moment method was applied to the
number of satisfying assignments (with certain additional ``symmetry properties''),
a crucial feature of the present approach is that it is based on a ``relaxed'' concept of satisfying assignments called {\em covers}.
This notion plays a key role in the 1RSB cavity method.
We expect that this idea generalizes to a host of other problems.

{
In comparison to the extended abstract versions~\cite{SAT14,SAT}, this full
version of the paper contains a more streamlined proof.
For instance, the definition of the random variable and the formulas that emerge in the first/second moment calculations are simpler.
Additionally, the proof is based primarily on analytic arguments, rather than a blend of analytic and combinatorial considerations; this enhanced argument yields the aforementioned explicit and exponentially small value for $\eps_k$. Finally, this paper corrects an error in the definition of the relevant random variables in~\cite{SAT14}, which mistakenly forced the first moment to be prohibitively small.

After this paper was submitted, in a remarkable work Ding, Sly and Sun~\cite{DSS3} proved the satisfiability conjecture for all $k \ge k_0$
	 for some (unspecified) constant $k_0 \ge 3$.
In fact, they established the location of the threshold $\rk$ for $k\geq k_0$, thereby verifying the 1RSB prediction.
Ding, Sly and Sun build on two key ideas from this paper (and introduce many new ones).
First, \cite{DSS3} harnesses the idea of representing covers by means of a ``color code'' on the edges of the bipartite factor graph  of the $k$-SAT formula (whose vertices
correspond to the variables and clauses).
Second, \cite{DSS3} uses the notion of  {\em judicious configurations}, a vital trick to keep the second moment under control
	in the asymmetric case (cf.\ Section~\ref{Sec_theRandomVar}).
In a nutshell, while in the present work we construct a random variable that incorporates one iteration of the Survey Propagation equations
	(corresponding to conditioning on the direct neighborhood of variables/clauses in the factor graph),
Ding, Sly and Sun manage to deal with any bounded number of iterations.
}

After a discussion of related work, in \Sec~\ref{Sec_SP} we give an outline of the main ideas behind the proof of \Thm~\ref{Thm_main}.
There we also elaborate on the physics intuition upon which the proof is based.

\section{Related work}\label{Sec_related}

\subsection{The physics perspective}
Originally motivated by the study of ``disordered systems'' such as glasses or spin glasses,
physicists have turned the cavity method into an analytic but non-rigorous machinery for
the study of  problems in which the interactions between variables are induced by a sparse random graph or hypergraph.
The random $k$-SAT problem is a prime example.
Additionally, the cavity method has been applied to a wealth of problems, ranging from classical physics models to low-density parity check codes to compressive sensing.
Hence the importance of providing a solid mathematical foundation for this approach.
For an excellent introduction to the physics work we refer to~\cite{MM}.

The cavity method comes in two installments.
In addition to the aforementioned 1RSB variant, there is a simpler version called the {\em replica symmetric ansatz}.
Its key ingredient is the {\em Belief Propagation} message passing technique.
Applied to the random $k$-SAT problem, the replica symmetric ansatz predicts upper and lower bounds, namely~\cite{Monasson}
	\begin{equation}\label{eqRSprediction}
	\rcond=2^k\ln2-\frac32\ln2-o_k(1)\leq\rk\leq2^k\ln2-\ln2/2.
	\end{equation}
However, the replica symmetric ansatz is insufficient to obtain the precise $k$-SAT threshold.
The reason for this is a phenomenon called {\em condensation}~\cite{pnas}, which we will also encounter in the proof of \Thm~\ref{Thm_main},
and which has a dramatic impact on the probabilistic nature of the problem.

The 1RSB cavity method can be used to put forward a prediction as to the precise  value of $\lim_{N\ra\infty}\rk(N)$ of the sharp
threshold sequence (which is not rigorously known to converge) for \emph{any} $k\geq3$.
This prediction comes in terms of the solution to an intricate fixed point problem on the (infinite-dimensional) space of probability measures on the 3-simplex~\cite{Mertens,MPZ}. A proof of this exact formula for any $k\geq3$	remains an open problem.

\subsection{Other rigorous work}

\noindent
This is one of the first papers to vindicate the 1RSB cavity method rigorously, and the first to do so in an asymmetric problem.
In~\cite{Catching} we obtained a  result similar to \Thm~\ref{Thm_main} for the (symmetric) random $k$-NAESAT problem. Of course, in symmetric problems
many of the maneuvers that we are going to have to go through (e.g., clause/variable types, see \Sec~\ref{Sec_theRandomVar}) are unnecessary.
Independently of the present work,
Ding, Sly and Sun~\cite{DSS1,DSS2} verified the 1RSB prediction in the random regular $k$-NAESAT problem (where
each variable appears exactly $d$ times), and in the independent set problem in random regular graphs.
Both of these problems are symmetric. 
The proofs in~\cite{DSS1,DSS2} are based on the second moment method applied to a notion of ``cover'' appropriate
	for NAESAT/independent sets, while~\cite{Catching} relies on an ad-hoc concept called ``heavy solutions''.
Furthermore, in~\cite{CEH} we applied the methods from~\cite{Danny} to obtain 
a precise result on the $k$-colorability ``threshold'' in random regular graphs for infinitely many values of $k$.

In all other random constraint satisfaction problems where the threshold for the existence of solutions is known it matches
the prediction of the replica symmetric version of the cavity method.
An example of this is the random $k$-XORSAT problem (random linear equations mod 2)~\cite{Dubois,PittelSorkin}.
Furthermore, the exact satisfiability threshold is known in random $2$-SAT~\cite{mick,Goerdt}.
This is, of course, a special case, as $2$-SAT admits a simple criterion for (un)satisfiability, on which the proofs 
hinge.
In several other examples the replica symmetric predictions have been validated rigorously (see, e.g.,~\cite[\Chap~15--17]{MM}).

As mentioned earlier, the best prior bounds on the $k$-SAT threshold were obtained by far simpler second moment arguments.
The use of the second moment method was pioneered in this context by Frieze and Wormald~\cite{FriezeWormald} and
Achlioptas and Moore~\cite{nae}, who got within (about) a factor of two of the $k$-SAT threshold.
Subsequently, this result was improved by Achlioptas and Peres~\cite{yuval},
who established the aforementioned lower bound~(\ref{eqyuval}).
In both of these papers the inherent asymmetry of the $k$-SAT problem is sidestepped
by applying the second moment method to a random variable that counts satisfying assignments with additional symmetry properties.
Indeed, \cite{nae} applies the second moment method to satisfying assignments $\sigma$ whose inverse assignment $\bar\sigma$ is also satisfying.
Moreover, in~\cite{yuval} symmetry is enforced by counting ``balanced'' satisfying assignments under which exactly half the literal occurrences in the formula are set to true.
However, as pointed out in~\cite{yuval}, it impossible to remove the $\frac{k\ln 2}2$ gap in~(\ref{eqyuval}) by considering such a symmetrized random variable.
The best current algorithms for random $k$-SAT find satisfying assignments \whp\
for densities up to $\approx 1.817\cdot 2^k/k$ (better for small $k$) resp.\ $2^k\ln k/k$ (better for large $k$)~\cite{BetterAlg,FrSu},
a factor of $\Theta(k/\ln k)$ below the satisfiability threshold.

The notion of covers, which plays a key role in the 1RSB cavity method, has so far received only  limited attention in rigorous work.
In an important conceptual contribution, Maneva, Mossel and Wainwright~\cite{MMW} introduced a similar concept (``core assignments'')
to show that (generalized) Survey Propagation can be viewed as Belief Propagation on a modified Markov random field.
Furthermore, Maneva and Sinclair~\cite{Maneva} used covers to prove a (conditional) upper bound on the 3-SAT threshold
in uniformly random formulas.
A similar method was applied in~\cite{ACOCovers} to the random graph coloring problem.

\section{Outline}\label{Sec_SP}

\subsection{The second moment method} 
As pointed out in the seminal paper by Achlioptas and Moore~\cite{nae},
the second moment method can be used to prove lower bounds on the $k$-SAT threshold.
The general strategy is as follows.
Suppose that $Y=Y(\PHI)\geq0$ is a random variable such that $Y(\PHI)>0$ only if $\PHI=\PHI_k(N,M)$ is satisfiable.
Assume, moreover, that there is a number $C=C(k)>0$ that may depend on $k$ but not on $n$ such that
	\begin{equation}\label{eqsmm}
	0<\Erw[Y^2]\leq C\cdot\Erw[Y]^2.
	\end{equation}
Then the {\em Paley-Zygmund inequality}
	$\pr\brk{Y>0}\geq\Erw[Y]^2/\Erw[Y^2]$
implies that
	\begin{equation}\label{eqPZ}
	\liminf_{n\ra\infty}\,\pr\brk{\PHI\mbox{ is satisfiable}}\geq\liminf_{n\ra\infty}\,\pr\brk{Y>0}\geq1/C>0.
	\end{equation}
The following consequence of Friedgut's sharp threshold theorem turns~(\ref{eqPZ}) into a lower bound on $\rk$.
From here on out, we always let $M=\lceil rN\rceil$ for some number $r>0$, the {\em density}, that remains fixed as $N\ra\infty$.

\begin{lemma}[\cite{Ehud}]\label{Lemma_Ehud}
If $r>0$ is such that $\liminf_{N\ra\infty}\pr\brk{\PHI\mbox{ is satisfiable}}>0$, then $\liminf_{N\ra\infty}\rk(N)\geq r$.
\end{lemma}

\noindent
Thus, we ``just'' need to come up with a random variable $Y$ that satisfies~(\ref{eqsmm}).

\subsection{The majority vote} 
The obvious candidate for such a random variable seems to be the total number $Z$ of satisfying assignments of  $\PHI$. 
Then the second moment $\Erw[Z^2]$ is nothing but the expected number of {\em pairs} of satisfying assignments.
In effect, a necessary condition for the success of the second moment method turn out to be that in a {\em random pair} $(\sigma,\tau)$ of satisfying assignments of $\PHI$,
$\sigma,\tau$ ``look uncorrelated''.
In particular, 
as shown in~\cite{nae,yuval},
(\ref{eqsmm}) can only hold if the average Hamming distance of $\sigma,\tau$ is $(1+o(1))\frac n2$. However, in random $k$-SAT this is not the case~\cite{nae}.
In effect, (\ref{eqsmm}) does {\em not} hold for $Y=Z$ for \emph{any} density $r>0$. 

As observed in~\cite{nae,yuval}, the source of these correlations is the asymmetry of the $k$-SAT problem.
More precisely, let  $D_{x_i}$ denote the {\em degree} of the variable $x_i$,
i.e., number of times that $x_i$ occurs positively in the formula $\PHI$, and let $D_{\neg x_i}$ be
the degree of $\neg x_i$, i.e., number of times that $x_i$ occurs negatively in $\PHI$.
Furthermore, consider the {\em majority vote} assignment $\sigmaMAJ$, where we let $\sigmaMAJ(x_i)=1$
if $D_{x_i}>D_{\neg x_i}$, $\sigmaMAJ(x_i)=0$ if $D_{x_i}<D_{\neg x_i}$, and, say, choose $\sigmaMAJ(x_i)\in\{0,1\}$ randomly if $D_{x_i}=D_{\neg x_i}$.
Here and throughout, we represent `true' by $1$ and `false' by $0$.
Clearly, if the only information that we are given about $\PHI$ is the literal degrees $D_{x_1},D_{\neg x_1},\ldots,D_{x_n},D_{\neg x_n}$,
then $\sigmaMAJ$ is the assignment with the greatest probability of being satisfying.
This is because $\sigmaMAJ$ maximizes the total number of true literal occurrences throughout the formula.
To be precise, out of the $kM$ literals a 
	$$\wMAJ=\frac1{kM}\sum_{i=1}^N\max\cbc{D_{x_i},D_{\neg x_i}}$$
fraction set to true under $\sigmaMAJ$.
Moreover, if we draw an assignment $\sigma$ at random, then the closer $\sigma$ is to $\sigmaMAJ$ in Hamming distance the larger the expected
number of true literal occurrences.
As a consequence, we expect that most satisfying assignments ``lean towards'' the majority assignment $\sigmaMAJ$.
This induces a subtle correlation amongst pairs of satisfying assignments,
which dooms the second moment method.

This issue was sidestepped in~\cite{nae,yuval} by considering an artificially symmetrized random variable.
For instance, Achlioptas and Moore~\cite{nae} apply the second moment method to the number $Z_{\mathrm{NAE}}$ of satisfying assignments
$\sigma:\{x_1,\ldots,x_N\}\ra\{0,1\}$ whose inverse assignment $\bar\sigma:x\mapsto1-\sigma(x)$ is satisfying as well.
Satisfying assignments of this type are called Not-All-Equal-assignments, because under $\sigma$ every clause must contain both a literal
that is true under $\sigma$ and one that is false.
Intuitively, the Not-All-Equal requirement prevents the assignments from pandering towards $\sigmaMAJ$, because moving $\sigma$
towards $\sigmaMAJ$ makes it less likely that $\bar\sigma$ is satisfying.
As a consequence, it turns out that $Z_{\mathrm{NAE}}$ satisfies~(\ref{eqsmm}) for densities $r\leq2^{k-1}\ln2-\frac{1+\ln2}2+o_k(1)$,
	about a factor of two below the $k$-SAT threshold.
Moreover, (\ref{eqsmm}) cannot hold for much larger densities, because for $r>2^{k-1}\ln2-\frac{\ln2}2+o_k(1)$,
the first moment $\Erw[Z_{\mathrm{NAE}}]$, and in effect $\pr\brk{Z_{\mathrm{NAE}}>0}$, tends to $0$ as $N\ra\infty$.

A more subtle approach was suggested by Achlioptas and Peres~\cite{yuval}.
They apply the second moment method to the number $Z_{\mathrm{bal}}$ of {\em balanced} satisfying assignments,
	i.e., satisfying assignments $\sigma$ such that the fraction of true literal occurrences is about $1/2$; formally,
		\begin{equation}\label{eqBalancedSAT}
		\frac1{KM}\sum_{i=1}^N\sigma(x_i)D_{x_i}+(1-\sigma(x_i))D_{\neg x_i}=\frac12+O(N^{-1/2}).
		\end{equation}
Technically, Achlioptas and Peres use an elegant weighting scheme to enforce~(\ref{eqBalancedSAT}).
The dominant contribution to $Z_{\mathrm{bal}}$ comes from satisfying assignments at Hamming distance about $N/2$ from $\sigmaMAJ$.
Thus, considering balanced assignments stems the drift towards the majority vote assignment.
The condition~(\ref{eqsmm}) holds for $r\leq 2^k\ln2-\frac{k\ln2}2-(1+\frac{\ln 2}2)-o_k(1)$.
Conversely, it is pointed out in~\cite{yuval}
that $\Erw[Z_{\mathrm bal}]$ tends to $0$ as $N\ra\infty$  for $r> 2^k\ln2-\frac{k\ln2}2-\frac{\ln 2}2+o_k(1)$.
Thus, to bridge the gap of about $\frac{k\ln2}2$ between this lower bound 
and the upper bound~(\ref{eqKKKS}), it is inevitable to deal with satisfying assignments that lean towards $\sigmaMAJ$.

\subsection{Condensation}
But according to the cavity method, near the $k$-SAT threshold satisfying assignments are
subject to far more severe correlations than just via the subtle drift towards $\sigmaMAJ$.
To explain this, we sketch the physics predictions~\cite{pnas}
as to the geometry of the set $\cS(\PHI)$ of satisfying assignments of $\PHI$.
According to the cavity method, already for densities $r>(1+o_k(1))2^k\ln k/k$, way below the $k$-SAT threshold,  \whp\ the set $\cS(\PHI)$
has a decomposition 
	$\textstyle\cS(\PHI)=\bigcup_{i=1}^\Sigma\cC_i$
into an exponential number $\Sigma=\exp(\Omega(N))$ of ``clusters'' $\cC_i$.
These clusters are well-separated.
That is, any two assignments in different clusters have Hamming distance $\Omega(N)$.
More specifically, 
if $\sigma_1,\ldots,\sigma_l\in\cS(\PHI)$ is a sequence of satisfying assignments such that $\sigma_1$ and $\sigma_l$ belong to different clusters,
then there is a step $1\leq i<l$ such that $\sigma_i$ and $\sigma_{i+1}$ have Hamming distance $\Omega(n)$.
Furthermore, within each cluster $\cC_i$ most variables (say, at least $0.99N$) are {\em frozen}, i.e.,
they take the same truth value under all the assignments in $\cC_i$.
Finally, each cluster is expected to be internally ``well-connected''.
That is, one can walk within the cluster $\cC_i$ from any $\sigma\in\cC_i$ to any other $\tau\in\cC_i$
by only altering, say, $\ln N$ variables at each step.
The existence of clusters and frozen variables has by now  been established rigorously~\cite{Barriers,fede,Molloy}.

As the density $r$ increases, both the individual cluster sizes and the total number of satisfying assignments decrease.
But the cavity method predicts that the total number of satisfying assignments drops at a faster rate~\cite{pnas}.
More specifically, the prediction is that there exists a critical density $\rcond=2^k\ln2-\frac32\ln2+o_k(1)$ such that for $r<\rcond$,
each cluster $\cC_i$ contains only an $\exp(-\Omega(N))$ fraction of the entire set $\cS(\PHI)$.
In effect,  if $r<\rcond$ and we draw two satisfying assignments $\sigma,\tau$ of $\PHI$ independently at random, then most likely they belong to different clusters.
Thus, we expect $\sigma,\tau$ to have a large Hamming distance.
In particular, it is conceivable that they ``look uncorrelated'', apart, of course, from the inevitable drift towards $\sigmaMAJ$.

By contrast,  for $\rcond<r<\rk$ the largest cluster is expected to contain a {\em constant}, i.e., $\Omega(1)$ fraction of the set $\cS(\PHI)$  \whp\
This phenomenon is called {\em condensation} in physics jargon.
Consequently, if we draw two satisfying assignments $\sigma,\tau$ independently at random, then there is a good chance that $\sigma,\tau$ belong
to the same cluster.
In that case, they will be {\em heavily} correlated, because they coincide on all variables that are frozen in that cluster.
Though there is currently no rigorous proof that condensation occurs in random $k$-SAT, the phenomenon has been
established rigorously in other, symmetric problems~\cite{Danny,Lenka}.

The correlations that condensation induces
not only derail the second moment method, but also the physicists' ``replica symmetric ansatz''.
The 1RSB cavity method surmounts this obstacle by switching to a different random variable, namely
	the number $\Sigma$ of {\em clusters}. 
Provably, $\Sigma$ must remain exponentially large \whp\ right up to the $k$-SAT threshold~\cite{fede}.
Hence, as clusters are well-separated, there might be a chance that two random clusters decorrelate,
even as two randomly chosen satisfying assignments do not.
We are going to turn this intuition into a rigorous proof.

To this end, we represent each cluster $\cC_i$ by a map $\zeta_i:\{x_1,\ldots,x_N\}\ra\cbc{0,1,*}$
in which each variable either takes a Boolean value $0,1$ or the ``joker value'' $*$.
The idea is that $\zeta_i(x_j)=1$ means that $x_j$ is frozen to the value $1$ in the cluster $\cC_i$.
Similarly, $\zeta_i(x_j)=0$ indicates that $x_j$ is frozen to $0$.
By contrast, $\zeta_i(x_j)=*$  means that $x_j$ is unfrozen in  $\cC_i$.
In other words, $x_j$ takes the value $1$ in some of the assignments in $\cC_i$ and the value $0$ in others.
Fortunately, there is a neat description of the resulting ``relaxed assignments'' that does not depend on a precise technical definition of ``clusters'', ``frozen variables'' etc. 

\begin{definition}[\cite{BZ,Maneva}]\label{Def_cover}
A map $\zeta:\{x_1,\ldots,x_N\}\ra\cbc{0,1,*}$ is a \bemph{cover} of  $\Phi=\Phi_1\wedge\cdots\wedge\Phi_m$ if the following two conditions are satisfied.
Extend $\zeta$ to a map from the set of literals to $\cbc{0,1,*}$ by letting $\zeta(\neg x_j)=\neg\zeta(x_j)$, with  $\neg 0=1,\neg1=0,\neg *=*$.
Then
\begin{description}
\item[CV1] 
	each clause either contains a literal that takes the value $1$ under $\zeta$, or two literals that take the value $*$,
\item[CV2] any literal $l$ such that $\zeta(l)=1$ occurs in a clause whose other literals are all set to $0$.
\end{description}
\end{definition}

In terms of the cluster intuition,  {\bf CV1} provides that each clause either contains one literal that is frozen to `true',
or at least two literals that are unfrozen (for no unfrozen literal $l$ may occur in a clause whose other $k-1$ literals are frozen to $0$, 
as that clause would freeze $l$ to $1$).
In addition, {\bf CV2} ensures that each variable mapped to $0$ or $1$ is frozen to this value, meaning that there is a clause $\Phi_i$ whose other $k-1$ literals are frozen to values that do not satisfy $\Phi_i$.
Hence, we expect that the clusters and covers of $\PHI$ are (essentially) in one-to-one correspondence, and our proof vindicates this notion.

The proof strategy in this work is to perform a second moment argument for the number of covers.%
	\footnote{Dimitris Achlioptas suggested the general strategy of applying the second moment method to ``covers'' as early as 2007/8.
		But at the time it was not clear (to us) 
		how to carry out such a second moment argument.}
Yet matters are far from straightforward as the asymmetry of the $k$-SAT problem implies, much like for satisfying assignments, that covers lean towards $\sigmaMAJ$ and thus are subtly correlated.
In effect, as we previously saw in the case of satisfying assignments, a ``vanilla'' second moment argument cannot succeed.%

To accommodate the drift towards $\sigmaMAJ$  we will employ the physicists' Survey Propagation technique. Survey Propagation is a message passing procedure for (heuristically) calculating the marginal probability that a fixed variable $x_j$ takes each value $0,1,*$
in a random cover $\vec\zeta$ of $\PHI$~\cite{BMZ,MM}.
The details of Survey Propagation are intricate (e.g., they involve  a seriously complicated
	fixed point problem on the space of probability measures on the 3-simplex), and the result is not explicit.
However, asymptotically the dominant terms result from the literal degrees $D_{x_j},D_{\neg x_j}$.
Indeed, for densities $\rcond<r<\rk$ Survey Propagation predicts that
	\begin{align}\label{eqSP0}
	\pr\brk{\vec\zeta(x_j)=z\, |\, D_{x_j},D_{\neg x_j}}&=
		\thet%
			^z(D_{x_j}-D_{\neg x_j})+o_k(2^{-k}),\qquad\mbox{where}\\
	\thet%
		^z(\delta)&=	\left\{
			\begin{array}{cl}
			\frac12+\frac{\delta}{2^{k+1}}-2^{-k-2}&\mbox{ if }z=1,\\
			\frac12-\frac{\delta}{2^{k+1}}-2^{-k-2}&\mbox{ if }z=0,\\
			2^{-k-1}&\mbox{ if }z=*.
			\end{array}
			\right.\label{eqSP}
	\end{align}
The probability term on the l.h.s.\ of~(\ref{eqSP0})
refers to choosing a random formula $\PHI$ and then a random cover $\vec\zeta$ of $\PHI$, given the degrees of $x_j,\neg x_j$.
The approximation~(\ref{eqSP0}) is expected to be valid so long as $|D_{x_j}-D_{\neg x_j}|=o_k(2^k)$,
a condition that holds \whp\ for the vast majority of the variables.
Observe that the formula~(\ref{eqSP}) is very much in line with our intuition that covers lean towards $\sigmaMAJ$.
In \Sec~\ref{Sec_theRandomVar} we are going to craft a random variable around~(\ref{eqSP}) that allows us to incorporate
this drift, and thus to perform a second moment argument for the number of covers.

\subsection{Preliminaries and notation}\label{Sec_pre}
We conclude this section by introducing some notation and a few basic facts that will be used repeatedly throughout the paper.
For a natural number $Q$ we denote by $\brk Q$ the set $\{1,\ldots,Q\}$.
Moreover, we continue to denote by $\PHI_i$ the $i$th clause of the random formula $\PHI$ and by $\PHI_{ij}$ the $j$th literal of $\PHI_i$ ($i\in[M],j\in[k]$).
Furthermore, we let $V=V(N)=\cbc{x_1,\ldots,x_N}$ be the set of variables of $\PHI$ and $L=L(N)=\cbc{x_1,\neg x_1,\ldots,x_N,\neg x_N}$  the set of literals.
For each literal $l\in L$ we let $\abs l$ signify the underlying variable;
	that is $|x_i|=|\neg x_i|=x_i$ for  $i\in[N]$.

Unless otherwise specified, we always assume that $k,N$ are sufficiently large for our various estimates to hold.
We use asymptotic notation with respect to both $N$ and $k$.
More precisely, the plain notation $f=O(g)$ denotes asymptotics in $N$, while asymptotics is $k$ is denoted by $f=O_k(g)$.
In addition to the standard symbols, $o,O,\Omega,\Theta$, we write
	$f=\tilde O_k(g)$ to denote the fact that there exist $k_1,C>0$ such that for all $k>k_1$ we have $|f(k)|\leq k^C|g(k)|$.
Similarly, $f=\tilde \Omega_k(g)$ signifies that 
there exist $k_1,C>0$ such that for all $k>k_1$ we have $f(k)\geq k^{-C}|g(k)|$.
In particular, $f=\tilde\Omega_k(1)$ means that 
there exist $k_1,C>0$ such that for all $k>k_1$ we have $f(k)\geq k^{-C}$.
Finally, we write $f\sim g$ for $f=(1+o(1))g$.

Additionally, to avoid rounding issues it will be convenient to use the following notation.
Fixing a large enough constant $C_k>0$, we write $f(N)\doteq g(N)$ if $\exp(-C_k/N)f(N)\leq g(N)\leq\exp(C_k/N)f(N)$.

For a finite set $\cX$ we let $\cP(\cX)$ denote the set of probability distributions on $\cX$.
We identify $\cP(\cX)$ with the set of all vectors $(p^x)_{x\in\cX}$ with entries $p^x\in[0,1]$ such that $\sum_{x\in\cX}p^x=1$.
For $p\in\cP(\cX)$ we let
	$$H(p)=-\sum_{x\in\cX}p^x\ln p^x$$
signify the entropy of $p$; here and throughout, we use the convention that $0\ln0=0$.
Further, if $p,q\in\cP(\cX)$, then
	$$\KL{q}{p}=\sum_{x\in\cX}q^x\ln\frac{q^x}{p^x}$$
denotes the Kullback-Leibler divergence of $q,p$ (with the usual convention that $0\ln\frac00=0$ and
that $\KL qp=\infty$ if there is $x\in\cX$ such that $q^x>0=p^x$).

If we fix an element $x_0\in\cX$, then a probability distribution $p\in\cP(\cX)$ is actually determined by the
vector $p_0=(p^x)_{x\in\cX\setminus\cbc{x_0}}$
	(because the entries $p^x$, $x\in\cX$, must sum to $1$).
Therefore, for notational convenience, we sometimes just write $p_0$ instead of $p$.
In particular, we use the shorthands $H(p_0)=H(p)$ and $\KL{q_0}{p_0}=\KL qp$ if $q\in\cP(\cX)$ is another probability distribution.
Thus, for $p,q\in\brk{0,1}$ we let
	\begin{align*}
	H(p)&=H(p,1-p)=-p\ln p-(1-p)\ln(1-p),\\
	\KL qp&=\KL{(q,1-q)}{(p,1-p)}=q\ln\frac qp+(1-q)\ln\frac{1-q}{1-p}.
	\end{align*}
We recall that the Kullback-Leibler divergence is non-negative and convex.
The derivatives of its generic summand are
	\begin{align}\label{eqD2psi0}
	\frac{\partial}{\partial q}q\ln\frac{q}{p}&=1+\ln\frac{q}{p},&\frac{\partial}{\partial p}q\ln\frac{q}{p}&=-\frac{q}{p},\\
	\frac{\partial^2}{\partial q^2}q\ln\frac{q}{p}&=\frac1{q},&\frac{\partial^2}{\partial p^2}q\ln\frac{q}{p}&=\frac{q}{p^2},&
		\frac{\partial^2}{\partial p\partial q}q\ln\frac{q}{p}&=-\frac{1}{p}.\label{eqD2psi}
	\end{align}
If $\cX=\cX_1\times\cX_2$ and $p=(p^{(x_1,x_2)})_{x_1\in\cX_1,x_2\in\cX_2}\in\cP(\cX)$, then for $A_1\subset \cX_1$, $A_2\subset\cX_2$ we let
	$$p^{A_1\nix}=\sum_{x_1\in A_1}\sum_{x_2\in\cX_2}p^{(x_1,x_2)},\qquad
		p^{\nix A_2}=\sum_{x_1\in\cX_1}\sum_{x_2\in A_2}p^{(x_1,x_2)}.$$
If $A_1=\cbc{a_1}$, then we just write $p^{a_1\nix}$ instead of $p^{\cbc{a_1}\nix}$, and similarly for $A_2$.
We also write $p^{x_1x_2}$ instead of $p^{(x_1,x_2)}$.

We will frequently use the following facts. The entropy function is well-known to yield the exponential part of the multinomial coefficient.
\begin{fact}\label{Fact_entropyFunction}
Let $\cX$ be a finite set and suppose that $(p_n)_{n}$ is a sequence of probability distributions on $\cX$ such that $np(x)$ is an integer for every $x\in\cX$
and all $n$.
Then
	\begin{align}\label{eqFact_entropyFunction1}
	\bink n{(n p_n^x)_{x\in\cX}}&=\exp(n H(p_n)+O(\ln n))\qquad\mbox{as }n\to\infty.
	\end{align}
If, furthermore, there is a fixed $\eps>0$ such that for all $n$ we have $\min_{x\in\cX}p_n^x>\eps$, then
	\begin{align}\label{eqFact_entropyFunction2}
	\bink n{(n p_n^x)_{x\in\cX}}&=\exp\big(n H(p_n)-((|\cX|-1)\ln n)/2+O(1)\big)\qquad\mbox{as }n\to\infty.
	\end{align}	
\end{fact}
\begin{proof}
Stirling's formula yields $n!\sim\sqrt{2\pi n}n^n\exp(-n)$.
Moreover, for all $x\in\cX$ such that $p^x>0$ we have~\cite{Robbins55}
	\begin{align}\label{eqFact_entropyFunction3}
	\sqrt{2\pi p^xn}(np^x)^{n p^x}\exp(-np^x)\leq(n p^x)!&\leq\sqrt{2\pi p^xn}(np^x)^{n p^x}\exp(1/(12 np^x)-np^x).
	\end{align}
Since $\cX$ is finite, multiplying (\ref{eqFact_entropyFunction3}) up over $x\in\cX$ and cancelling yields (\ref{eqFact_entropyFunction1}).
Now, assume that $\min_{x\in\cX}p_n^x>\eps$.
Then $np^x\geq\eps n$ for all $n$ and thus $\sqrt{2\pi p^xn}=\Theta(\sqrt n)$.
Hence, (\ref{eqFact_entropyFunction2}) follows from (\ref{eqFact_entropyFunction3}).
\end{proof}

The Kullback-Leibler divergence enters our analysis as the rate function of the multinomial distribution (cf.\ \cite[\Sec~2.1]{DZ}). Both assertions made below follow immediately from Fact~\ref{Fact_entropyFunction}.

\begin{fact}\label{Fact_binomialLargeDev}
Let $\cX$ be a finite set, let $q\in\cP(\cX)$ be a probability distribution such that $q^x>0$ for all $x\in\cX$
and let $(p_n)_{n}$ be a sequence of probability distributions on $\cX$ such that
$np_n^x$ is an integer for all $x\in\cX,n\geq1$.
Then
	\begin{align*}
	\bink n{(n p_n^x)_{x\in\cX}}\prod_{x\in\cX}(q^x)^{np_n^x}&=\exp(-n\KL{p_n}q+O(\ln n))\qquad\mbox{as }n\to\infty.
	\end{align*}
Moreover, if for a fixed $\eps>0$ we have $\min_{x\in\cX}p_n^x>\eps$ for all $n$, then
	\begin{align*}
	\bink n{(n p_n^x)_{x\in\cX}}\prod_{x\in\cX}(q^x)^{np_n^x}&=\exp(-n\KL{p_n}{q}-((|\cX|-1)\ln n)/2+O(1))\qquad\mbox{as }n\to\infty.
	\end{align*}		
\end{fact}
The following is a special case of
the local limit theorem for sums of independent random vectors from~\cite[\Thm~3]{DavisMcDonald} tailored for our needs.
\begin{theorem}\label{Lemma_LLT}
Let $\cX\subset\ZZ^d$ be a finite set and let $(X_n)_{n\geq1}$ be a sequence of i.i.d.\ random variables with values in $\cX$.
With $\vecone_j\in\ZZ^d$ denoting the vector whose $j$th component is $1$ and whose other components are $0$, assume
that there is a number $\alpha>0$ such that
	$$\forall n\geq1,j\in[d]:\max_{x\in\cX}\min\{\pr\brk{X_n=x},\pr\brk{X_n=x+\vecone_j}\}\geq\alpha.$$
Then for the sequence $(S_n)_{n\geq1}$ with $S_n=X_1+\cdots+X_n$ the following statement is true.
Let $\mu_n=\Erw[S_n]$ and let $\Sigma$ be the $d\times d$-covariance matrix of   $X_1$.
Let $\psi$ denote the density function of the normal distribution with mean $0$ and covariance matrix $\Sigma$.
Then
	$$\lim_{n\ra\infty}\sup_{s\in\ZZ^d}\abs{n^{d/2}\pr\brk{S_n=s}-\psi\bcfr{s-\mu_n}{\sqrt n}}=0.$$
\end{theorem}

We also need the following well-known {\em Chernoff bound} (e.g.,~\cite{JLR}).

\begin{lemma}\label{Lemma_Chernoff}
Let $\varphi(x)=(1+x)\ln(1+x)-x$.
Let $X$ be a binomial or a Poisson random variable with mean $\mu>0$. 
Then for any $t>0$ we have
	\begin{align*}
	\pr\brk{X>\Erw\brk X+t}&\leq\exp(-\mu\cdot\varphi(t/\mu)),&
	\pr\brk{X<\Erw\brk X-t}&\leq\exp(-\mu\cdot\varphi(-t/\mu)).
	\end{align*}
In particular, for any $t>1$ we have
	$\pr\brk{X>t\mu}\leq\exp\brk{-t\mu\ln(t/\eul)}.$
\end{lemma}

\noindent
If $A,B$ are $n\times n$ matrices, then $A\preceq B$ means that $B-A$ is positive semidefinite. {Finally, in Appendix~\ref{sec:notation} we present a listing of the most important pieces of notation that is used throughout the paper.}

\section{Colors, types and shades}\label{Sec_theRandomVar}

\noindent
The aim in this section is to design the random variable upon which the proof of \Thm~\ref{Thm_main} is based.
We also summarise the result of the first and the second moment analysis.
Let $M=\lceil rN\rceil$ for  $r=2^k\ln2-\frac{1+\ln2}2-\eps_k$ with $\eps_k=\tilde O_k(2^{-k/2})$.

\subsection{The pruning step}\label{Sec_pruning}

The approximate Survey Propagation formula~(\ref{eqSP0}) only applies to literals $l$ such that both $D_l,D_{\neg l}$ are close to their expected value $kr/2$.
However,  \whp\ the random formula $\PHI$ features a few literals whose degrees deviate from $kr/2$ significantly.
In fact, it is well-known that the random formula $\PHI$ can be viewed as the result of the following experiment, known as the ``Poisson cloning model''~\cite{Kim}.
First, we choose the vector $\vec D=(D_l)_{l\in L}$ of literal degrees.
Its distribution is described easily: let $\vec D'=(D_l')_{l\in L}$ be a family of independent Poisson variables, each with mean $kr/2$.
Then the distribution of $\vec D$ is identical to that of $\vec D'$ given $\sum_{l}D_l'=kM$.
Further, given $\vec D$, we obtain $\PHI$ as follows.
Let $\cL(\vec D)=\bigcup_{l}\cbc l\times[D_l]$ be a set that contains $D_l$ ``clones'' of each literal $l\in L$.
Moreover, let $\cI(M,k)=[M]\times[k]$ be a set representing the $kM$ ``slots'' in the formula where the literals are placed ($k$ slots for each clause).
Now, choose a bijection $\PHI(\vec D):\cI(M,k)\ra\cL(\vec D)$, $(i,j)\mapsto\PHI_{ij}(\vec D)$ uniformly at random.
Then we obtain  $\PHI$ by letting $\PHI_{ij}$ be the literal $l$ such that $\PHI_{ij}(\vec D)\in\cbc l\times[D_l]$.
Intuitively, one could think of $\cL(\vec D)$ as a deck of cards that contains $D_l$ copies of each literal $l$.
The random formula $\PHI$ is obtained by shuffling the cards and reading the literals out in the resulting order.

Since $\vec D$ is closely related to the vector $\vec D'$ of independent Poisson variables, the random formula $\PHI$ is likely to contain a small
but linear (in $N$) number of literals whose degrees deviate substantially from $kr/2$.
To get rid of these literals, we subject $\PHI$ to a pruning operation.
More precisely, we perform the following three steps.

\begin{description}
\item[PR1] Initially, let $U$ be the set of all variables $x$ such that
		\begin{equation}\label{eqPrune1a}
		\max\cbc{|D_x-kr/2|,|D_{\neg x}-kr/2|}>k^32^{k/2}.
		\end{equation}
\item[PR2] While there is a clause that features at least three variables from $U$,
		\begin{itemize}
		\item remove all such clauses from the formula, and
		\item add  to $U$ each variable $x$ 
			such that (in the reduced formula) either the degree of $x$ or the degree of $\neg x$
			differs by more than $k^32^{k/2}$ from $kr/2$.
		\end{itemize}
\item[PR3] Remove the variables in $U$ from all the remaining clauses.
\end{description}

Let $\PHI'$ denote the formula obtained 
via {\bf PR1--PR3} and let $ V'=V\setminus U$ be its variable set.
Let $L'=\cbc{x,\neg x:x\in V'}$ be the set of literals of $\PHI'$.
Moreover,  for $x\in V'$ let $ d_x, d_{\neg x}$ denote the degrees of the literals $x,\neg x$ in $\PHI'$.
By construction,
	\begin{equation}\label{eqDegreesAfterPruning}
	\textstyle | d_x-\frac12kr|,| d_{\neg x}-\frac12kr|\leq k^32^{k/2}\qquad\mbox{ for all $x\in V'$.}
	\end{equation}
The following proposition summarizes the effect of the pruning operation.

\begin{proposition}\label{Prop_pruning}
\Whp\ the random formula $\PHI$ has the following properties.
\begin{enumerate}
\item Any satisfying assignment $\sigma'$ of $\PHI'$ extends to a satisfying assignment of $\PHI$.
\item 
	We have $|V'|\geq(1-\exp(-k^2))N$ and $\sum_{x\not\in V'}D_x+D_{\neg x}\leq\exp(-k^2)N$.
\item If $d^+,d^-$ are integers such that $|d^+-kr/2|,|d^--kr/2|\leq k^32^{k/2}$, then
		\[
		\Omega(1)\leq\frac{\abs{\cbc{l\in L':d_l=d^+,d_{\neg l}=d^-}}}{2N}
		=
		\pr\brk{\Po(kr/2)=d^+}\pr\brk{\Po(kr/2)=d^-}+O_k(\exp(-k^2)).
	\]
\end{enumerate}
\end{proposition}

\noindent
The proof of \Prop~\ref{Prop_pruning}, which is very much based on standard arguments, can be found in  Appendix~\ref{Sec_Prop_pruning}.

Let $n=|V'|$.
We  assume without loss of generality that the variable set of $\PHI'$ is $ V'=\cbc{x_1,\ldots,x_{ n}}$.
Further, let us denote the clauses that the pruned formula $\PHI'$ consists of by $\PHI_1',\ldots,\PHI_{m}'$.
In particular, in the rest of the paper $m$ is going to signify the number clauses of $\PHI'$.
For each $i\in[m]$ we let $k_i\in\{k-2,k-1,k\}$ denote the length of $\PHI_i'$, i.e., the number of literals that the clause contains.
Let $\cD$ be the $\sigma$-algebra generated by the random variables $n,m,d_l,k_i$ ($l\in L'$, $i\in[m]$).
\Prop~\ref{Prop_pruning} implies that
	$\liminf_{n\ra\infty}\pr[\PHI\mbox{ is satisfiable}]\geq\liminf_{n\ra\infty}\Erw[\pr[\PHI'\mbox{ is satisfiable}|\cD]].$
Therefore, we are left to prove that
	\begin{equation}\label{eqCondOnDegs}
	\liminf_{n\ra\infty}\Erw\brk{\pr[\PHI'\mbox{ is satisfiable}|\cD]}>0.
	\end{equation}
By the principle of deferred decisions, the distribution of $\PHI'$ given $\cD$ can characterized be as follows.
\begin{fact}\label{Fact_pruning}
Given $\cD$, $\PHI'$ is a uniformly random formula with variables $x_1,\ldots,x_{n}$, literal degrees $d_{x_i},d_{\neg x_i}$,
and $m$ clauses of lengths $k_1,\ldots,k_{m}$.
\end{fact}

In light of Fact~\ref{Fact_pruning}, we can describe the distribution of $\PHI'$ by means of an experiment that
resembles the Poisson cloning model (or the
	``configuration model'' of random graphs, e.g.,~\cite{JLR}).
Let
	$\textstyle\cL'=\bigcup_{l\in L'}\cbc{l}\times[d_l]$
be a set that contains $d_l$ clones $(l,j)$, $j\in[d_l]$, of each literal $l$.
Moreover, let 
	$\textstyle\cI'=\bigcup_{i\in[m]}\cbc i\times\brk{k_i}$
be the set of all literal slots of $\PHI'$.
Given $\cD$, let
	\begin{equation}\label{eqRandomBijection}
	\hat\PHI:\cI'\ra\cL',\quad(i,j)\mapsto\hat\PHI_{ij}.
	\end{equation}
be a uniformly random bijection.
Then we obtain $$\textstyle\PHI'=\bigwedge_{i\in[m']}\bigvee_{j\in[k_i]}\PHI'_{ij}$$ by
letting $\PHI'_{ij}$ be the literal $l$ such that $\hat\PHI_{ij}\in\cbc l\times[d_l]$.

The rest of the paper is devoted to the proof of~(\ref{eqCondOnDegs}).
Throughout, we always use the characterization of $\PHI'$ by way of $\hat\PHI$.
It may be helpful to think of $\hat\PHI$ in graph-theoretic terms:
	$\hat\PHI$ is nothing but a (uniformly random) perfect matching between the set $\cI'$ of clause slots and the set $\cL'$ of literal clones.

\subsection{The color code}\label{Sec_colorCode}
To prove~(\ref{eqCondOnDegs}) we are going to perform a second moment argument over the number of covers of $\PHI'$. 
By comparison to satisfying assignments, covers involve one significant twist.
While condition {\bf CV1} is similar in spirit to the notion of a ``satisfying assignment'',
{\bf CV2} imposes the additional requirement  that each literal set to $1$ be ``frozen''.
In effect, \emph{critical clauses}, i.e., clauses that contain one literal set to $1$ while all other literals are set to $0$, play a special role:
	each literal that is set to $1$ must occur in one of them.

To accommodate the significance of critical clauses we introduce a ``color code''.
If $\zeta$ is a cover of $\PHI'$, then we use the perfect matching $\hat\PHI$ upon which $\PHI'$ is based
to extend $\zeta$ to a map $\xi$ from the set $\cL'$
of literal clones to the colors red, blue, green, yellow ($\rot,\blau,\grun,\y$, for short).
The semantics  is as follows.
All clones $(l,j)\in\cL'$ such that $\zeta(l)=*$ are colored green and all $(l,j)\in\cL'$ such that $\zeta(l)=0$ are colored yellow.
Moreover, clones $(l,j)$ such that $\zeta(l)=1$ are colored either red or blue:
if $(l,j)$ occurs in a critical clause then it is colored red, otherwise blue.
The colorings that emerge in this way admit the following neat characterization.

\begin{definition}\label{Def_shade}
A map $\xi:\cL'\ra\cbc{\rot,\blau,\grun,\y}$ is a \bemph{shade} 
	if the following conditions are satisfied.
\begin{description}
\item[SD1] For any literal $l\in L'$ exactly one of the following is true:
	\begin{itemize}
	\item all clones of both $l$ and $\neg l$ are colored green under $\xi$.
	\item all clones of $l$ are colored either red or blue, and all clones of $\neg l$ are colored yellow under $\xi$.
	\item all clones of $l$ are colored yellow, and all clones of $\neg l$ are colored red or blue under $\xi$.
	\end{itemize}
\item[SD2] There is no literal $l\in L'$ all  of whose clones are colored blue under $\xi$.
\end{description}
\end{definition}
\noindent
Condition {\bf SD2} is  to ensure that a literal set to $1$ is ``frozen'' by a critical clause, represented by a  red clone.

It will be convenient to introduce two additional colors:
	a clone is \emph{cyan} (`$\cyan$') if it is  blue or green and \emph{purple} (`$\purpur$') if it is red, blue or green.
Thus, $\cyan=\cbc{\blau,\grun}$, $\purpur=\cbc{\rot,\blau,\grun}$.
We will frequently work with vectors $q=(q^z)_{z\in\cols}$ (for example representing probability distributions) indexed by the above colors.
For such vectors let
	$$q^{1}=q^{\blau}+q^{\rot},\qquad q^{0}=q^{\y},\qquad q^{*}=q^{\grun},\qquad
	q^{\cyan} = q^{\{\blau, \grun\}}=q^{\blau}+q^{\grun},\qquad q^{\purpur}=q^{\{\rot,\blau,\grun\}}=q^{\rot}+q^{\blau}+q^{\grun}.$$
In view of this notation we may think of $1 = \{\rot, \blau\}$ as being an auxiliary color as well.
In terms of the coloring we can express easily when a shade $\xi$ corresponds to a cover. 

\begin{definition}\label{Def_valid} 
A shade $\xi$
	 is \bemph{valid} in $\hat\PHI$ if the following two conditions are satisfied.
\begin{description}
\item[V1] If a clause contains a red clone, then all its other clones are yellow.
\item[V2] Any clause without a red clone contains at least two cyan clones.
\end{description}
\end{definition}

\noindent
In particular, under a valid cover each clause contains at least one purple clone. Definition~\ref{Def_valid} ensures that a valid shade $\xi:\cL'\ra\cols$ gives rise
to a cover $\hat\xi:L'\ra\spins$ by setting
	$\hat\xi(l)=1$ if $\xi(l,1)\in\cbc{\rot,\blau}$, $\hat\xi(l)=0$ if $\xi(l,1)=\y$ and $\hat\xi(l)=*$ if $\xi(l,1)=\grun$.
Thus, there is a one-to-one correspondence between the valid shades of $\hat\PHI$ and the covers  of $\PHI'$.
Hence, we are going to perform a second moment argument for the number of valid shades of $\hat\PHI$.

\subsection{Types}\label{Sec_Types}

As explained in \Sec~\ref{Sec_SP},
a key issue with this idea is the drift towards the majority vote assignment.
To deal with this, we are going to define an appropriate ``slice'' within the set of all shades such that
two randomly chosen valid shades ``look uncorrelated''  {\em within this slice}.
In order to define the slice, we are going to assign to each literal a ``type'' that 
provides for each clone of that literal a probability distribution over $\cols$.
Additionally, we will assign each clause a type that comprises the types of the literals that the clause contains.
Ultimately, the construction will involve the Survey Propagation ``guess''~(\ref{eqSP}) as to
the marginal probability that a given literal is set to each of the values $0,1,*$ under a randomly chosen cover.
For the sake of clarity, we shall describe the  construction in relative generality and we will fix the parameters later. The starting point is the following definition.

\begin{definition}\label{Def_typeAssignement}
A \bemph{type assignment} of $\hat\PHI$ is a map
	$\theta:\cL'\ra\Pcols,\ (l,j)\mapsto\theta_{l,j} = (\theta_{l,j}^z)_{z \in \cols}$
that satisfies the following conditions:
	\begin{description}
	\item[TY1] for any $l\in L'$ and any $j,j'\in[d_l]$ we have $\theta_{l,j}^z=\theta_{l,j'}^z$ for all $z\in\cbc{0,1,*}$.
	\item[TY2] for any $l\in L'$, any $j\in[d_l]$ and any $j'\in[d_{\neg l}]$
		we have $\theta_{l,j}^*=\theta_{\neg l,j'}^*$ and $\theta_{l,j}^1=\theta_{\neg l,j'}^0$.
	\end{description}
\end{definition}

Thus, a type assignment maps each literal clone to a probability distribution over $\cols$.
The conditions {\bf TY1}--{\bf TY2}  provide a degree of consistency between the distributions
assigned to the clones of a literal $l\in L'$ and of the clones of $\neg l$.
Namely, {\bf TY1} provides that for any two $j,j'\in[d_l]$ we have $\theta_{l,j}^\y=\theta_{l,j'}^\y$, $\theta_{l,j}^\grun=\theta_{l,j'}^\grun$
and $\theta_{l,j}^\rot+\theta_{l,j}^\blau=\theta_{l,j'}^\rot+\theta_{l,j'}^\blau$.
Thus, only the partition of the probability mass between the colors $\rot,\blau$ may vary between the different clones of the same literal.
Additionally, {\bf TY2} ensures that the distributions assigned to the clones are in line with the notion that
the Boolean value assigned to $\neg l$ must be the opposite of that assigned to $l$.

\begin{example}\label{Ex_ideal}
The ideal example of a type assignment of $\hat\PHI$ is the following.
For each clone $(l,j)\in\cL'$ and every color $z\in\cols$,  let 
	$\theta_{l,j}^z$ be the number of valid shades $\xi$ of $\hat\PHI$ such that $\xi(l,j)=z$
		divided by the total number of valid shades
		(provided that it is positive).
In other words, $\theta_{l,j}^z$ is the \emph{marginal probability} that $(l,j)$ takes color $z$ in a randomly chosen valid shade of $\hat\PHI$. Clearly, this map satisfies {\bf TY1--TY2}. However, it is very difficult to get a handle on this ideal type assignment.
Therefore, we will ultimately use the Survey Propagation prediction~(\ref{eqSP}) to design an approximation.
\end{example}

Let $\theta$ be a type assignment of $\hat\PHI$.
The \emph{$\theta$-type} of a literal $l\in L'$ is
the tuple {$\theta_l=(d_l,d_{\neg l},(\theta_{l,j})_{j \in [d_l]},(\theta_{\neg l,j})_{j\in[d_{\neg l}]})$}. Thus, the $\theta$-type comprises the degree of $l$, the degree of its negation $\neg l$,
and the distributions on $\cols$ associated with each clone of $l$ and $\neg l$.
Let $T_\theta=\cbc{\theta_l:l\in L'}$ be the set of all $\theta$-types.
For each $t \in T_\theta$ we introduce the notation $d_t=d_l$, $t_j=\theta_{l,j}$ for $j\in[d_l]$,
and $\neg t=\theta_{\neg l}$, where $l$ is \emph{any} literal such that $t = \theta_l$.
Thus, $t_j\in\cP(\cols)$ for all $j\in[d_t]$.
Furthermore,  condition {\bf TY1} vindicates the notation
	$$t^1=t_{1}^1,\quad t^0=t_{1}^0,\quad t^*=t_{1}^*.$$
Thus, $(t^1,t^0,t^*)\in\cP(\spins)$.
Moreover, for $t\in T_\theta$ and $j\in[d_t]$ we let
$$
	L_t'=\cbc{l\in L':\theta_l=t}, \quad L_{t,j}'=\cbc{(l,j) \in {\mathcal L}': l\in L_t'},\quad
		n_t=\abs{L_t'},\quad
		\pi_t=\frac{n_t}{2n}.
$$
As a next step, we define clause types.
Let $i\in[m]$ and let $(l_{i,j},h_{i,j})=\hat\PHI_{ij}$ for $j=1,\ldots,k_i$.
Then we call
	$$\ell(i)=((\theta_{l_{i,1}},h_{i,1}),\ldots,(\theta_{l_{i,k_i}},h_{i,k_i}))$$
the \emph{$\theta$-type} of the clause $\PHI_i'$.
Thus, $\ell(i)$ contains the $\theta$-types of all the literals that appear in $\PHI_i'$, and also indicates which clone of a literal of that type appears in the clause.
Let $T^*_{\theta}=\cbc{\ell(i):i\in[m]}$. Further, for $\ell\in T^*_{\theta}$ let
$$
	M_\ell=\cbc{i\in[m]:\ell(i)=\ell},\quad
	m_\ell=\abs{M_\ell},\quad
	\pi_\ell=\frac{m_\ell}{m}.
$$
Thus, each clause type $\ell\in T^*_{\theta}$ is a tuple $((t(1),h(1)),\ldots,(t(k_\ell),h(k_\ell)))$
with $t(1),\ldots,t(k_\ell)\in T_\theta$ and $h(j)\in[d_{t(j)}]$ for $j\in[k_\ell]$.
We always write $k_\ell$ for the length of this tuple. 
Since $k_\ell$ is nothing but the length of any corresponding clause in $\PHI'$, the pruning step ensures
that $k_\ell\in\{k-2,k-1,k\}$ for all $\ell\in T^*_\theta$.
Further,  for $j\in[k_\ell]$ we write $\partial(\ell,j)=(t(j),h(j))$ for the $j$th component of $\ell$.
Additionally,  recalling that  $t_{h(j)}(j)$ is a probability distribution  on  $\cols$
for each $j\in[k_\ell]$,  we let {$\ell_j=t(j)_{h(j)}$}. Hence, $\ell_j\in\cP(\cols)$. 

In summary,  given a type assignment $\theta$, we  have assigned each literal and each clause a $\theta$-type.
The definition of the literal/clause types is such that the matching $\hat\PHI$  ``respects the types''.
More precisely, let $\ell\in T_\theta^*$ be a clause type.
Then for each $i\in M_\ell$, $j\in[k_\ell]$ we have
	\begin{equation}\label{eqTypePreserving}
	\hat\PHI_{ij}\in L_{\partial(\ell,j)}'.  
	\end{equation}
Conversely, for a literal type $t\in T_\theta$ and $h\in[d_t]$ we 
define
	$\partial(t,h)=\cbc{(\ell,j):\ell\in T_\theta^*,j\in[k_\ell],(t,h)=\partial(\ell,j)}.$
In words, $\partial(t,h)$ is the set pairs $(\ell,j)$ such that the $h$th clone of a literal of type $t$ may
appear in the $j$th position of a clause of type $\ell$.
Thus, we obtain a bipartite ``type graph''
whose vertices are the pairs $(t,h)$ with $t\in T_{\theta}$ and $h\in[d_t]$
and $(\ell,j)$ with $\ell\in T_\theta^*$ and $j\in[k_\ell]$.
Every vertex $(\ell,j)$ has a unique neighbor, namely  $\partial(\ell,j)$.
But for each vertex $(t,h)$ the neighborhood $\partial(t,h)$ may contain several vertices.

As a next step, we will explain how the literal/clause types identify a ``slice'' within the set of all valid shades.
The following definition basically provides that the empirical distribution of the colors is as prescribed by the types.

\begin{definition}\label{Def_thetaShade}
Let $\theta$ be a type assignment of $\hat\PHI$.
A shade $\xi$ is called a \bemph{$\theta$-shade}  of $\hat\PHI$ if the following conditions are satisfied.
\begin{enumerate}
\item For any $t\in T_\theta$, $h\in[d_t]$, $z\in\cols$ we have	
		$\abs{\cbc{l\in L_t':\xi(l,h)=z}}\doteq n_tt_h^z.$
\item For any $\ell\in T^*_{\theta}$, $j\in[k_\ell]$, $z\in\{\rot,\y\}$
	we have
	$|\{i\in M_\ell:\xi(\hat\PHI_{ij})=z\}|\doteq m_\ell\ell_j^z.$
\end{enumerate}
\end{definition}

\noindent
In words, in a $\theta$-shade
for each type $t\in T_\theta$, every $h\in[d_t]$ and all colors $z\in\cols$,
the fraction of literals $l$ of type $t$ whose $h$th clone is colored $z$ is (about) $t_h^z$.
Additionally, for each clause type $\ell$ and each $j\in[k_\ell]$ the
fraction of clauses of type $\ell$ whose $j$th clone has color $z\in\{\rot,\y\}$ is (approximately) equal to $\ell_j^z$.
{This second requirement corresponds to the ``judicious'' condition from~\cite{SAT}.
The purpose is to restrict the impact of asymmetry to direct neighborhoods.}

\subsection{An educated guess}
We are going to apply the second moment method to the number of valid  $\theta$-shades
	for a type assignment $\theta$
	that provides a good enough approximation to the ``ideal'' type assignment from Example~\ref{Ex_ideal}.
In this section we construct this type assignment. The starting point is the map $\thet:\ZZ\ra\Pspins$ from~(\ref{eqSP}).
Following the Survey Propagation intuition, for each literal $l$ we let $\thet_l=\thet(d_l-d_{\neg l})\in\Pspins$.
We call $\thet_l$ the \emph{signature of~$l$}.
Crucially, the signature of $l$ is determined by $d_l,d_{\neg l}$ only.

While $\thet_l$ is a distribution over $\spins$ for each literal $l$, our aim is to construct a type assignment that
provides a distribution over $\cols$ for each clone $(l,h)$.
This distribution will depend on the signatures of the other literals that get matched to the same clause as $(l,h)$.
More precisely, for $i\in[m]$  we call the vector
	$$\thet_i=(\thet_{\PHI_{i,1}'},\ldots,\thet_{\PHI_{i,k_i}'})\in\Pspins^{k_i}$$
the \emph{signature of $\PHI_i'$}.
In words, $\thet_i$ consists of the signatures of the $k_i$ literals that appear in clause $\PHI_i'$. In order to turn the signature $\thet_i$ into probability distributions on $\cols$, we define a map
	$$\Lambda:\bigcup_{\kappa=k-2}^k
		\cbc{\thet_l:l\in L'}^\kappa
		\to\bigcup_{\kappa=k-2}^k\cP(\cols)^\kappa,\quad
		(t_1,\ldots,t_\kappa)\mapsto
		\big(\Lambda_1(t_1,\ldots,t_\kappa),\ldots,\Lambda_\kappa(t_1,\ldots,t_\kappa)\big)$$
by letting for $j \in [\kappa]$  
	\begin{align}\label{Lemma_Lambda1}
	\Lambda_j^\rot(t_1,\ldots,t_\kappa)&=(t_j^1+t_j^*)\prod_{j'\neq j}t_{j'}^0,&
	\Lambda_j^\blau(t_1,\ldots,t_\kappa)&=t_j^1-(t_j^1+t_j^*)\prod_{j'\neq j}t_{j'}^0,\\
	\Lambda_j^\y(t_1,\ldots,t_\kappa)&=t_j^\y,&\Lambda_j^\grun(t_1,\ldots,t_\kappa)&=t_j^\grun.\label{Lemma_Lambda2}
	\end{align}
The definition is motivated by the fact that the $j$th clone of a clause must be colored red if all other clones are set to $0$.
Because $t_j^1,t_j^0=\frac12+\tilde O_k(2^{-k/2})$ and $t_j^*=O_k(2^{-k})$ for all $j$, $\Lambda_j(t_1,\ldots,t_\kappa)$ is a probability distribution for sufficiently large $k$. Moreover, $\Lambda_j^z(t_1,\ldots,t_\kappa)=t_j^z$ for $z\in\spins$.
Finally,  for $i\in[m]$ and $j\in[k_i]$ we define
	\begin{equation}\label{eqDefMyTheta}
	\theta_{\hat\PHI_{ij}}
		=\Lambda_j(\thet_i)\in\Pcols.
	\end{equation}
Thus, at this point we have constructed a type assignment $\theta=\theta(\hat\PHI):\cL'\ra\Pcols$.

In the rest of the paper, we are exclusively going to work with the type assignment from~(\ref{eqDefMyTheta}).
Therefore, we are consistently going to drop the index $\theta$ from symbols such as $T_\theta,T^*_\theta$ and just write $T,T^*$ etc.\ instead.
Having constructed the type assignment, we obtain the $\theta$-types of the literals/clauses via the framework described in the previous section.
Let $\cT\supset\cD$ denote the coarsest $\sigma$-algebra with respect to which all
types $\theta_l$, $\ell_i$ ($l\in L',i\in[m])$ are measurable.
The conditional distribution of the random formula $\hat\PHI$ given $\cT$ admits the following neat description
as a ``type-preserving random matching'' (cf.~(\ref{eqTypePreserving}) and the subsequent discussion).

\begin{fact}\label{Fact_typePreserving}
Given $\cT$, $\hat\PHI:\cI'\ra\cL'$ is a uniformly random bijection subject to the 
condition that $\hat\PHI_{ij}\in L'_{\partial(\ell,j)}$	for all $\ell\in T^*$, $i\in M_\ell$, $j\in[k_\ell]$.
\end{fact}

\subsection{The random variable}
In this section we define the precise random variable to which we apply the second moment method and
summarise the result of the first/second moment calculations.
At this point, the obvious choice seems to be the number $\cZ'$ of valid $\theta$-shades of $\hat\PHI$.
However, there are two more technical issues that we need to tackle.

First, we saw that any valid shade $\xi$ of $\hat\PHI$ gives rise to a cover $\hat\xi$ of $\PHI'$.
But of course our overall goal is to exhibit a satisfying assignment of $\PHI'$, not merely a cover.
Hence, we call $\xi$ \emph{extendible} if 
$\PHI'$ has a satisfying assignment $\sigma$ such that
$\sigma(l)=\hat\xi(l)$ for all literals $l$ such that $\hat\xi(l)\in\cbc{0,1}$.
Thus, we can think of $\sigma$ as being obtained by substituting actual truth values for $l$ such that $\hat\xi(l)=*$.

Additionally, we introduce a condition to facilitate the second moment computation.
According to the physics picture, we expect that covers are ``well-separated''.
To hard-wire this geometry into our random variable, we call a valid shade $\xi$ of $\hat\PHI$ \emph{separable} if 
	there are no more than $\Erw[\cZ'|\cT]$ valid $\theta$-shades $\zeta$ of $\hat\PHI$ such that
		$$\frac1{2n}\abs{\cbc{l\in L':\hat\xi(l)\neq\hat\zeta(l)}}
			\not\in
			\brk{\frac12-2^{-0.49k},\frac12+2^{-0.49k}}.$$

\begin{definition}
A $\theta$-shade $\xi$ is \bemph{good} in $\hat\PHI$  if it is  valid, extendible and separable.
\end{definition}

\noindent
Let $\cZ$ denote the number of good $\theta$-shades of $\hat\PHI$.
In \Sec~\ref{Sec_TheFirstMoment} we will calculate the first moment of $\cZ$ to prove

\begin{proposition}\label{Prop_firstMoment}
There is  $\eps_k=\Theta_k(2^{-k/2})$ such that for $r=2^k\ln2-\frac{1+\ln2}{2}-\eps_k$
we have  $\Erw[\cZ|\cT]	=\exp(\Omega(n))$ \whp
\end{proposition}

\noindent
Furthermore, in \Sec~\ref{Sec_smmFull} we estimate the second moment to establish the following.

\begin{proposition}\label{Prop_secondMoment}
If $\eps_k=\tilde O_k(2^{-k/2})$ is such that
for $r=2^k\ln2-\frac{1+\ln2}{2}-\eps_k$ we have
 $\Erw[\cZ|\cT]=\exp(\Omega(n))$ \whp, then
	$\Erw[\cZ^2|\cT]\leq O(\Erw[\cZ|\cT]^2)$ \whp
\end{proposition}

\begin{proof}[Proof of \Thm~\ref{Thm_main} (assuming \Prop~\ref{Prop_firstMoment}--\ref{Prop_secondMoment})]
With $\eps_k$ and $r$ from \Prop~\ref{Prop_firstMoment} we obtain from
\Prop s~\ref{Prop_firstMoment} and~\ref{Prop_secondMoment} that 
	$\Erw[\cZ|\cT]\geq\exp(\Omega(n))$ and 
	$\Erw[\cZ^2|\cT]\leq O(\Erw[\cZ|\cT]^2)$ \whp\
Hence, the Paley-Zygmund inequality yields
	\begin{align}\label{eqThm_main_1}
	\liminf_{N\ra\infty}\Erw[\pr\brk{\cZ>0|\cT}]&>0.
	\end{align}
Since $\cZ$ counts good, and thus extendible shades, $\hat\PHI$  is satisfiable if $\cZ>0$.
Hence, (\ref{eqThm_main_1}) implies that
	\begin{align}\label{eqThm_main_2}
	\liminf_{N\ra\infty}\Erw[\pr[\mbox{$\hat\PHI$ is satisfiable}|\cT]]&>0.
	\end{align}
As $\cT\supset\cD$, (\ref{eqThm_main_2}) yields
	$\liminf_{N\ra\infty}\Erw\brk{\pr[\PHI'\mbox{ is satisfiable}|\cD]}>0$, i.e.,~\eqref{eqCondOnDegs} is established. Finally, \Thm~\ref{Thm_main} follows from \Prop~\ref{Prop_pruning}.
\end{proof}

\subsection{A few observations}\label{Sec_observations}
We conclude this section with a few basic observations that will be important in due course.

\begin{lemma}\label{Lemma_Lambda}
For any $\ell\in T^*$, $j\in[k_\ell]$
we have {$\ell_j^\rot=\ell_j^\purpur\prod_{j'\neq j}\ell_j^\y=2^{-k_\ell}+\tilde O_k(2^{-3k/2})$}.
\end{lemma}
\begin{proof}
The first equality sign is immediate from~(\ref{Lemma_Lambda1})--(\ref{Lemma_Lambda2}) and the fact that $\ell_j^\purpur=\ell_j^1+\ell_j^*$.
The second one follows from~(\ref{eqSP}), since the pruning step~\eqref{eqDegreesAfterPruning} guarantees that $\ell_{j'}^\purpur,\ell_{j'}^\y=1/2+\tilde O_k(2^{-k/2})$ for all $j'\in[k_\ell]$ and $k_\ell\in\{k-2,k-1,k\}$.
\end{proof}

\begin{lemma}
\Whp\ we have $\pi_t,\pi_\ell=\Omega(1)$ for all $t\in T,\ell\in T^*$.
\end{lemma}
\begin{proof}
Let $\cA$ be the set of all pairs $(d^+,d^-)$ of integers such that $|d^+-kr/2|,|d^--kr/2|\leq k^32^{k/2}$.
\Prop~\ref{Prop_pruning} shows that for any $(d^+,d^-)\in\cA$ the set $L'(d^+,d^-)$ of literals $l$ such that $d_l=d^+$, $d_{\neg l}=d^-$ 
has size $\Omega(n)$ \whp\
Furthermore, the construction in~\eqref{eqDefMyTheta} ensures that the type of a literal $l$ is determined by $d_l,d_{\neg l}$ and
the degrees of the literals that appear in the clauses that contain $l$.
Because $\PHI'$ is uniformly random given $\cD$ and $\cA$ is bounded, any possible constellation appears $\Omega(n)$ times \whp\
Hence, $\pi_t=\Omega(1)$ for all $t\in T$ \whp\
Similarly, the type of a clause $\PHI_i'$ is governed by the degrees of the literals that the clause contains
and the degrees of the literals that appear in a clause that contains a literal $l$ such that either $l$ or $\neg l$ appears in $\PHI_i'$.
Once more because $\PHI'$ is uniformly random given $\cD$, any possible constellation appears $\Omega(m)$ times \whp\
Hence, $\pi_\ell=\Omega(1)$ for all $\ell\in T^*$.
\end{proof}

\noindent
For a set $T_0\subset T$ define $\Vol(T_0)=\sum_{t\in T_0}\pi_t$.
Similarly, for $\cM\subset T^*$ let $\Vol(\cM)=\sum_{\ell\in\cM}\pi_\ell$.
The formula $\PHI'$ inherits certain discrepancy properties from the plain random formula $\PHI$.  
\begin{lemma}\label{Lemma_expansion}
\Whp\ $\PHI'$ enjoys the following properties. For $\ell \in T^*$ we write $\partial \ell = \{\partial(\ell,j): j \in [k_\ell]\}$.
\begin{description}
\item[DISC1] Assume that $A\subset T$ is such that $\Vol(A)\geq0.01$.
	Let $\cM$ be the set of all $\ell\in T^*$ such that  $|\partial \ell \cap A|\geq0.001k$.
	Then $\Vol(\cM)\geq1-\exp(-\Omega_k(k))$.
\item[DISC2] Assume that $A,B\subset T$ are disjoint sets of types such that $\Vol(A),\Vol(B)\geq0.47$.
	Let $\cM$ be the set of all $\ell\in T^*$ such that 
		$|\partial\ell\cap A|\geq0.4k$ and $|\partial\ell\cap B|\geq0.4k$.
	Then $\Vol(\cM)\geq1-k^{-9}$.
\item[DISC3] Assume that $A\subset T$ has satisfies $\Vol(A)\leq k^{-9}$.
	Let $\cM$ be the set of all $\ell\in T^*$ such that $|\partial\ell\cap A|\geq0.9k$.
	Then $\Vol(\cM)\leq\tilde O_k(2^{-k})\Vol(A)$.
\end{description}
\end{lemma}

\noindent
The proof of \Lem~\ref{Lemma_expansion}, which is very much based on standard arguments, can be found in Appendix~\ref{Sec_Lemma_expansion}. 
Finally, we define  $\brk T=\cbc{\cbc{t,\neg t}:t\in T}$.

\medskip
\noindent
{\bf\em In the rest of the paper we tacitly assume that $\pi_t=\Omega(1)$ and $\pi_\ell=\Omega(1)$ for all $t\in T$, $\ell\in T^*$,
	that statements (2) and (3) of \Prop~\ref{Prop_pruning} hold, and that $\PHI'$ satisfies {\bf DISC1}--{\bf DISC3} from \Lem~\ref{Lemma_expansion}.
	In addition, we assume that $r=M/N=2^k\ln2-(1+\ln 2)/2-\eps_k$ with $\eps_k=\tilde O_k(2^{-k/2})$, { and that $k$ is sufficiently large for various estimates to hold}.}

\section{The First Moment}\label{Sec_TheFirstMoment}

\subsection{An explicit formula}

The aim in this section is to prove \Prop~\ref{Prop_firstMoment}, i.e., to compute a lower bound for the expected number of good $\theta$-shades. To this end, we are first going to provide an exact, explicit formula for the first moment. Let $\cZ'$ denote the number of valid $\theta$-shades of $\PHI'$.
We sometimes use the notation $\pr_\cT\brk\nix=\pr\brk{\nix|\cT}(\PHI)$, $\Erw_\cT\brk\nix=\Erw\brk{\nix|\cT}(\PHI)$.

\begin{proposition}\label{Prop_firstMomentFormula}
There exist unique numbers 
	$q_{t,h}^{\rot},q_{\ell,j}^{\purpur}	\in(0,1)$ 
such that with $q_{\ell,j}^\y=1-q_{\ell,j}^\purpur$ the numbers $e_{t,h}^\rot$, $e_{\ell,j}^\purpur$ defined in Figure~\ref{Fig_firstMomentFormula} satisfy
	\begin{align*}
	e_{t,h}^\rot&=t_h^\rot~\mbox{  for all }t\in T,h\in[d_t],&e_{\ell,j}^\purpur&=\ell_j^\purpur~\mbox{  for all }\ell\in T^*,j\in[k_\ell].
	\end{align*}
Furthermore, with the expressions from Figure~\ref{Fig_firstMomentFormula},
	\begin{align}\nonumber
	\frac1{n}\ln\Erw_\cT[\cZ']&=-\frac{C\ln n}n+
	\sum_{t\in T}\pi_t\brk{H(t^0,t^1,t^*)+2\phioct}+{\frac{m}n}\sum_{\ell\in T^*}\pi_\ell\phivall+O(1/n),
	\end{align}
where	
	\begin{align}
	C&=\abs{\{\{t, \neg t\}: t \in T\}}+\sum_{\ell\in T^*} \frac{k_\ell}2+\sum_{t\in T}\sum_{h\in[d_t]}\frac{|\partial(t,h)|-1}{2},\label{eqC}\\
	\phioct&=t^1\ln s_t+t^*\ln(1-s_t)+\sum_{h\in[d_t]}\KL{t_h^\rot/t_h^\purpur}{q_{t,h}^\rot},\nonumber\\
	\phivall&=-\KL{\ell_1^\rot,\ldots,\ell_{k_\ell}^\rot,1-\ell_1^\rot-\cdots-\ell_{k_\ell}^\rot}{g_{\ell,1}^\rot,\ldots,g_{\ell,k_\ell}^\rot,g_{\ell}^\cyan}
		+\sum_{j\in[k_\ell]}\KL{\ell_j^{\purpur}}{q_{\ell,j}^{\purpur}}.\nonumber
	\end{align}
\end{proposition}

\begin{figure}[b]
\small
\begin{align*}
s_t&=1-\prod_{h\in[d_t]}(1-q_{t,h}^{\rot}),&
	g_\ell^{\cyan}&=1-\prod_{j\in[k_\ell]}q_{\ell,j}^{\y}-\sum_{j\in[k_\ell]}q_{\ell,j}^{\purpur}\prod_{j'\in[k_\ell] \setminus \{j\} }q_{\ell,j'}^{\y},&
	g_{\ell,j}^{\rot}&=q_{\ell,j}^{\purpur}\prod_{j'\in[k_\ell] \setminus \{j\}}q_{\ell,j'}^{\y},\\
e_{t,h}^{\rot}&=\frac{t^1q_{t,h}^\rot}{s_t},&
	 e_{\ell,j}^{\purpur}&  =\ell_j^\rot+\frac{q_{\ell,j}^{\purpur}}{g_\ell^{\cyan}}\bc{1-\sum_{j'\in[k_\ell]}\ell_{j'}^\rot}\bc{1-\prod_{j'\in[k_\ell] \setminus \{j\}}q_{\ell,j'}^\y}.
\end{align*}
\caption{The formulas for \Prop~\ref{Prop_firstMomentFormula}.}\label{Fig_firstMomentFormula}
\end{figure}

To prove \Prop~\ref{Prop_firstMomentFormula}, we  express the property of being a valid $\theta$-shade as a combination of events that are easy to describe in terms
of independent random variables.
The basic idea is to separate the property of being valid, which concerns how the colors are distributed
amongst the clauses, from the property of being a $\theta$-shade, which deals with how the clones of the individual literals are colored.
Due to condition {\bf V2} from \Def~\ref{Def_valid}, this last point introduces a smidgen of an occupancy problem into our analysis.
More specifically, we prove \Prop~\ref{Prop_firstMomentFormula} in the following three subsections,
	dealing first with the entropy, then with the validity probability (corresponding essentially to the $\phivall$ terms) and finally with the occupancy aspect (corresponding to the $\phioct$ terms).

\subsubsection{The entropy}
We saw that any valid $\theta$-shade $\xi$ of $\hat\PHI$ induces a cover $\hat\xi$ of $\PHI'$.
In fact, \Def~\ref{Def_thetaShade} pins down the fraction of literals of each type that are set to $0,1,*$ under $\hat\xi$.
Particularly, $|\{l\in L_t':\hat\xi(l)=z\}|\doteq n_t t^z$ for all $z\in\spins$.
Furthermore, the map $\hat\xi$ clearly has the property that $\hat\xi(\neg l)=\neg\hat\xi(l)$ for all $l\in L'$.
We begin by counting maps with these two properties.

\begin{lemma}\label{Lemma_theEntropy}
\Whp\ the total number of maps $\zeta:L'\ra\cbc{0,1,*}$ such that
	\begin{equation}\label{eqTheEntropy1}
	\abs{\cbc{l\in L_t':\zeta(l)=z}}\doteq n_t t^z\qquad\mbox{for all }z\in\spins, t\in T
	\end{equation}
and such that $\zeta(\neg l)=\neg\zeta(l)$ for all $l\in L'$ is
	$\Theta\bc{n^{-|\brk T|}}\exp\brk{n\sum_{t\in T}\pi_t H(t^0,t^1,t^*)}$, where $\brk T=\cbc{\cbc{t,\neg t}:t\in T}$.
\end{lemma}
\begin{proof}
We introduce an equivalence relation on $T$ by letting $t\equiv t'$ if $t=\neg t'$.
Then $\brk T $ is the set of equivalence classes.
Let $t_1,\ldots,t_\nu\in T$ be a sequence that contains precisely one representative from each equivalence class.
Due to the condition $\zeta(\neg l)=\neg\zeta(l)$, we just need to count maps $\zeta_i:L_{t_i}'\ra\spins$ 
such that $|\{l\in L_{t_i}':\zeta(l)=z\}|\doteq n_{t_i}t_i^z$ for all $z\in\spins$.
There are two cases.
\begin{description}
\item[Case 1: $t_i\neq\neg t_i$]
	by Fact~\ref{Fact_entropyFunction}, the total number of ways of setting $n_{t_i}t_i^z+O(1)$ literals $l\in L_{t_i}'$ to $z$ for each $z\in\spins$ is
		\begin{equation}\label{eqLemma_theEntropy1}
		O(1)\cdot\bink{n_{t_i}}{n_{t_i}t_i^0,n_{t_i}t_i^1,n_{t_i}t_i^*}=\Theta(n_{t_i}^{-1})\exp\brk{n_{t_i}H(t_i^0,t_i^1,t_i^*)}
			=\Theta(n^{-1})\exp\brk{2n\pi_{t_i}H(t_i^0,t_i^1,t_i^*)}.
		\end{equation}
\item[Case 2: $t_i=\neg t_i$]
	we merely get to pick the values $\zeta(l)$ for variables $x_i\in L_{t_i}'$ (as $\zeta(\neg x_i)$ is implied).
	Therefore,  the number of possible maps comes to
		\begin{equation}\label{eqLemma_theEntropy2}
		O(1)\cdot\bink{n_{t_i}/2}{n_{t_i}t_i^0/2,n_{t_i}t_i^1/2,n_{t_i}t_i^*/2}
			=\Theta(n^{-1})\exp\brk{n\pi_{t_i}H(t_i^0,t_i^1,t_i^*)}.
		\end{equation}		
\end{description}
Multiplying~(\ref{eqLemma_theEntropy1}) and~(\ref{eqLemma_theEntropy2}) up for $i=1,\ldots,\nu$ completes the proof.
\end{proof}

\subsubsection{The validity probability}\label{Sec_firstMoment_validity_probability}
Fix a map $\zeta:L'\ra\spins$ 
that satisfies~(\ref{eqTheEntropy1}) such that $\zeta(\neg l)=\neg\zeta(l)$ for all $l\in L'$.
If $\hat\PHI$ has a valid $\theta$-shade $\xi$ such that $\zeta=\hat\xi$, then the following two events occur for every clause type $\ell\in T^*$.
First, to satisfy condition (2) in \Def~\ref{Def_thetaShade},  for each $\ell\in T^*$ the event 
	$$\cB_\ell(\zeta)=\cbc{\forall j\in[k_\ell]:\abs{\cbc{i\in M_\ell:\zeta(\PHI_{ij}')=0}}\doteq\ell_j^\y m_\ell}$$
must occur.
Let $\cB(\zeta)=\bigcap_{\ell \in T^*}\cB_\ell(\zeta)$.
To define the second event, let
	\begin{align}\label{eqGelljrot}
	\Gamma_{\ell,j}^\rot(\zeta)&=\abs{\cbc{i\in M_\ell:\zeta(\PHI_{ij}')\in\{*,1\}\mbox{ and }\zeta(\PHI_{ij'}')=0\mbox{ for all }j' \in [k_\ell] \setminus \{j\}}},& j\in[k_\ell],\\
	\Gamma_{\ell}^\cyan(\zeta)&=\abs{\cbc{i\in M_\ell:\exists 1\leq j_1<j_2\leq k_\ell:\zeta(\PHI_{ij_1}'),\zeta(\PHI_{ij_2}')\in\{*,1\}}}.
		\nonumber
	\end{align}
In words, $\Gamma_{\ell,j}^\rot$ is the number of clauses of type $\ell$ such that the $j$th literal takes value either $1$ or $*$, while all other literals are set to false.
Moreover, $\Gamma_{\ell}^\cyan$ is the number of clauses of type $\ell$ that contain at least two literals assigned $1$ or $*$. Set
	$$\cS_\ell(\zeta)=\bigg\{\forall j\in[k_\ell]:\Gamma_{\ell,j}^\rot(\zeta)\doteq\ell_j^\rot m_\ell\mbox{ and }
		\Gamma_{\ell}^\cyan(\zeta)=m_\ell-\sum_{j\in[k_\ell]}\Gamma_{\ell,j}^\rot(\zeta)\bigg\}$$
and $\cS(\zeta)=\bigcap_{\ell \in T^*}\cS_\ell(\zeta)$.
If $\zeta=\hat\xi$ for a  valid $\theta$-shade $\xi$, then $\cB(\zeta)\cap\cS(\zeta)$ occurs
	(however, the converse is not  true).

\begin{lemma}\label{Lemma_firstMomentSell}
Let $\ell\in T^*$.
For each $j\in[k_\ell]$ there exist $q_{\ell,j}^\purpur,q_{\ell,j}^\y\in(0,1)$ such that
	$q_{\ell,j}^\purpur+q_{\ell,j}^\y=1$ and such that with $e_{\ell,j}^{\purpur}$ from Figure~\ref{Fig_firstMomentFormula} we have
	$e_{\ell,j}^\purpur=\ell_j^\purpur$.
With these $q_{\ell,j}^\purpur,q_{\ell,j}^\y$ we have, again with the notation from Figure~\ref{Fig_firstMomentFormula},
	\begin{align*}
	\frac1{m_\ell}{\ln\pr_{\cT}\brk{\cS_\ell(\zeta)|\cB(\zeta)}}&=
		-\KL{\ell_1^\rot,\ldots,\ell_{k_\ell}^\rot,1-\ell_1^\rot-\cdots-\ell_{k_\ell}^\rot}{g_{\ell,1}^\rot,\ldots,g_{\ell,k_\ell}^\rot,g_{\ell}^\cyan}\\
			&\qquad\qquad\qquad+\sum_{j=1}^{k_\ell}\KL{\ell_j^{\purpur}}{q_{\ell,j}^{\purpur}}
				-\frac{k_\ell\ln n}{2m_\ell}+O(1/n),\\
	\frac1n\ln\pr_{\cT}\brk{\cB(\zeta)}&=O(1/n)-\sum_{t\in T}\sum_{h\in[d_t]}\frac{(|\partial(t,h)|-1)\ln n}{2n}.
	\end{align*}
\end{lemma}

\noindent
In the rest of this section we prove \Lem~\ref{Lemma_firstMomentSell}.
We begin with calculating the probability of the event $\cB(\zeta)$.

\begin{claim}\label{Claim_firstMomentValPoly}
We have $\frac1{m_\ell}\ln\pr_\cT\brk{\cB(\zeta)}=O(1/n)-\sum_{t\in T}\sum_{h\in[d_t]}\frac{(|\partial(t,h)|-1)\ln n}{2m_\ell}$.
\end{claim}
\begin{proof}
Due to the requirement that $\zeta$ satisfies~(\ref{eqTheEntropy1}), 
we can write down an explicit formula for $\pr_\cT\brk{\bigcap_\ell\cB_\ell(\zeta)}$.
Namely,
	\begin{align}\label{eqClaim_firstMomentValPoly1}
	\pr_\cT\brk{\cB(\zeta)}&=\Theta\bc 1\prod_{t\in T}\prod_{h\in[d_t]}\frac{\prod_{(\ell,j)\in\partial(t,h)}\bink{m_\ell}{\ell_j^\y m_\ell}}{\bink{n_t}{n_t t^0}}.
	\end{align}
Note that the construction of the clause types ensures that $\ell_j^\y=t^0$ if $(\ell,j)\in\partial(t,h)$ for some $h\in[d_t]$. Fact~\ref{Fact_entropyFunction} and the fact $n_t = \sum_{(\ell, j) \in \partial(t,h)} m_\ell$ show that for any $t,h$,
	\begin{align*}
	\bink{n_t}{n_t t^0}^{-1}\prod_{(\ell,j)\in\partial(t,h)}\bink{m_\ell}{\ell_j^\y m_\ell}=\Theta(n^{(1-|\partial(t,h)|)/2}).
	\end{align*}
The assertion follows then form~(\ref{eqClaim_firstMomentValPoly1}).
\end{proof}

To derive the desired formula for $\pr_\cT\brk{\cS_\ell(\zeta)|\cB(\zeta)}$, we fix a clause type $\ell\in T^*$.
We need to  establish the existence of the parameters $q_{\ell,j}^{\purpur},q_{\ell,j}^{\y}$.

\begin{claim}\label{Claim_firstMomentSell3}
There is a unique vector $q_\ell=(q_{\ell,j}^\purpur,q_{\ell,j}^\y)_{j\in[k_\ell]}$ such that
	$q_{\ell,j}^\purpur+q_{\ell,j}^\y=1$,
	$q_{\ell,j}^\purpur=\ell_j^\purpur-2^{-k_\ell-1}+\tilde O_k(2^{-3k/2})$
	and $e_{\ell,j}^\purpur=\ell_{j}^\purpur$ for all $j\in[k_\ell]$.
\end{claim}
\begin{proof}
Consider the map (see also Figure~\ref{Fig_firstMomentFormula})
	$$e_\ell^\purpur:(0,1)^{k_\ell}\ra(0,1)^{k_\ell},~~~  (q_{\ell,j}^\purpur)_{j\in[k_\ell]}\mapsto 
		(e_{\ell,j}^\purpur)_{j\in[k_\ell]}=\bc{\ell_j^\rot+
			\frac{q_{\ell,j}^\purpur}{g_\ell^{\cyan}}\Big(1-\sum_{j' \in [k_\ell]}\ell_{j'}^\rot\Big)\Big(1-\prod_{j' \in [k_\ell]\setminus \{j\}}q_{\ell,j'}^\y\Big)
			}_{j\in[k_\ell]}.$$
If $|q_{\ell,j}^\purpur - 1/2| \le \tilde O_k(2^{-k/2})$ for all $j\in[k_\ell]$, then we verify that
	$e_{\ell,j}^\purpur=1/2+\tilde O_k(2^{-k/2})$ and 
	\begin{align*}
	\frac{\partial e_{\ell,j}^\purpur}{\partial q_{\ell,j}^\purpur}&=1+\tilde O_k(2^{-k/2}),&
		\frac{\partial e_{\ell,j}^\purpur}{\partial q_{\ell,j'}^\purpur}&=\tilde O_k(2^{-k/2})\quad\mbox{for all }j,j'\in[k_\ell]\setminus\{j\}.
	\end{align*}
Thus, the Jacobian is (strictly) diagonally dominant and invertible, and the assertion follows readily from the inverse function theorem.
\end{proof}

To calculate $\pr\brk{\cS_\ell(\zeta)|\cB(\zeta)}$ 
we introduce a new probability space in which the colors of the individual literal clones correspond to independent random variables.
Let $\vec\chi_\ell=(\vec\chi_{\ell,j}(i))_{i\in[m_\ell],j\in[k_\ell]}$ be a  random vector whose entries are independent random variables with values in $\cbc{\purpur,\y}$
such that
	$\pr\brk{\vec\chi_{\ell,j}(i)=\purpur}=q_{\ell,j}^{\purpur}\mbox{ for each  }i\in[m_\ell], j\in[k_\ell].$
We further introduce the random variables
	\begin{align}
	b_{\ell,j}^z&=\abs{\cbc{i\in[m_\ell]:\vec\chi_{\ell,j}=z}}, \qquad z\in\{\purpur,\y\},\nonumber\\
	G_{\ell,j}^\rot&=\big\{i\in[m_\ell]:\vec\chi_{\ell,j}=\purpur\mbox{ and }\vec\chi_{\ell,j'}=\y\mbox{ for all }j'\in[k_\ell]\setminus\cbc j\big\},
		\nonumber\\
	G_{\ell}^\cyan&=\cbc{i\in[m_\ell]:\exists 1\leq j<j'\leq k_\ell:\vec\chi_{\ell,j}=\vec\chi_{\ell,j'}=\purpur},\nonumber\\
	G_{\ell}^\y&=\cbc{i\in[m_\ell]: \forall j\in [k_\ell]:\vec\chi_{\ell,j}=\y}.\nonumber
	\end{align}
Define the events
 	\begin{align*}
	B_\ell&=\{\forall j\in[k_\ell]:b_{\ell,j}^\y\doteq\ell_j^\y m_\ell\},&
	S_\ell&=\bigg\{\forall j\in[k_\ell]:|G_{\ell,j}^\rot|\doteq\ell_j^\rot m_\ell\mbox{ and }|G_\ell^\cyan|=m_\ell-\sum_{j=1}^{k_\ell}|G_{\ell,j}^\rot|\bigg\}.
	\end{align*}
This construction ensures that 
	\begin{equation}\label{Fact_firstMomentSell}
	\frac1{m_\ell}\ln\pr_\cT\brk{\cS_\ell(\zeta)|\cB(\zeta)}=\frac1{m_\ell}\ln\pr\brk{S_\ell|B_\ell}+O(1/n).
	\end{equation}
Crucially, since the entries of $\vec\chi_\ell$ are independent,  $\pr\brk{S_\ell}$, $\pr\brk{B_\ell}$ are easy to calculate.

\begin{claim}\label{Claim_firstMomentSell1}
With $g_\ell^\cyan$, $g_{\ell,j}^\rot$ as in Figure~\ref{Fig_firstMomentFormula}  
	$$\frac1{m_\ell}\ln\pr\brk{S_\ell}=-\KL{\ell_1^\rot,\ldots,\ell_{k_\ell}^\rot,1-\ell_1^\rot-\cdots-\ell_{k_\ell}^\rot}{g_{\ell,1}^\rot,\ldots,g_{\ell,k_\ell}^\rot,g_{\ell}^\cyan}
		-{\frac{k_\ell\ln n}{2m_\ell}}+O(1/n).$$
\end{claim}
\begin{proof}
Because the entries $\vec\chi_{\ell,j}(i)$ are mutually independent, the
random vector $(|G_{\ell,1}^\rot|,\ldots,|G_{\ell,k_\ell}^\rot|,|G_{\ell}^\cyan|,|G_{\ell}^\y|)$
is multinomially distributed with 
	\begin{align*}
	\Erw[|G_{\ell,j}^\rot|]&=m_\ell g_{\ell,j}^\rot,&
	\Erw[|G_{\ell}^\cyan|]&=m_\ell g_{\ell}^\cyan,&
	\Erw[|G_{\ell}^\y|]&\textstyle= m_\ell \, \prod_{j\in[k_\ell]} q^\y_{\ell, j} =m_\ell\big(1-g_{\ell}^\cyan-\sum_{j \in [k_\ell]}g_{\ell,j}^\rot\big).
	\end{align*}
Hence, the assertion follows from Fact~\ref{Fact_binomialLargeDev}.
\end{proof}

\begin{claim}\label{Claim_firstMomentSell2}
We have
$\frac1{m_\ell}\ln\pr[B_\ell]=-\frac{k_\ell\ln n}{2{ m_\ell}}+\sum_{j \in [k_\ell]}\KL{\ell_j^{\purpur}}{q_{\ell,j}^{\purpur}}+O(1/n).$
\end{claim}
\begin{proof}
Once more due to the independence of the $\vec\chi_{\ell,j}(i)$,
the vector $(b_{\ell,j}^\y)_{j\in[k_\ell]}$ consists of independent binomial variables with means $\Erw[b_{\ell,j}]=q_{\ell,j}^\y m_\ell$.
Since $\ell_j^\purpur+\ell_j^\y=q_{\ell,j}^{\purpur}+q_{\ell,j}^{\y}=1$, the claim follows from Fact~\ref{Fact_binomialLargeDev}.
\end{proof}
To calculate the conditional probability $\pr\brk{S_\ell|B_\ell}$, we use Bayes' formula, according to which
	\begin{equation}\label{eqValBayes}
	\pr\brk{S_\ell|B_\ell}=\frac{\pr\brk{B_\ell|S_\ell}}{\pr\brk{B_\ell}}\cdot\pr\brk{S_\ell}.
	\end{equation}
We first compute  $\pr\brk{B_\ell|S_\ell}$.

\begin{claim}\label{Claim_firstMomentSell4}
We have
	$\frac1{m_\ell}\ln\pr\brk{B_\ell|S_\ell}=-\frac{k_\ell\ln n}{2{m_\ell}}+O(1/n).$
\end{claim}
\begin{proof}
Let $G_\ell^\rot=G_{\ell,1}^\rot\cup\cdots\cup G_{\ell,k_\ell}^\rot$.
Given that $S_\ell$ occurs and given the set $G_\ell^\rot$, the vectors $\vec\chi_\ell(i)=(\vec\chi_{\ell,j}(i))_{j\in[k_\ell]}$
with $i\in[m_\ell]\setminus G_\ell^\rot$ are mutually independent.
Thus, 
	$b_{\ell,j}'=\sum_{i\in[m_\ell]\setminus G_\ell^\rot}\vecone_{\chi_{\ell,j}(i)=\purpur}$
is a sum of independent random variables for each $j\in[k_\ell]$.
Hence, the vector $(b_{\ell,j}')_{j\in[k_\ell]}$ satisfies the assumptions of \Thm~\ref{Lemma_LLT}.
Furthermore, since $e_{\ell,j}^\purpur=\ell_j^\purpur$ by the choice of the parameters $q_{\ell,j}^\purpur,q_{\ell,j}^\y$, we have
	\begin{equation}\label{eqClaim_firstMomentSell4}
	\Erw[b_{\ell,j}'|S_\ell]\doteq 
		\frac{q_{\ell,j}^{\purpur}}{g_\ell^{\cyan}}\bc{1-\prod_{j' \in [k_\ell]\setminus \{j\}}q_{\ell,j'}^\y}\bc{1-\sum_{j'\in[k_\ell]}\ell_{j'}^\rot}m_\ell
			=(e_{\ell,j}-\ell_j^\rot)m_\ell.
	\end{equation}
Since given $S_\ell$ we have $b_{\ell,j}^\purpur\doteq b_{\ell,j}'+\ell_{j}^\rot m_\ell$
and because $b_{\ell,j}^\y=m_\ell-b_{\ell,j}^\purpur$, (\ref{eqClaim_firstMomentSell4}) and \Thm~\ref{Lemma_LLT}
imply that $\pr\brk{B_\ell|S_\ell}=\Theta(n^{-k_\ell/2})$, as desired.
\end{proof}

Finally, \Lem~\ref{Lemma_firstMomentSell} follows from~(\ref{Fact_firstMomentSell}), (\ref{eqValBayes}) and Claims~\ref{Claim_firstMomentSell1}--\ref{Claim_firstMomentSell4}. We conclude this section with the following statement that will prove useful later.
\begin{corollary}\label{Claim_muell}
For $\ell\in T^*$ and $j\in[k_\ell]$ let $\mu_{\ell,h}$ be the number of clauses of type $\ell$ that contain precisely $k_\ell-j$ yellow clones.
Let $I=[\bink kj 2^{-1-k},\bink kj 2^{1-k}]$.
Then 
$\pr_\cT[\mu_{\ell,h}/m_\ell\not\in I|\cS(\zeta),\cB(\zeta)]\leq\exp(-\Omega(n))$.
\end{corollary}
\begin{proof}
Let $\nu_{\ell,j}$ be the number of indices $i\in[m_\ell]$ such that
	$\abs{\cbc{h\in[k_\ell]:\vec\chi_{\ell,h}(i)=\y}}=j$.
Then~(\ref{Fact_firstMomentSell}) implies that
	\begin{align}\label{eqClaim_muell_1}
	\pr_\cT\brk{\mu_{\ell,h}/m_\ell\not\in I|\cS(\zeta),\cB(\zeta)}\leq O(\pr\brk{\nu_{\ell,h}/m_\ell\not\in I|S_\ell,B_\ell}).
	\end{align}
Furthermore, Claim~\ref{Claim_firstMomentSell4} entails that
	\begin{align}\label{eqClaim_muell_2}
	\pr\brk{\nu_{\ell,h}/m_\ell\not\in I|S_\ell,B_\ell}=\exp(o(n))\pr\brk{\nu_{\ell,h}/m_\ell\not\in I|S_\ell}.
	\end{align}
In addition, since $\ell_h^\y=\frac12+\tilde O_k(2^{-k/2})$
and thus $q_{\ell,h}^\y=\frac12+\tilde O_k(2^{-k/2})$ for all $h\in[k_\ell]$ by Claim~\ref{Claim_firstMomentSell3}, we see that
	\begin{align}\label{eqClaim_muell_3}
	\Erw[\nu_{\ell,j}|S_\ell]&=(1+o_k(1))m_\ell\bink{k}h2^{-k}.
	\end{align}
Further, given $S_\ell$, $\nu_{\ell,h}$ is a sum of $m_\ell$ independent random variables.
Therefore, the Chernoff bound and~(\ref{eqClaim_muell_3}) imply that $\pr\brk{\nu_{\ell,h}/m_\ell\not\in I|S_\ell}\leq\exp(-\Omega(n))$.
Hence, the assertion follows from~(\ref{eqClaim_muell_1}) and~(\ref{eqClaim_muell_2}).
\end{proof}

\subsubsection{The occupancy probability}\label{ssec:first_moment_occ}
Assume that $\zeta:L'\ra\spins$ is a map 
such that $\zeta(\neg l)=\neg\zeta(l)$ for all $l\in L'$ and such that~(\ref{eqTheEntropy1}) holds and such that the events $\cB(\zeta)$, $\cS(\zeta)$ occur.
We saw that these are necessary conditions for the existence of a valid $\theta$-shade $\xi$ such that $\zeta=\hat\xi$.
But there is a further important necessary condition.
Namely, with $\Gamma_{\ell,j}^\rot(\zeta)$ the sets from~(\ref{eqGelljrot}), we define
	$$\Gamma_{t,h}^\rot(\zeta)=\bigcup_{(\ell,j)\in\partial(t,h)}\cbc{\PHI_{ij}':i\in \Gamma_{\ell,j}^\rot(\zeta)}\qquad\mbox{for each $t\in T$, $h\in[d_t]$}.$$
In words, $\Gamma_{t,h}^\rot(\zeta)$ is the set of all literals of type $t$ that are assigned either $*$ or $1$ and whose $h$th clone appears in a clause where all other literals are set to $0$.
Then {\bf SD1--SD2} from \Def~\ref{Def_shade} require that the following two conditions hold for any $t\in T$:
\begin{description}
\item[RED1] If $l\in L_t'$ is such that $\zeta(l)=1$, then there is $h\in[d_t]$ such that $l\in\Gamma_{t,h}^\rot(\zeta)$.
\item[RED2] If $l\in L_t'$ is such that $\zeta(l)=*$, then  for all $h\in[d_t]$ we have $l\not\in\Gamma_{t,h}^\rot(\zeta)$.
\end{description}
Let $\cR_t(\zeta)$ be the event that {\bf RED1--RED2} hold for $t\in T$ and let $\cR(\zeta)=\bigcap_{t\in T}\cR_t(\zeta)$. We will prove the following statement in this subsection.

\begin{lemma}\label{Lemma_occFirstMoment}
For $t\in T$ there exists a unique vector $q_t^\rot=(q_{t,h}^\rot)_{h\in[d_t]}$
with entries $q_{t,h}^\rot=t_h^\rot/t^1+\tilde O_k(4^{-k})$ 
such that with the notation of Figure~\ref{Fig_firstMomentFormula} we have $e_{t,h}^\rot=t_h^\rot$ for all $h\in[d_t]$.
In terms of these vectors $q_t^\rot$ we have
	$$\frac1n\ln\pr_\cT\brk{\cR(\zeta)|\cS(\zeta),\cB(\zeta)}=\sum_{t\in T}2\pi_t\brk{t^1\ln s_t+t^*\ln(1-s_t)+
		(t^1+t^*)\sum_{h\in[d_t]}\KL{t_h^\rot}{q_{t,h}^\rot}}+O(1/n).$$
\end{lemma}

As in the previous section, we are going to introduce a new probability space in which the individual clones of the literals of any particular type correspond to independent events.
Let $t\in T$ and set $n_t^z=|\{l\in L_t':\zeta(l)=z\}$ for $z\in\spins$.
Then $n_t^z\doteq t^zn_t$ due to~(\ref{eqTheEntropy1}).
Let $n_t^\purpur=n_t^1+n_t^*$.

\begin{claim}\label{Claim_occFirstMoment_3}
There is a unique vector $q_t^\rot=(q_{t,h}^\rot)_{h\in[d_t]}$ with entries $q_{t,h}^\rot={t_h^\rot}/{t^1}+\tilde O_k(4^{-k})$ such that
	$e_{t,h}^\rot=t_h^\rot$ for all $h\in[d_t]$.
\end{claim}
\begin{proof}
We consider the map (see also Figure~\ref{Fig_firstMomentFormula})
	$$e_t^\rot:(0,1)^{d_t}\ra(0,1)^{d_t},\qquad(q_{t,h}^\rot)_{h\in[d_t]}\mapsto(e_{t,h}^\rot)_{h\in[d_t]}=(t^1q_{t,h}^\rot/s_t)_{h\in[d_t]},$$
	where $s_t = 1 - \prod_{h \in [d_t]}(1 - q_{t,h}^\rot)$. If $|q_{t,h}^\rot - t_h^\rot/t^1|\in{o_k(1)}$ for all $h\in[d_t]$, then $e_{t,h}^\rot={t_h^\rot}+o_k(2^{-k})$ and
	\begin{align*}
	\frac{\partial e_{t,h}^\rot}{\partial q_{t,h}^\rot}&=t^1+\tilde O_k(2^{-k}),&
		\frac{\partial e_{t,h}^\rot}{\partial q_{t,h'}^\rot}&=\tilde O_k(4^{-k})\qquad\mbox{for all }h,h'\in[d_t],h\neq h'.
	\end{align*}
As in the proof of Claim~\ref{Claim_firstMomentSell3}, the assertion follows from the inverse function theorem.
\end{proof}

Equipped with the vector $q_t^\rot=(q_{t,h}^\rot)_{h\in[d_t]}$ from Claim~\ref{Claim_occFirstMoment_3},
	we let $\vec\rho=(\vec\rho_{t,h}(i))_{h\in[d_t],i\in[n_t^\purpur]}$ be a vector whose entries 
	are independent random variables with values in $\cbc{\rot,\cyan}$ such that
	$\pr\brk{\vec\rho_{t,h}(i)=\rot}=q_{t,h}^\rot$ for all $i\in[n_t^\purpur],h\in[d_t]$.
We are going to consider the random variables
	\begin{align*}
	b_{t,h}^\rot&=\abs{\cbc{i\in[n_t^\purpur]:\vec\rho_{t,h}(i)=\rot}},\qquad h\in[d_t].
	\end{align*}
Let $B_{t,h}$ be the event that $b_{t,h}^\rot\doteq t_h^\rot n_t$ and let $B_t=\bigcap_{h\in [d_t]}B_{t,h}$.
Moreover, let $R_t^1$ be the event that for each $i\in[n_t^1]$ there exists $h\in[d_t]$ such that $\vec\rho_{t,h}(i)=\rot$.
Further, let $R_t^*$ be the event that for any $i\in[n_t^\purpur]\setminus [n_t^1]$ and any $h\in[d_t]$ we have $\vec\rho_{t,h}(i)=\cyan$,
and let $R_t=R_t^1\cap R_t^*$.
This construction ensures that
	\begin{equation}\label{Fact_occFirstMoment}
	\frac1n\ln\pr\brk{\cR(\zeta)|\cB(\zeta),\cS(\zeta)}=O(1)+\frac1n\sum_{t\in T}\ln\pr\brk{R_t|B_t}.
	\end{equation}

\noindent
To calculate the r.h.s.\ we are going to compute $\pr\brk{R_t}$, $\pr\brk{B_t}$ and $\pr\brk{B_t|R_t}$.

\begin{claim}\label{Claim_occFirstMoment_1}
We have
	$\ln\pr\brk{R_t}=n_t^\purpur(t^1\ln s_t+t^*\ln(1-s_t))+O(1).$
\end{claim}
\begin{proof}
Due to the independence of the entries $\vec\rho_{t,h}(i)$, $s_t$ is simply the probability that for a given $i\in[n_t^\purpur]$ there is $h\in[d_t]$ such that $\vec\rho_{t,h}(i)=\rot$.
Thus, the assertion follows from the fact that $n_t^1\doteq t^1n_t$ and $n_t^*\doteq t^*n_t$.
\end{proof}

\begin{claim}\label{Claim_occFirstMoment_2}
We have $\ln\pr\brk{B_t}=\frac{d_t\ln n}2-n_t^\purpur\sum_{h\in[d_t]}\KL{t_h^\rot}{q_{t,h}^\rot}+O(1)$.
\end{claim}
\begin{proof}
Because the entries $\vec\rho_{t,h}(i)$ are independent, the random variables $b_{t,h}^\rot$ are independent and binomially distributed with mean $n_t^\purpur q_{t,h}^\rot + O(1)$.
Hence, the assertion follows  from Fact~\ref{Fact_binomialLargeDev}.
\end{proof}

\begin{claim}\label{Claim_occFirstMoment_4}
We have $\ln\pr\brk{B_t|R_t}=\frac{d_t\ln n}2+O(1)$.
\end{claim}
\begin{proof}
Given that $R_t$ occurs, each $b_{t,h}^\rot$ is a sum of independent random variables, namely
	$b_{t,h}^\rot=\sum_{i\in[n_t^1]}\vecone_{\vec\rho_{t,h}(i)=\rot}.$
Furthermore, as $e_{t,h}^\rot=t_h^\rot$ we see that $\Erw[b_{t,h}^\rot|R_t]\doteq n_tt_h^\rot$.
Hence, the assertion follows from \Thm~\ref{Lemma_LLT}.
\end{proof}

\noindent
Finally, \Lem~\ref{Lemma_occFirstMoment} follows from~(\ref{Fact_occFirstMoment}) and Claims~\ref{Claim_occFirstMoment_1}--\ref{Claim_occFirstMoment_4}.

\begin{proof}[Proof of \Prop~\ref{Prop_firstMomentFormula}]
Let $\zeta:L'\ra\spins$ be a map as in \Lem~\ref{Lemma_theEntropy}.
Then $\hat\PHI$ has a valid $\theta$-shade $\xi$ such that $\hat\xi=\zeta$ iff the events $\cB(\zeta)$, $\cS(\zeta)$ and $\cR(\zeta)$ occur.
Therefore, \Prop~\ref{Prop_firstMomentFormula} follows from \Lem s~\ref{Lemma_theEntropy}, \ref{Lemma_firstMomentSell} and~\ref{Lemma_occFirstMoment}.
\end{proof}

\subsection{The asymptotic expansion}
To prove \Prop~\ref{Prop_firstMoment} we derive the following asymptotic expansion of the formula from \Prop~\ref{Prop_firstMomentFormula}.

\begin{corollary}\label{Cor_firstMomentFormula}
\Whp\ we have $\frac1{n}\ln\Erw[\cZ'|\cT]=\eps_k2^{-k}+\tilde O_k(2^{-3k/2})$.
\end{corollary}

\noindent
To prove \Cor~\ref{Cor_firstMomentFormula} we derive asymptotic formulas
for the entropy, the validity probability and the occupancy probability separately.

\begin{claim}\label{Lemma_firstMomentEntropy}
\Whp\ we have $\sum_{t\in T}\pi_t H(t^0,t^1,t^*)=\ln2+2^{-k-1}+\tilde O_k(2^{-3k/2})$.
\end{claim}
\begin{proof}
Let $\delta_t=d_t-d_{\neg t}$ for any $t \in T$. Since $t^* = 2^{-k-1}$ and $t^1 = 1/2 + \delta_t 2^{-k-1} - 2^{-k-2}$ we infer with \Prop~\ref{Prop_pruning} that \whp
	\begin{align*}
	\sum_{t\in T}\pi_t H(t) &=-\sum_{t\in T}\pi_t(t^0\ln t^0+t^1\ln t^1+t^*\ln t^*)
		=\frac{(k+1)\ln2}{2^{k+1}}-\sum_{t\in T}\pi_t(t^0\ln t^0+t^1\ln t^1) + \tilde O_k(2^{-3k/2})\\
		&=\frac{(k+1)\ln2}{2^{k+1}}-\sum_{t\in T}\pi_t\brk{t^0\ln\bc{\frac12-\frac{\delta_t}{2^{k+1}}-\frac1{2^{k+2}}}+t^1\ln
			\bc{\frac12+\frac{\delta_t}{2^{k+1}}-\frac1{2^{k+2}}}} + \tilde O_k(2^{-3k/2})\\
		&=\bc{1+\frac{k }{2^{k+1}}}\ln2
			-\sum_{t\in T}\pi_t\brk{t^0\ln\bc{1-\frac{1+2\delta_t}{2^{k+1}}}+t^1\ln\bc{1-\frac{1-2\delta_t}{2^{k+1}}}} + \tilde O_k(2^{-3k/2}).
	\end{align*}
Let $x_t=2^{-k-1}(1+2\delta_t)$ and $y_t=2^{-k-1}(1-2\delta_t)$.
Using the expansion $\ln(1+x)=x-x^2/2+O(x^3)$ as $x\ra0$, we obtain
	\begin{align*}
	-t^0\ln\bc{1-\frac{1+2\delta_t}{2^{k+1}}}&-t^1\ln\bc{1-\frac{1-2\delta_t}{2^{k+1}}}
		=-\frac12\brk{(1-x_t)\ln(1-x_t)+(1-y_t)\ln(1-y_t)}
			\\
		&=\frac{x_t+y_t}2-\frac{x_t^2+y_t^2}4+\tilde O_k(2^{-3k/2})
			=2^{-k-1}-\frac{\delta_t^2}{2^{2k+1}}+\tilde O_k(2^{-3k/2}).
	\end{align*}
Since part (3) of \Prop~\ref{Prop_pruning} implies that
 $\sum_t\pi_t\delta_t^2=k2^k\ln2+\tilde O(2^{k/2})$, the assertion follows.
\end{proof}

\begin{claim}\label{Claim_myLittleCalculation} 
\Whp\ we have 
	$\sum_{\ell}\pi_\ell\phivall=-2^{-k} - 2^{-2k-1}+k2^{-2k}+\tilde O_k(2^{-5k/2})$.
\end{claim}
\begin{proof}
Recall that for $\ell \in T^*$
	\begin{align*}
	\phivall&=-\KL{\ell_1^\rot,\ldots,\ell_{k_\ell}^\rot,1-\ell_1^\rot-\cdots-\ell_{k_\ell}^\rot}{g_{\ell,1}^\rot,\ldots,g_{\ell,k_\ell}^\rot,g_{\ell}^\cyan}
		+\sum_{j\in[k_\ell]}\KL{\ell_j^{\purpur}}{q_{\ell,j}^{\purpur}},
	\end{align*}
	cf.~(\ref{eqC}).
Claim~\ref{Claim_firstMomentSell3} asserts that $q_{\ell,j}^\purpur=\ell_j^\purpur-2^{-k_\ell-1}+\tilde O_k(2^{-3k/2})$. Using that $\KL{x}{x + \delta} = \frac{\delta^2}{2x(1-x)} + O(\delta^3)$ for $1/4 \le x \le 3/4$ and $\delta \to 0$  and recalling that $\ell_j^\purpur=1/2+\tilde O_k(2^{-k/2})$ yields
	\begin{equation}\label{eqClaim_myLittleCalculation2}
	\KL{\ell_j^\purpur}{q_{\ell,j}^\purpur}
	=2^{-2k_\ell-1}+\tilde O_k(2^{-5k/2}).
	\end{equation}
	Further, note that since $q_{\ell,j}^\purpur = \ell_j^\purpur - 2^{-k_\ell-1} + \tilde O_k(2^{-3k/2})$ and $q_{\ell,j}^\y = \ell_j^\y + 2^{-k_\ell-1} + \tilde O_k(2^{-3k/2})$
	\[
		g_{\ell,j}^\rot = \ell_j^\rot + \varepsilon_j, \text{where } \varepsilon_j = \tilde O_k(2^{-2k})
		\quad\text{and}\quad
		g_\ell^\cyan = 1-\sum_{j \in [k_\ell]}\ell_j^\rot + \varepsilon_{k_\ell+1}, \text{where } \varepsilon_{k_\ell+1} = - \prod_{j\in[k_\ell]}q_{\ell,j}^\y +  \tilde O_k(2^{-2k}).
	\]
Using that $x\log\left(\frac{x+\delta}{x}\right) = x - \frac{\delta^2}{2x} + O(\delta^3/x^2)$ for any $x > 0$ and $\delta > -x$ we obtain that the first term in the expression for $\phivall$ equals
\[
	\sum_{j \in[k_\ell+1]}\varepsilon_j - \sum_{j\in [k_\ell]} \frac{\varepsilon_j^2}{2 \ell_j^\rot} - \frac{\varepsilon_{k_\ell+1}^2}{2(1-\sum_{j \in [k_\ell]}\ell_j^\rot)} + \tilde O_k(2^{-3k}).
\]
Note that $\sum_{j \in[k_\ell+1]}\varepsilon_j = -\prod_{j \in [k_\ell]}q_{\ell,j}^\y$. Using that $\varepsilon_j = \tilde O_k(2^{-2k})$ and $\ell_j^\rot = O(2^{-k})$ for $j \in [k_\ell]$ we obtain that 
\begin{equation}\label{eq:dklqy}
\begin{split}
	-\KL{\ell_1^\rot,\ldots,\ell_{k_\ell}^\rot,1-\ell_1^\rot-\cdots-\ell_{k_\ell}^\rot}{g_{\ell,1}^\rot,\ldots,g_{\ell,k_\ell}^\rot,g_{\ell}^\cyan}
	& =-\prod_{j \in [k_\ell]}q_{\ell,j}^\y - \frac12\Big(\prod_{j \in [k_\ell]}q_{\ell,j}^\y\Big)^2 + \tilde O_k(2^{-3k}) \\
	& = -\prod_{j \in [k_\ell]}q_{\ell,j}^\y - 2^{-2k_\ell-1} + \tilde O_k(2^{-5k/2}).
\end{split}
\end{equation}
Since $q_{\ell,j}^\y = 1 - q_{\ell,j}^\purpur = 1 - \ell_j^\purpur + 2^{-k_\ell-1} + \tilde O_k(2^{-3k/2})$
\[
	\prod_{j \in [k_\ell]}q_{\ell,j}^\y  = 
	\prod_{j \in [k_\ell]}(1- \ell_j^\purpur) + k_\ell2^{-2k_\ell} + \tilde O_k(2^{-5k/2}).
\]
By plugging this into~\eqref{eq:dklqy} and using~\eqref{eqClaim_myLittleCalculation2} we arrive with \Prop~\ref{Prop_pruning} at the expression
\begin{equation}\label{eqClaim_myLittleCalculation4'}
	\sum_{\ell\in T^*}\pi_\ell\phivall = -(k+1)2^{-2k-1} -\sum_{\ell\in T^*}\pi_\ell\prod_{j\in[k_\ell]}(1-\ell_j^\purpur) + \tilde O_k(2^{-5k/2}).
\end{equation}
For each clause type $\ell$ and every $j$, the value $\ell_j^\purpur$ is determined merely by the signature of the $j$th literal.
Thus, for integers $d^+,d^-$ let
	$\rho(d^+,d^-)=\abs{\cbc{(l,j)\in\cL':d_l=d^+,d_{\neg l}=d^-}}/(2n).$
Then \Prop~\ref{Prop_pruning} implies that \whp\ for all $d^+,d^-$ we have
	\begin{equation*}\label{eqClaim_myLittleCalculation5}
	\rho(d^+,d^-)=\frac{d^+}{kr/2}\pr\brk{\Po(kr/2)=d^+}\pr\brk{\Po(kr/2)=d^-}+O_k(\exp(-k^2)).
	\end{equation*}
Furthermore, for a sequence $(d_1^+,d_1^-,\ldots,d_{k}^+,d_{k}^-)$
let $m(d_1^+,d_1^-,\ldots,d_{k}^+,d_{k}^-)$ be the
the number of indices $i\in[m]$ such that
	$d_{\PHI_{ij}'}=d_j^+$, $d_{\neg\PHI_{ij}'}=d_j^-$ for all $j\in[k]$.
Then by \Prop~\ref{Prop_pruning}  \whp
	\begin{equation*}\label{eqClaim_myLittleCalculation6}
	m(d_1^+,d_1^-,\ldots,d_{k}^+,d_{k}^-)/m=O_k(\exp(-k^2))+\prod_{j\in[k]}\rho(d_j^+,d_j^-).
	\end{equation*}
Letting
	$s=\sum_{d^+,d^-\geq0}\frac{2d^+\thet^0(d^+-d^-)}{kr}\cdot\frac{(kr/2)^{d^++d^-}}{(d^+)!(d^-)!\exp(kr)},$
we obtain from~\eqref{eqClaim_myLittleCalculation4'}
	\begin{align*}\label{eqClaim_myLittleCalculation7}
	\sum_{\ell\in T^*}\pi_\ell\phivall&=-(k+1)2^{-2k-1} - s^k +\tilde O_k(2^{-5k/2}).
	\end{align*}
By plugging in the definition of $\thet^0(d^+-d^-)$ we obtain
	$s=\frac12(1-3\cdot2^{-k-1})$ and the claim follows.
\end{proof}

\begin{claim}\label{Claim_asymptoticOcc} 
\Whp\ we have $2\sum_{t\in T}\pi_t\phioct=-2^{-k}-k2^{-k}\ln2+\tilde O_k(2^{-3k/2})$.
\end{claim}
\begin{proof}
Note that $\KL{x}{x+\delta} = \frac{\delta^2}{2x(1-x)} + O(x^{-1}\delta^3)$ for $x \in (0,1/2)$. Using \Lem~\ref{Lemma_Lambda} and Claim~\ref{Claim_occFirstMoment_3}, we obtain
	\begin{equation}\label{eqClaim_asymptoticOcc1}
	\sum_{h\in[d_t]}\KL{t_h^\rot/t_h^\purpur}{q_{t,h}^\rot}=d_t\tilde O_k(8^{-k})=\tilde O_k(2^{-3k/2})\quad\mbox{ for any $t\in T$}.
	\end{equation}
Further, note that Claim~\ref{Claim_occFirstMoment_3} guarantees that $$q_{t,h}^\rot = 2^{-k_t + 1} + \tilde O_k(2^{-3k/2})$$ for any $t\in T$ and some {$k_t \in \{k-2,k-1,k\}$}. Invoking \Lem~\ref{Lemma_Lambda} and (\ref{eqDegreesAfterPruning}), we find $s_t=1-\prod_{h\in[d_t]}(1-q_{t,h}^\rot)=1 - 2^{-k_t}+\tilde O_k(2^{-3k/2})$ for any $t\in T$. The expansion $\ln(1-x)=-x+O(x^2)$ as $x\ra0$ then yields
	\begin{equation}\label{eqClaim_asymptoticOcc2}
	t^1\ln(s_t)=-2^{-k_t-1}+\tilde O_k(2^{-3k/2})\quad\mbox{ for any $t\in T$}.
	\end{equation}	
In a similar fashion we obtain by applying again (\ref{eqDegreesAfterPruning})
	\begin{align}\label{eqClaim_asymptoticOcc3}
	t^*\ln(1-s_t)&=2^{-k-1}\sum_{h\in[d_t]}\ln(1-q_{t,h}^\rot)=-k2^{-k_t-1}\ln2+\tilde O_k(2^{-3k/2})\quad\mbox{ for any $t\in T$}.
	\end{align}
	Note that the number of types $t \in T$ such that $k_t \neq k$ is in $\exp(-O_k(k^2))n$, by Proposition~\ref{Prop_pruning}.
Combining~(\ref{eqClaim_asymptoticOcc1})--(\ref{eqClaim_asymptoticOcc3}) and summing over all $t \in T$  completes the proof.
\end{proof}

\noindent
Finally, \Cor~\ref{Cor_firstMomentFormula} follows from Claims~\ref{Lemma_firstMomentEntropy}, \ref{Claim_myLittleCalculation} and~\ref{Claim_asymptoticOcc}.

\subsection{Extendibility}\label{Sec_faithful}
The aim in this section is to establish

\begin{lemma}\label{Prop_faithful}
Let $\cZ''$ be the number of valid $\theta$-shades that fail to be extendible.
Then $\Erw_\cT[\cZ'']=o(\Erw_\cT[\cZ'])$ \whp	
\end{lemma}

To prove \Prop~\ref{Prop_faithful}, we are going to argue that given that $\xi$ is a valid $\theta$-shade,
the probability that $\xi$ is extendible is $1-o(1)$.
Thus, let $\xi$ be a $\theta$-shade, and let $\cV$ be the event that $\xi$ is valid.
To extend $\xi$ to a satisfying assignment, we need to assign actual truth values to literals $l$ such that $\hat\xi(l)=*$ in such a way that all clauses are satisfied.
In the course of this we just need to watch out for clauses that contain yellow and green clones only, because all other clauses already contain a literal set to $1$ under $\hat\xi$.
Thus, let $Y_\ell$ be the number of clauses of type $\ell$ containing green and yellow clones only.
We begin by showing a rough estimate regarding yellow clones only.
\begin{claim}\label{Claim_exmuell}
For $\ell\in T^*$ and $j\in[k_\ell]$ let $\mu_{\ell,j}$ be the number of clauses of type $\ell$ that contain precisely $k_\ell-j$ yellow clones.
Then $\pr_\cT[\bink kj 2^{-1-k}\leq\mu_{\ell,h}/{m_\ell}\leq \bink kj 2^{1-k}~|~\cV]=1-o(1)$ \whp
\end{claim}
\begin{proof}
Let $\tilde\cZ$ be the number of valid $\theta$-shades such that 
the number of clauses of type $\ell$ that contain precisely $k_\ell-j$ yellow clones does not lie in the interval  $I=[\bink kj 2^{-1-k}m_\ell, \bink kj 2^{1-k}m_\ell]$.
Recall the events $\cB(\xi),\cS(\xi)$ that are defined in Section~\ref{Sec_firstMoment_validity_probability} and the event $\cR(\xi)$ from Section~\ref{ssec:first_moment_occ}. Then $\cV = \cB(\xi) \cap \cS(\xi) \cap \cR(\xi)$. Moreover, by \Cor~\ref{Claim_muell} the probability of the event $\cI(\xi)$ that the number of clauses of type $\ell$ with precisely $k_\ell - j$ yellow clones does not belong to $I$ satisfies $\pr_\cT\brk{\cI(\xi)|\cB(\xi),\cS(\xi)}\leq\exp(-\Omega(n))$. We thus obtain
\[
	\pr_\cT[\mu_{\ell,h}/m_\ell \not \in I ~|~ \cV]
	= \frac{\pr_\cT\brk{\cI(\xi), \cB(\xi),\cS(\xi),\cR(\xi)}}{\pr_\cT[\cV]}
	= \frac{\pr_\cT\brk{\cI(\xi), \cR(\xi) ~|~ \cB(\xi),\cS(\xi)}}{\pr_\cT[\cV]}
\]
Note that the events $\cI(\xi), \cR(\xi)$ are independent upon conditioning on $\cB(\xi),\cS(\xi)$; the claimed bound follows. 
\end{proof}
We continue with a rough bound on the number $Y_\ell$ of clauses of type $\ell$ containing only green and yellow clones.
\begin{claim}\label{Lemma_greenAndYellow}
Let $\ell\in T^*$.
Then $\pr_\cT[Y_\ell\leq k^32^{-3k}m_\ell~|~\cV]=1-o(1)$ \whp
\end{claim}
\begin{proof} Let $\mu_{\ell,j}$ be the number of clauses of type $\ell$ that contain precisely $k_\ell-j$ yellow clones. By Claim~\ref{Claim_exmuell} \whp
	\begin{equation}\label{eqLemma_greenAndYellow1}
	\bink kj 2^{-1-k}\leq\frac{\mu_{\ell,j}}{m_\ell}\leq \bink kj 2^{1-k}\qquad\mbox{ for all $2\leq j\leq k_\ell$.}
	\end{equation}
If a clause contains $k_\ell-j$ yellow clones for some $2\leq j\leq k_\ell$, then the other $j$ clones are colored either green or blue (and there is no red clone). Let $Y_{\ell,j}$ be the number of clauses of type $\ell$ with precisely $j$ green clones and $k_\ell-j$ yellow clones.
Since for each type $t\in T$ we have $t^\blau=\frac12+o_k(1)$ and $t^\grun= 2^{-k-1}$, we see that 
	\begin{equation}\label{eqLemma_greenAndYellow2}
	\Erw_\cT[Y_{\ell,j}|\cV]\leq (1+o_k(1))2^{-kj}\mu_{\ell, j}.
	\end{equation}
Furthermore, since $n_t=\Omega(n)$ for all $t\in T$, the events that for two given clauses
of type $\ell$ with $k_\ell-j$ yellow clones the other $j$ clones are green are asymptotically independent.
Hence, 
	\begin{equation*}\label{eqLemma_greenAndYellow3}
	\Erw_\cT[Y_{\ell,j}^2|\cV]=(1+o(1))\Erw_\cT[Y_{\ell,j}|\cV]^2.
	\end{equation*}
Combining this with~(\ref{eqLemma_greenAndYellow2}), we conclude that 
	\begin{equation}\label{eqLemma_greenAndYellow4}
	\pr_\cT[Y_{\ell,j}\leq (1+o_k(1))2^{-kj}\mu_{\ell,j}|\cV]=1-o(1).
	\end{equation}
Since $Y_\ell=\sum_{j=2}^{k_\ell}Y_{\ell,j}$, combining~(\ref{eqLemma_greenAndYellow4}) and~(\ref{eqLemma_greenAndYellow1}) yields
$Y_\ell\leq k^32^{-3k}m_\ell$ \whp, as desired.
\end{proof}

Equipped with Claim~\ref{Lemma_greenAndYellow}, we are going to reduce the problem of extending $\xi$ to a  satisfying assignment of $\hat\PHI$ to a 2-SAT problem.
More precisely, let $\tilde\PHI$ be the $2$-SAT formula obtained from $\hat\PHI$ as follows:
\begin{itemize}
\item remove all clauses that contain a blue or a red clone.
\item turn all the remaining clauses (that consist of yellow clones and at least two green clones each) into clauses of length two
	by only keeping the first two green clones.
\end{itemize}
To satisfy $\tilde\PHI$, we borrow an argument from prior work on random $2$-SAT~\cite{mick,Goerdt}.
Namely, for $h\geq1$ we call a literal sequence $l_0,\ldots,l_{h+1}\in\hat\xi^{-1}(*)$ an \bemph{$h$-bicyle} if
 the following conditions are satisfied.
\begin{description}
\item[BC1]  For any $i=0,\ldots,h$ the $2$-clause $\neg l_i\vee l_{i+1}$ occurs in  $\tilde\PHI$.
\item[BC2] The variables $|l_1|,\ldots,|l_h|$ are   distinct.
\item[BC3] We have $|l_0|,|l_{h+1}|\in\cbc{|l_1|,\ldots,|l_h|}$.
\end{description}
It is well-known that a $2$-SAT formula is satisfiable unless it contains an $h$-bicycle for some $h\geq1$.
Thus, let $C_h$ be the number of $h$-bicycles in $\tilde\PHI$.
To get a handle on $C_h$, we use the following lemma.

\begin{lemma}\label{eqLemma_longBicycles}
There is an event $\cA$ with $\pr_\cT\brk{\cA}=1-o(1)$ such that the following is true.
Let $1\leq h\leq\ln^2n$ be an integer and let $C_h'$ be the number of sequences 
	$\vec l=(\neg l_1,j_1),(l_2,j_2'),(\neg l_2,j_2),\ldots,(l_{h-1},j_{h-1}'),(\neg l_{h-1},j_{h-1}),(l_{h},j_{h}')$
of distinct literal clones in $\xi^{-1}(\grun)$ such that 
$\tilde\PHI$ contains clauses consisting of the clones $(\neg l_i,j_i),(l_{i+1},j_{i+1}')$ for all $1\leq i<h$.
Then
	$\Erw_{\cT}[C_h'|\cA,\cV]\leq n\tilde O_k(2^{-k})^{h}$ \whp
\end{lemma}
\begin{proof}
Let $\cA$ be the event that  $Y_\ell\leq k^38^{-k}m_\ell$ for all $\ell\in T^*$.
Then $\pr_\cT\brk{\cA | \cV}=1-o(1)$ by \Lem~\ref{Lemma_greenAndYellow}.
We can estimate $\Erw_{\cT}[C_h'|\cA,\cV]$ as follows.
Let $l_1$ be a literal and let $\vec j=(j_1,j_2',j_2,\ldots,j'_{h-1},j_{h-1},j_h)\in[k2^k]$, $\vec i=(i_1,i_2,\ldots,i_h)\in[k]$ be sequences of indices.
Given $l_1$, $\vec j$, $\vec i$, we attempt to construct a sequence $l_2,\ldots,l_h$ of literals as follows.
If $j_1\leq d_{\neg l_1}$, then $l_2$ is the $i_1$th literal of the clause of $\hat\PHI$ that $(\neg l_1,j_1)$ occurs in, provided that $i_1$ does not exceed the length of that clause.
Similarly, assuming that $l_a$ has been defined already for some $1<a<h$ and that $j_a\leq d_{\neg l_a}$, let $l_{a+1}$ be the $i_a$th literal of the clause
of $\hat\PHI$ that the clone $(\neg l_a,j_a)$ occurs in.
For $b\in[h]$ let $\cE_b(l_1,\vec i,\vec j)$ be the event that the above construction yields a literal sequence $(l_1,\ldots,l_h)$, that for each $a\in[b]$ the clause
that $(\neg l_a,j_a)$ appears in contains green and yellow clones only, and that $\xi(l_a,j_a)=\grun$ for all $a\in[b]$.
Further, let $\cE(l_1,\vec i,\vec j)=\bigcap_{b\in[h]}\cE_b(l_1,\vec i,\vec j)$. We claim that
	\begin{align}\label{eqBike1}
	\pr_\cT\brk{\cE_{b+1}(l_1,\vec i,\vec j)|\cA,\cV,\cE_{b}(l_1,\vec i,\vec j)}&\leq\tilde O_k(4^{-k})&\mbox{for all }b<h.
	\end{align}
Indeed, let $\ell_{b+1}$ be the type of the clause that $(\neg l_b,j_b)$ appears in.
Given that $\xi(l_b,j_b)=\grun$, the probability that the clause contains green and yellow clones only is $\tilde O_k(4^{-k})$
	(due to our conditioning on $\cA$).
Multiplying~(\ref{eqBike1}) up for $b<h$, we obtain
	\begin{align}\label{eqBike2}
	\pr_\cT\brk{\cE(l_1,\vec i,\vec j)|\cA,\cV}&=\tilde O_k(4^{-k})^{h-1}.
	\end{align}
To complete the proof, we use the union bound.
The total number of ways of choosing $\vec i,\vec j$ is bounded by $\tilde O_k(2^k)^h$
	(note that we do not have to choose the indices $j_{a+1}'$; they are implied by $\neg l_a,j_a,i_a$).
Further, the total number of ways of choosing a literal $l_1$ with $\hat\xi(l_1)=*$ is $\tilde O_k(2^{-k})n$.
Combining these bounds with~(\ref{eqBike2}) yields the assertion.
\end{proof}

\begin{proof}[Proof of \Prop~\ref{Prop_faithful}]
Assume that $\tilde\PHI$ contains an $h$-bicycle for some $h>\ln n$.
Then there is a sequence $\vec l=(\neg l_1,j_1),(l_2,j_2'),(\neg l_2,j_2),\ldots,(l_{h^*},j_{h^*}')$
of length $h^*=\lfloor\ln n\rfloor$ of distinct clones in $\xi^{-1}(\grun)$ such that
$\tilde\PHI$ contains clauses consisting of  $(\neg l_i,j_i),(l_{i+1},j_{i+1}')$ for all $1\leq i<h^*$.
But by \Lem~\ref{eqLemma_longBicycles} the probability of this event is $o(1)$.
Thus, \whp\ there is no $h$-bicyle with $h>\ln n$.

We are left to show that \whp\ $C_h=0$ for all $1\leq h<\ln n$. Note that the number of choices for $l_0$ and $l_{h+1}$ is bounded by $O(\ln^2 n)$. Moreover, the number of the respective clones, and the positions where they appear in the corresponding clauses is bounded by $\tilde O_k(2^k)$. Once more by \Lem~\ref{eqLemma_longBicycles}, for any such $h$ we have
	$\Erw[C_h|\cV,\cA]\leq O(\ln^2 n/n)$.
Taking the union bound over all $1 \le h<\ln n$ completes the proof.
\end{proof}

\subsection{Separability}\label{Sec_separable}

The aim of this section is to prove the following statement.
\begin{lemma}\label{Prop_sep}
Let $\cZ'''$ be the number of valid $\theta$-shades that are not separable.
Then $\Erw[\cZ'''|\cT]=o(\Erw[\cZ'|\cT])$ \whp	
\end{lemma}

\noindent
In the proof we consider the set $\cY(\PHI)$ of all maps $\xi:L\ra\spins$ that enjoy the following properties.
	\begin{enumerate}[(i)]
	\item $\xi(\neg l)=\neg\xi(l)$ for all literals $l\in L$.
	\item $|\xi^{-1}(*)|\doteq 2^{-k}N$.
	\item Call a clause {\em critical} under $\xi$ if it contains one literal that is set to $1$ under $\xi$, while all others are set to $0$ then the number clauses of $\PHI$ that are critical under $\xi$ is  $(k2^{-k}+\tilde O_k(2^{-3k/2}))M$.
	\item  The restriction $\xi|_{L'}$ is a cover of $\PHI'$.
	\end{enumerate}
Further, for two maps $\xi_1,\xi_2:L\ra\cbc{0,1,*}$ and $z_1,z_2\in\spins$ define
	$$\cO^{z_1z_2}(\xi_1,\xi_2)=\frac{\abs{\xi_1^{-1}(z_1)\cap\xi_2^{-1}(z_2)}}{2N}\quad\mbox{and}
			\quad\cO(\xi_1,\xi_2)=(\cO^{z_1z_2}(\xi_1,\xi_2))_{z_1,z_2\in\spins}.$$
Let $\Pi=\cbc{\cO(\xi_1,\xi_2):\xi_1,\xi_2:L\ra\spins,\,|\xi_1^{-1}(*)|,|\xi_2^{-1}(*)|\doteq 2^{-k}N}$.
There are certain affine relations amongst the entries of $\cO\in\Pi$ that are implied by properties (i) and (ii):
	\begin{align}\label{eqOrel1}
	\cO^{10}&=\cO^{01},&\cO^{1*}=\cO^{0*}&=\cO^{*1}=\cO^{*0},\\
	\cO^{11}+\cO^{10}+\cO^{1*}&\doteq\frac12-2^{-k-1},
		&\cO^{00}+\cO^{01}+\cO^{0*}&\doteq\frac12-2^{-k-1},
		&\cO^{**}&\doteq2^{-k}-2\cO^{1*}.
		\label{eqOrel2}
	\end{align}
Here $A \doteq B$ shall be understood as $|A-B| = O(N^{-1})$. Note that due to these affine relations we can express all the entries of $\cO$ in terms of $\cO^{10},\cO^{1*}$.

  For $\cO\in\Pi$ let $Y(\cO)$ be the set of  pairs $\xi_1,\xi_2\in\cY(\PHI)$ with $\cO(\xi_1,\xi_2)=\cO$.
Moreover, for $z\in\spins$ we set $\cO^{z\nix}=\sum_{y\in\spins}\cO^{zy}$, $\cO^{\nix z}=\sum_{y\in\spins}\cO^{yz}$.
Further, we let $g=g(\cO)=(g^{\y\y},g^{\rot\grun},g^{\grun\rot},g^{\rot\y},g^{\y\rot})$ with
	\begin{align*}
	g^{\y\y}&=k(k-1)\cO^{10}\cO^{01}(\cO^{00})^{k-2},
		\qquad	g^{\rot\grun}=g^{\grun\rot}=k\cO^{1*}((\cO^{0\nix})^{k}-(\cO^{00})^{k}),\\
	g^{\rot\y}&=g^{\y\rot}=k\cO^{10}\big((\cO^{0\nix})^{k-1}-(\cO^{00})^{k-1}-{(k-1)(\cO^{01}+\cO^{0*}) (\cO^{00})^{k-2}}\big),\\
	g^{\cyan\cyan}&=1-(\cO^{0\nix})^k-(\cO^{\nix0})^k+(\cO^{00})^k
		-k(\cO^{*\nix} + \cO^{1\nix})(\cO^{0\nix})^{k-1}
		-k(\cO^{\nix*} + \cO^{1\nix})(\cO^{\nix0})^{k-1}\\
		&\qquad + {k(k-1)(\cO^{*0} + \cO^{10})(\cO^{0*} + \cO^{01})(\cO^{00})^{k-2} + k(\cO^{**}+\cO^{*1}+\cO^{1*}+\cO^{11})(\cO^{00})^{k-1}}.
	\end{align*}
Additionally, let $\Gamma(\cO)$ be the set of all vectors
	$\gamma=(\gamma^{\y\y},\gamma^{\rot\grun},\gamma^{\rot\y},\gamma^{\grun\rot},\gamma^{\y\rot},\gamma^{\cyan\cyan})$
with non-negative entries such that
	$\gamma^{\y\y}+\gamma^{\rot\grun}+\gamma^{\rot\y}+\gamma^{\grun\rot}+\gamma^{\y\rot}+\gamma^{\cyan\cyan}=1$
and
	\begin{align*}
	\gamma^{\y\y}+\gamma^{\rot\y}&\geq\frac{2N}M\cO^{10}-8^{-k},&
	\gamma^{\y\y}+\gamma^{\y\rot}&\geq\frac{2N}M\cO^{01}-8^{-k},&
	\gamma^{\rot\grun}&\geq\frac{2N}M\cO^{1*}-8^{-k},&
	\gamma^{\grun\rot}&\geq\frac{2N}M\cO^{*1}-8^{-k}.
	\end{align*}
Set $\psi(\cO,\gamma)=H(\cO)- \frac MN\KL{\gamma}{g}-{2^{-k}}.$

\begin{claim}\label{Claim_Y''}
\Whp\ we have $\frac1N\ln\Erw_\cT|Y(\cO)|\leq\max_{\gamma\in\Gamma(\cO)}\psi(\cO,\gamma)+\tilde O_k(2^{-3k/2})$.
\end{claim}
\begin{proof}
Let $\xi_1,\xi_2:L\ra\spins$.
Under $(\xi_1,\xi_2)$, a clause $\PHI_i'$ of length $k_i=k$ is a 
	\begin{itemize}
	\item \bemph{$(\y,\y)$-clause} if there exist $j_1,j_2\in[k]$, $j_1\neq j_2$, such that $\xi_1(\PHI_{ij_1}')=1$, $\xi_2(\PHI_{ij_1}')=0$, 
		 $\xi_1(\PHI_{ij_2}')=0$, $\xi_2(\PHI_{ij_1}')=1$, and $\xi_1(\PHI_{ij}')=\xi_2(\PHI_{ij}')=0$ for all $j\in[k]\setminus\{j_1,j_2\}$.
	\item \bemph{$(\rot,\grun)$-clause} if there exist $j_1\neq j_2$ such that
		$\xi_1(\PHI_{ij_1}')=1$, $\xi_2(\PHI_{ij_1}')=*$, $\xi_2(\PHI_{ij_2}')\neq0$ and  $\xi_1(\PHI_{ij})=0$ for all $j\neq j_1$.
	\item \bemph{$(\grun,\rot)$-clause} if there exist $j_1\neq j_2$ such that
		$\xi_2(\PHI_{ij_1}')=1$, $\xi_1(\PHI_{ij_1}')=*$, $\xi_1(\PHI_{ij_2}')\neq0$ and  $\xi_2(\PHI_{ij})=0$ for all $j\neq j_1$.
	\item \bemph{$(\rot,\y)$-clause} if there exist distinct indices $j_1,j_2,j_3$ such that
		$\xi_1(\PHI_{ij_1}')=1$, $\xi_2(\PHI_{ij_1}')=0$, $\xi_2(\PHI_{ij_2}'),\xi_2(\PHI_{ij_3}')\neq0$ and if $\xi_1(\PHI_{ij})=0$ for all $j\neq j_1$.
	\item \bemph{$(\y,\rot)$-clause} if there exist distinct indices $j_1,j_2,j_3$ such that
		$\xi_2(\PHI_{ij_1}')=1$, $\xi_1(\PHI_{ij_1}')=0$, $\xi_1(\PHI_{ij_2}'),\xi_1(\PHI_{ij_3}')\neq0$ and if $\xi_2(\PHI_{ij})=0$ for all $j\neq j_1$.
	\item \bemph{$(\cyan,\cyan)$-clause} if there exist $j_1,j_2,j_1',j_2'$ such that $j_1\neq j_2$, $j_1'\neq j_2'$
		such that $\xi_1(\PHI_{ij_1}'),\xi_1(\PHI_{ij_2}')\neq0$ and  $\xi_2(\PHI_{ij_1'}'),\xi_2(\PHI_{ij_2'}')\neq0$.
	\end{itemize}
For a set $\cM\subset[M]$ of size $|\cM|\geq(1-\exp(-k^2))M$ let $\cE(\gamma,\cM,\xi_1,\xi_2)$ be the event that
	\begin{itemize}
	\item  for any $(z_1,z_2)\in\{(\y,\y),(\rot,\grun),(\grun,\rot),(\rot,\y),(\y,\rot),(\cyan,\cyan)\}$ there are
			$\gamma^{z_1z_2}|\cM|$ indices $i\in\cM$ such that $\PHI_i$ is a $(z_1,z_2)$-clause under $(\xi_1,\xi_2)$, and
	\item there are $(2^{-k}+O_k(2^{-3k/2})M$ indices $i\in\cM$ such that $\PHI_i$ is critical under $\xi_1$.
	\end{itemize}
By the independence of the clauses we have
	\begin{equation}\label{eqClaim_Y''3}
	\ln\pr\brk{\PHI\in\cE(\gamma,\cM,\xi_1,\xi_2)}\leq-|\cM|\KL{\gamma}{g}+o(1).
	\end{equation}
Further,  let $\cN(\cO)$ be the set of all pairs $(\xi_1,\xi_2)$ such that $\xi_1,\xi_2:L\ra\spins$ satisfy (i) and $\cO(\xi_1,\xi_2)=\cO$.
Then by Fact~\ref{Fact_entropyFunction},
	\begin{equation}\label{eqClaim_Y''2}
	|\cN(\cO)|=\exp(NH(\cO)+o(N)).
	\end{equation}
In addition, for a set $W$ of literals such that $\xi_1(l)=1$ for all $l\in W$ and $|W|\geq(1-O_k(2^{-k}))N$ let $\cB(W,\xi_1)$ be the event that
each $w\in W$ occurs in a clause of $\PHI$ that is critical under $\xi_1$.
Then 
	\begin{equation}\label{eqClaim_Y''4}
	\ln\pr\brk{\PHI\in\cB(W,\xi_1)|\cE(\gamma,\cM,\xi_1,\xi_2)}\leq N(-2^{-k}+\tilde O_k(2^{-3k/2})).
	\end{equation}
Indeed, given $\cE(\gamma,\cM,\xi_1,\xi_2)$ there are $(2^{-k}+O_k(2^{-3k/2}))M=(k\ln 2+O_k(2^{-k/2}))N$ clauses that are critical under $\xi_1$.
If we think of these clauses as ``red balls'' that are tossed into bins corresponding to the literals $W$, then a short calculation shows that the probability that no bin remains empty is $\exp(-|W|(2^{-k}+O_k(2^{-3k/2})))$.

By \Prop~\ref{Prop_pruning} we may assume
that $|V'|\geq N(1-\exp(-k^2))$ and $\sum_{x\not\in V'}D_x+D_{\neg x}\leq\exp(-k^2)N$.
If so, then there exist $\cM$, $W$, $\gamma\in\Gamma(\cO)$ such that $\cB(W,\xi_1)\cap\cE(\gamma,\cM,\xi_1,\xi_2)$ occurs for all $(\xi_1,\xi_2)\in Y(\cO)$.
To see this, let
$\cM$ be the set of all $i\in[M]$ such that $k_i=k$.
If $\xi_1,\xi_2$ are covers of $\PHI'$, then every $\PHI_i$ with $i\in\cM$ must
be a $(z_1,z_2)$-clause for some $(z_1,z_2)\in\{(\y,\y),(\rot,\grun),(\grun,\rot),(\rot,\y),(\y,\rot),(\cyan,\cyan)\}$.
In addition, let $L''$ be the set of all literals $l\in L'$ that do not occur in clauses $\PHI_i$ with $i\not\in\cM$.
Then each $l\in L''$ with $\xi_1(l)=1,\xi_2(l)=*$ must occur in a $(\rot,\grun)$-clause.
Hence, there are at least $(2\cO^{1*}-O_k(\exp(-k^2)))N$  $(\rot,\grun)$-clauses.
Arguing similarly for $(\y,\y),(\rot,\y),(\y,\rot),(\grun,\rot)$-clauses,
we conclude that there is $\gamma\in\Gamma(\cO)$ such that $\cE(\gamma,\cM,\xi_1,\xi_2)$ occurs.
Further, for $W=\{l\in L'':\xi_1(l)=1\}$ the event $\cB(W,\xi_1)$ occurs.
Finally, the assertion follows from~(\ref{eqClaim_Y''2})--(\ref{eqClaim_Y''4}) and the union bound.
\end{proof}

\noindent
For $\cO\in\Pi$ we set	$\Delta(\cO)=1-\sum_{z\in\spins}\cO^{zz}.$

\begin{claim}\label{Lemma_noMiddleGround}
Assume that $\Delta(\cO)\in[2^{-0.99k},\frac12-2^{-0.49k}]\cup[\frac12+2^{-0.49k},1]$. 
Then $\sup_{\gamma\in\Gamma(\cO)}\psi(\cO,\gamma)<-\Omega_k(2^{-k})$.
\end{claim}
\begin{proof}
We claim that
	\begin{equation}\label{eqLemma_noMiddleGround1}
	\frac1N\ln\Erw_\cT|Y(\cO)|\leq H(\cO)+\frac MN\ln\brk{1-(\cO^{0\nix})^k-(\cO^{\nix0})^k+(\cO^{00})^k}+o(1).
	\end{equation}
Indeed, it is straightforward to check that $g^{z_1z_2}\leq 1-(\cO^{0\nix})^k-(\cO^{\nix0})^k+(\cO^{00})^k$ for all $(z_1,z_2)$.
Hence, (\ref{eqLemma_noMiddleGround1}) follows from Claim~\ref{Claim_Y''}.
Further,  because $\cO^{0\nix}=\frac12+O_k(2^{-k})$, we find
	\begin{equation}\label{eqLemma_noMiddleGround2}
	\ln\brk{1-(\cO^{0\nix})^k-(\cO^{\nix0})^k+(\cO^{00})^k}\leq\ln\brk{1-2^{1-k}+(\cO^{00})^k}+\tilde O_k(2^{-k}).
	\end{equation}
In addition, because $|\xi_1^{-1}(*)|,|\xi_2^{-1}(*)|\doteq 2^{-k}N$, we have $H(\cO)\leq\tilde O_k(2^{-k})+H(\Delta(\cO))$.
Combining this estimate with (\ref{eqLemma_noMiddleGround1}) and~(\ref{eqLemma_noMiddleGround2}), we obtain
	\begin{equation}\label{eqLemma_noMiddleGround3}
	\frac1N\ln\Erw[Y(\cO)]\leq H(\Delta(\cO))+\frac MN\ln\brk{1-2^{1-k}+((1-\Delta(\cO))/2)^k}+\tilde O_k(2^{-k}).
	\end{equation}
Finally, it is elementary to verify that $\mbox{for all }y\in[2^{-0.99k},\frac12-2^{-0.49k}]\cup[\frac12+2^{-0.49k},1]$,
	\begin{equation}\label{eqLemma_noMiddleGround4}
	H(y)+\frac MN\ln\brk{1-2^{1-k}+((1-y)/2)^k}<-2^{-(1-\Omega_k(1))k}.
	\end{equation}
The assertion follows from~(\ref{eqLemma_noMiddleGround3}) and~(\ref{eqLemma_noMiddleGround4}).
\end{proof}

\noindent
Let $\Gamma'(\cO)$ be the set of all $\gamma\in\Gamma(\cO)$ such that
	\begin{align*}
	\gamma^{\y\y}+\gamma^{\rot\y}&=\frac{2N}M\cO^{10},&
	\gamma^{\y\y}+\gamma^{\y\rot}&=\frac{2N}M\cO^{01},&
	\gamma^{\rot\grun}&=\frac{2N}M\cO^{1*},&
	\gamma^{\grun\rot}&=\frac{2N}M\cO^{*1}.
	\end{align*}

\begin{claim}\label{Claim_gammaBoundary}
If $\cO\in\Pi$ is such that
	$\Delta(\cO)\leq2^{-0.99k}$, then
 $\sup_{\gamma\in\Gamma(\cO)}\psi(\cO,\gamma)\leq\sup_{\gamma\in\Gamma'(\cO)}\psi(\cO,\gamma)+\tilde O_k(2^{-3k/2})$.
\end{claim}
\begin{proof}
If $\Delta(\cO)\leq2^{-0.99k}$, then $\cO^{00}= 1/2 +O_k(2^{-2k/3})$ follows from the relations~(\ref{eqOrel1})--~(\ref{eqOrel2}). We find that 
	\begin{align}\label{eqClaim_gammaBoundary1}
	g^{\y\y}&=(\cO^{10})^2\tilde O_k(2^{-k}),&
		(\cO^{1*})^2\tilde\Omega_k(2^{-k})\leq g^{\rot\grun},g^{\grun\rot}&\leq\cO^{1*}2^{-k-\Omega_k(k)},\\
		g^{\cyan\cyan}&=1-\tilde O_k(2^{-k}),&
		(\cO^{10})^22^{-k-\Omega_k(k)}\leq g^{\rot\y},g^{\rot\y}&=\cO^{10}2^{-k-\Omega_k(k)}.
			\label{eqClaim_gammaBoundary2}
	\end{align}
Now, let $\gamma\in\Gamma(\cO)$ and obtain $\hat\gamma\in\Gamma(\cO)$ from $\gamma$ by
increasing the $(\y,\y)$, $(\rot,\grun)$, $(\grun,\rot)$ entries such that
	\begin{align*}
	\hat\gamma^{\y\y}+\hat\gamma^{\rot\y}&=\textstyle\max\{\frac{2N}M\cO^{10},\gamma^{\y\y}+\gamma^{\rot\y}\}&
		\hat\gamma^{\y\y}+\hat\gamma^{\y\rot}&=\textstyle\max\{\frac{2N}M\cO^{10},\gamma^{\y\y}+\gamma^{\y\rot}\},\\
	\hat\gamma^{\rot\grun}&=\textstyle\max\{\frac{2N}M\cO^{1*},\gamma^{\rot\grun}\}&
	\hat\gamma^{\grun\rot}&=\textstyle\max\{\frac{2N}M\cO^{*1},\gamma^{\grun\rot}\}
	\end{align*}
and by setting $\gamma^{\cyan\cyan}=1-\hat\gamma^{\y\y}-\hat\gamma^{\rot\y}-\hat\gamma^{\y\rot}-\hat\gamma^{\rot\grun}-\hat\gamma^{\grun\rot}$.
The bounds~(\ref{eqClaim_gammaBoundary1})--(\ref{eqClaim_gammaBoundary2}) imply that
for any $\alpha\in[0,1]$ at the point $\tilde\gamma=\alpha\hat\gamma+(1-\alpha)\gamma$ we have
	\begin{align*}
	\frac{\partial\KL{g}{\tilde\gamma}}{\partial\tilde\gamma^{\y\y}}&=\ln\gamma^{\y\y}-2\ln\cO^{10}+\tilde O_k(1),\\
	\frac{\partial\KL{g}{\tilde\gamma}}{\partial\tilde\gamma^{\rot\grun}}
		&=\ln\gamma^{\rot\grun}-2\ln\cO^{1*}+\tilde O_k(1),&
	\frac{\partial\KL{g}{\tilde\gamma}}{\partial\tilde\gamma^{\grun\rot}}
		&=\ln\gamma^{\grun\rot}-2\ln\cO^{1*}+\tilde O_k(1).
	\end{align*}
Integrating the above up for $\alpha\in[0,1]$ reveals that 
	\begin{equation}			\label{eqClaim_gammaBoundary3}
	\KL{\hat\gamma}{g}=\KL{\gamma}{g}+\tilde O_k(4^{-k}).
	\end{equation}
Finally, obtain $\dot\gamma\in\Gamma'(\cO)$ from $\hat\gamma$ by decreasing the $(\y,\y)$, $(\rot,\y)$, $(\y,\rot)$, $(\rot,\grun)$, $(\grun,\rot)$ entries.
Then~(\ref{eqClaim_gammaBoundary1})--(\ref{eqClaim_gammaBoundary2}) imply that $\KL{\dot\gamma}{g}\leq\KL{\hat\gamma}{g}$.
Thus, the assertion follows from~(\ref{eqClaim_gammaBoundary3}).
\end{proof}

\noindent
Recall that we can express all the entries of $\cO$ in terms of $\cO^{10},\cO^{1*}$.
With this substitution we obtain the following bound on the differential of $\psi$.

\begin{claim}\label{Claim_psiDiff}
If $\gamma\in\Gamma'(\cO)$ and
$\alpha\in[0,1]$ is such that $\gamma^{\y\y}=\alpha\frac{2N}M\cO^{10}$, then
	$$\bc{\frac{\partial\psi}{\partial\cO^{10}},\frac{\partial\psi}{\partial\cO^{1*}}}=
		\bc{-\Omega_k(k)-(1-\alpha)\ln\frac{\cO^{10}}{\cO^{10}+\cO^{1*}},
			-\Omega_k(k)-\ln\frac{\cO^{1*}}{\cO^{10}+\cO^{1*}}}.$$
\end{claim}
\begin{proof}
Because $\gamma\in\Gamma'(\cO)$,
the choice of $\alpha$ ensures that
	\begin{align*}
	\gamma^{\rot\y}&=(1-\alpha)\frac{2N}M\cO_{10},&\gamma^{\y\rot}&=(1-\alpha)\frac{2N}M\cO_{10},
	\end{align*}
For $(y_1,y_2)\neq(\cyan,\cyan)$ we obtain
	\begin{align*}
	\frac{\partial H(\cO)}{\partial\cO^{10}}&=2\ln\cO^{11}-2\ln\cO^{10},&
		\frac{\partial H(\cO)}{\partial\cO^{1*}}&=2\ln\cO^{11}+2\ln\cO^{**}-4\ln\cO^{1*},\\
	-\frac MN&\frac{\partial}{\partial g^{y_1y_2}}\KL{\gamma}{g}\frac{\partial g^{y_1y_2}}{\partial\cO^{10}}= O_k(1),&
	-\frac MN&\frac{\partial}{\partial g^{y_1y_2}}\KL{\gamma}{g}\frac{\partial g^{y_1y_2}}{\partial\cO^{1*}}= O_k(1).
	\end{align*}
Further,
	\begin{align*}
	-\frac MN&\frac{\partial}{\partial g^{\cyan\cyan}}\KL{\gamma}{g}\frac{\partial g^{\cyan\cyan}}{\partial\cO^{10}}=-\Omega_k(k),&
	-\frac MN&\frac{\partial}{\partial g^{\cyan\cyan}}\KL{\gamma}{g}\frac{\partial g^{\cyan\cyan}}{\partial\cO^{1*}}=-\Omega_k(k).
	\end{align*}
In addition,
	\begin{align*}
	-\frac MN&\frac{\partial}{\partial \gamma^{\y\y}}\KL{\gamma}{g}\frac{\partial\gamma^{\y\y}}{\partial\cO^{10}}=
		2\alpha\brk{\ln\cO^{10}+\ln\alpha+O_k(\ln k)},\\
	-\frac MN&\frac{\partial}{\partial \gamma^{\rot\y}}\KL{\gamma}{g}\frac{\partial\gamma^{\rot\y}}{\partial\cO^{10}}=
		2(1-\alpha)\brk{2\ln(\cO^{10}+\cO^{1*})+\ln(1-\alpha)+O_k(\ln k)},\\
	-\frac MN&\frac{\partial}{\partial \gamma^{\rot\grun}}\KL{\gamma}{g}\frac{\partial\gamma^{\rot\grun}}{\partial\cO^{1*}}=
		2\ln(\cO^{10}+\cO^{1*})+O_k(\ln k),\\	
	-\frac MN&\frac{\partial}{\partial \gamma^{\cyan\cyan}}\KL{\gamma}{g}\frac{\partial\gamma^{\cyan\cyan}}{\partial\cO^{10}}=
		O_k(1),\qquad
	-\frac MN\frac{\partial}{\partial \gamma^{\cyan\cyan}}\KL{\gamma}{g}\frac{\partial\gamma^{\cyan\cyan}}{\partial\cO^{1*}}=
		O_k(1).
	\end{align*}
Combining these estimates yields the assertion.
\end{proof}

\begin{claim}\label{Claim_noBoundary}
If $\cO\in\Pi$ is such that $\Delta(\cO)\leq2^{-0.99k}$, then $\sup_{\gamma\in\Gamma(\cO)}\psi(\cO,\gamma)\leq\eps_k2^{-k}+\tilde O_k(2^{-3k/2})$.
\end{claim}
\begin{proof}
By Claim~\ref{Claim_gammaBoundary} it suffices to show that $\sup_{\gamma\in\Gamma'(\cO)}\psi(\cO,\gamma)\leq\eps_k2^{-k}+\tilde O_k(2^{-3k/2})$.
To bound $\psi(\cO,\gamma)$ for $\gamma\in\Gamma'(\cO)$, let
$\cO_0$ be such that $\Delta(\cO_0)=0$ and $\gamma_0$ such that $\gamma_0^{\cyan\cyan}=1$.
Integrating the bound on the differential of $\psi$ from Claim~\ref{Claim_psiDiff} along the straight line from $(\cO,\gamma)$ to $(\cO_0,\gamma_0)$, we obtain
	$$\sup_{\gamma\in\Gamma'(\cO)}\psi(\cO,\gamma)\leq\psi(\cO_0,\gamma_0) +\tilde O_k(2^{-3k/2}).$$
Finally, an elementary calculation yields $\psi(\cO_0,\gamma_0)=\eps_k2^{-k}+\tilde O_k(4^{-k})$.
\end{proof}

\begin{proof}[Proof of \Lem~\ref{Prop_sep}.]
Let $X$ be the number of pairs $(\xi_1,\xi_2)\in\cY(\PHI)^2$ such
that $\Delta(\cO(\xi_1,\xi_2))\not\in\cI=[\frac12-2^{-0.49k},\frac12+2^{-0.49k}]$.
Claims~\ref{Claim_Y''}, \ref{Lemma_noMiddleGround} and~\ref{Claim_noBoundary} imply that
	\begin{equation}\label{eqproofProp_sep1}
	\pr\brk{\frac1N\ln\Erw[X|\cT]\leq\eps_k2^{-k}+\tilde O_k(2^{-3k/2})}=1-o(1).
	\end{equation}
If $\xi$ is a valid $\theta$-shade that fails to be separable, then there are $\Erw[\cZ'|\cT]$ $\theta$-shades $\zeta$
such that $\Delta(\cO(\hat\xi,\hat\zeta))\not\in\cI$.
Therefore, if $\Erw[\cZ'''|\cT]\geq\Erw[\cZ'|\cT]/N$ with a non-vanishing probability, then $X\geq\Erw[\cZ'|\cT]^2/N$ with a non-vanishing probability.
But this contradicts~(\ref{eqproofProp_sep1}), 
as \Cor~\ref{Cor_firstMomentFormula} shows that $\frac1N\ln\Erw[\cZ'|\cT] = \eps_k2^{-k}+\tilde O_k(2^{-3k/2})$ \whp\
\end{proof}

\begin{proof}[Proof of \Prop~\ref{Prop_firstMoment}]
The proposition is immediate from \Cor~\ref{Cor_firstMomentFormula}, \Lem~\ref{Prop_faithful} and \Lem~\ref{Prop_sep}.
\end{proof}

\section{The second moment}\label{Sec_smmFull}

\subsection{The overlap}\label{Sec_theOverlap}

\noindent
The aim is to calculate the second moment $\Erw_\cT[\cZ^2]$ of the number $\cZ$ of good $\theta$-shades of  $\hat\PHI$.
Let $\Xi$ denote the set of all $\theta$-shades.
Of course, the second moment is nothing but the expected number of {\em pairs} $(\xi_1,\xi_2)$ of good $\theta$-shades.
As outlined in \Sec~\ref{Sec_SP}, what we need to show is that \whp\
the dominant contribution to the second moment comes from pairs $\xi_1,\xi_2$ that	``look uncorrelated''.

Thus, we need a measure of how ``similar'' two $\theta$-shades $\xi_1,\xi_2\in\Xi$ are.
For any literal type $t\in T$ and $z_1,z_2\in\spins$ we let
	$\omega^{z_1z_2}_{t}(\xi_1,\xi_2)$ be the fraction of literals $l$ of type $t$ such that $\hat\xi_1(l)=z_1$ and $\hat\xi_2(l)=z_2$.
That is,
	$$\omega^{z_1z_2}_{t}(\xi_1,\xi_2)=\frac1{n_t}{\abs{\cbc{l\in L_t':\hat\xi_1(l)=z_1,\hat\xi_2(l)=z_2}}}.$$
In addition, for $t\in T$, $h\in[d_t]$ and $(z_1,z_2)\in\cbc{(\rot,\rot),(\rot,\cyan),(\cyan,\rot),(\rot,\y),(\y,\rot)}$ we let
	$$\omega_{t,h}^{z_1z_2}(\xi_1,\xi_2)=\frac1{n_t}{\abs{\cbc{l\in L_t':\xi_1(l,h)=z_1,\xi_2(l,h)=z_2}}}.$$
Further, for a clause type $\ell\in T^*$, $j\in[k_\ell]$ and $z_1,z_2\in\cbc{\purpur,\y}$ we let
	$$\omega^{z_1z_2}_{\ell,j}(\xi_1,\xi_2)=\frac1{m_\ell}\abs{\cbc{i\in M_\ell:\xi_1(\hat\PHI_{ij})=z_1,\xi_2(\hat\PHI_{ij})=z_2}}.$$
The \emph{literal overlap} of $\xi_1,\xi_2$ is the vector $\omega(\xi_1,\xi_2)$ 
comprising all of the above.
Let $\Om=\cbc{\omega(\xi_1,\xi_2):\xi_1,\xi_2\in\Xi}.$ Given two $\theta$-shades $\xi_1,\xi_2$, we can think of each literal clone $(l,h)\in\cL$ as a ``domino'' adorned with two colors $(\xi_1(l,j),\xi_2(l,j))$.
Of course, if $(\xi_1,\xi_2)$ are good, then the placement of the dominos in the clauses has to satisfy certain constraints.
More precisely, every clause must satisfy one of the following seven conditions.

\begin{definition}\label{Def_clauseValidity}
Let $\ell\in T^*$ be a clause type and let $i\in M_\ell$.
Let $j,j'\in[k_\ell]$, $j\neq j'$.
We call $\hat\PHI_i$ a
\begin{enumerate}[(i)]
\item{\bf$(\rot,\rot,j)$-clause}
	if the domino in the $j$ position is colored $(\rot,\rot)$ and all other dominos are colored $(\y,\y)$.
	(Formally, $\xi_1(\hat\PHI_{ij})=\xi_2(\hat\PHI_{ij})=\rot$ and $\xi_1(\hat\PHI_{ij'})=\xi_2(\hat\PHI_{ij'})=\y$ for all $j'\neq j$.) 
\item{\bf $(\y,\y,j,j')$-clause} if the domino in position $j$ is colored $(\rot,\y)$, the domino in position $j'$ is colored $(\y,\rot)$, and all others are colored $(\y,\y)$.
\item {\bf $(\rot,\cyan,j)$-clause} if the domino in position $j$ is colored $(\rot,\cyan)$, all
	 other dominos are colored either $(\y,\y)$ or $(\y,\cyan)$, and there occurs at least one domino colored $(\y,\cyan)$.
\item {\bf$(\rot,\y,j)$-clause}
	if the domino in position $j$ is colored $(\rot,\y)$,
	all others are colored either $(\y,\y)$ or $(\y,\cyan)$, and there occur at least two dominos colored $(\y,\cyan)$.
\item{\bf $(\cyan,\rot,j)$} if the domino in position $j$ is colored $(\cyan,\rot)$, all
	 other dominos are colored either $(\y,\y)$ or $(\cyan,\y)$, and there occurs at least one domino colored $(\cyan,\y)$.
\item {\bf $(\y,\rot,j)$-clause}
	if the domino in position $j$ is colored $(\y,\rot)$,
	all others are colored either $(\y,\y)$ or $(\cyan,\y)$, and there occur at least two dominos colored $(\cyan,\y)$.	
\item{\bf $(\cyan,\cyan)$-clause} if all dominos are colored either $(\cyan,\cyan),(\cyan,\y),(\y,\cyan)$ or $(\y,\y)$
	and if there exist $j_1j_2,j_1',j_2'\in[k_\ell]$, $j_1\neq j_2$, $j_1'\neq j_2'$, such that the dominos in positions $j_1,j_2$ are colored either $(\cyan,\cyan)$ or $(\cyan,\y)$,
	and  the dominos in positions $j_1',j_2'$ are colored either $(\cyan,\cyan)$ or $(\y,\cyan)$.
\end{enumerate}
\end{definition}

For $\ell\in T^*$ and $j\in[k_\ell]$ let $\gamma^{\rot\rot}_{\ell,j}(\xi_1,\xi_2)$ denote the fraction of $(\rot,\rot,j)$-clauses among the clauses of type $\ell$, i.e.,
	$$\gamma^{\rot\rot}_{\ell,j}(\xi_1,\xi_2)=\frac1{m_\ell}{\abs{\cbc{i\in M_\ell:
			\hat\PHI_i\mbox{  is a $(\rot,\rot,j)$-clause}}}}.$$
We define $\gamma^{z_1z_2}_{\ell,j}(\xi_1,\xi_2)$  for $(z_1,z_2)\in\cbc{(\rot,\cyan),(\rot,\y),(\cyan,\rot),(\y,\rot)}$  analogously.
For $j_1,j_2\in[k_\ell]$, $j_1\neq j_2$ we let $\gamma^{\y\y}_{\ell,j_1,j_2}(\xi_1,\xi_2)$ signify the fraction of $(\y,\y,j_1,j_2)$-clauses among the clauses of type $\ell$.
In addition, let 	$\gamma^{\cyan\cyan}_{\ell}(\xi_1,\xi_2)$  be the fraction of $(\cyan,\cyan)$-clauses. 
Set
	$$
	\gamma_\ell(\xi_1,\xi_2)=(\gamma^{\rot\rot}_{\ell,j}(\xi_1,\xi_2),
		\gamma^{\rot\cyan}_{\ell,j}(\xi_1,\xi_2),\gamma^{\rot\y}_{\ell,j}(\xi_1,\xi_2),
		\gamma^{\cyan\rot}_{\ell,j}(\xi_1,\xi_2),\gamma^{\y\rot}_{\ell,j}(\xi_1,\xi_2),
			\gamma^{\y\y}_{\ell,j_1,j_2}(\xi_1,\xi_2),\gamma^{\cyan\cyan}_{\ell}(\xi_1,\xi_2))_{j,j_1\neq j_2}$$
and let $\gamma(\xi_1,\xi_2)=(\gamma_\ell(\xi_1,\xi_2))_{\ell\in T^*}$.
We call $\gamma(\xi_1,\xi_2)$ the \emph{clause overlap} of $\xi_1,\xi_2$.
For $\omega\in\Om$  let $$\Gamma(\omega)=\cbc{\gamma(\xi_1,\xi_2):\xi_1,\xi_2\mbox{ are good $\theta$-shades with }\omega(\xi_1,\xi_2)=\omega}.$$
There are some immediate affine relations between the entries of the literal and the clause overlap. More specifically, we have

\begin{fact}\label{Fact_affine}
If $\omega\in\Omega$ and $\gamma\in\Gamma(\omega)$, then for each $t\in T$ and $h\in[d_t]$ we have
	\begin{align*}
	\omega_{t,h}^{\rot\rot}&=\sum_{(\ell,j)\in\partial(t,h)}\frac{m_\ell}{n_t}\gamma_{\ell,j}^{\rot\rot},\quad
	\omega_{t,h}^{\rot\cyan}=\sum_{(\ell,j)\in\partial(t,h)}\frac{m_\ell}{n_t}\gamma_{\ell,j}^{\rot\cyan},\quad
	\omega_{t,h}^{\cyan\rot}=\sum_{(\ell,j)\in\partial(t,h)}\frac{m_\ell}{n_t}\gamma_{\ell,j}^{\cyan\rot},\\
	\omega_{t,h}^{\rot\y}&=\sum_{(\ell,j)\in\partial(t,h)}\frac{m_\ell}{n_t}\brk{\gamma_{\ell,j}^{\rot\y}+\sum_{j'\neq j}\gamma_{\ell,j,j'}^{\y\y}},\quad
		\omega_{t,h}^{\y\rot}=\sum_{(\ell,j)\in\partial(t,h)}\frac{m_\ell}{n_t}\brk{\gamma_{\ell,j}^{\y\rot}+\sum_{j'\neq j}\gamma_{\ell,j',j}^{\y\y}},\quad
	\end{align*}
Furthermore,
	\begin{align*}
	\omega_t^{11}+\omega_t^{1*}+\omega_t^{*1}+\omega_t^{**}&=\sum_{(\ell,j)\in\partial(t,h)}\frac{m_\ell}{n_t}\omega_{\ell,j}^{\purpur\purpur},&
		\omega_t^{10}+\omega_t^{*0}&=\sum_{(\ell,j)\in\partial(t,h)}\frac{m_\ell}{n_t}\omega_{\ell,j}^{\purpur\y},\\
	\omega_t^{01}+\omega_t^{0*}&=\sum_{(\ell,j)\in\partial(t,h)}\frac{m_\ell}{n_t}\omega_{\ell,j}^{\y\purpur},&
	\omega_t^{00}&=\sum_{(\ell,j)\in\partial(t,h)}\frac{m_\ell}{n_t}\omega_{\ell,j}^{\y\y}.
	\end{align*}
In addition, for each $y\in\spins$ we have
	\begin{align*}
	t^y&\doteq\sum_{z\in\spins}\omega_t^{yz}\doteq\sum_{z\in\spins}\omega_t^{zy}.
	\end{align*}
Finally, for all $\ell\in T^*$ and $j\in[k_\ell]$,
	\begin{align*}
	\ell_j^\purpur&\doteq\omega_{\ell,j}^{\purpur\purpur}+\omega_{\ell,j}^{\purpur\y},&
		\ell_j^\purpur&\doteq\omega_{\ell,j}^{\purpur\purpur}+\omega_{\ell,j}^{\y\purpur},&
		\ell_j^\y&\doteq\omega_{\ell,j}^{\y\y}+\omega_{\ell,j}^{\y\purpur},&
		\ell_j^\y&\doteq\omega_{\ell,j}^{\y\y}+\omega_{\ell,j}^{\purpur\y},\\
	\ell_j^\rot&=\sum_{z\in\cbc{\rot,\cyan,\y}}\gamma_{\ell,j}^\rot+\sum_{j'\neq j}\gamma_{\ell,j,j'}^{\y\y}=
		\sum_{z\in\cbc{\rot,\cyan,\y}}\gamma_{\ell,j}^\rot+\sum_{j'\in[k_\ell]\setminus\{j\}}\gamma_{\ell,j,j'}^{\y\y}.\hspace{-6cm}
	\end{align*}
\end{fact}

The ultimate goal is show that the the second moment $\Erw_\cT[\cZ^2]$ is dominated by pairs $(\xi_1,\xi_2)$ 
whose overlap is close to the ``uncorrelated'' value $\bar\omega,\bar\gamma$ defined by
	\begin{align*}
	\bar\omega_{t}^{z_1z_2}&=t^{z_1}t^{z_2},&(t\in T,z_1,z_2\in\spins),\\
	\bar\omega_{t,h}^{z_1z_2}&=t_h^{z_1}t_h^{z_2},&(t\in T,h\in[d_t],z_1,z_2\in\reds),\\
	\bar\omega_{\ell,j}^{z_1z_2}&=\ell_j^{z_1}\ell_j^{z_2},&(\ell\in T^*,j\in[k_\ell],z_1,z_2\in\cbc{\purpur,\y}),\\
	\bar\gamma_{\ell,j}^{z_1z_2}&=\ell_j^{z_1}\ell_j^{z_2},&
		(\ell\in T^*,j\in[k_\ell],(z_1,z_2)\in\cbc{(\rot,\rot),(\rot,\cyan),(\cyan,\rot)}),\\
	\bar\gamma_{\ell,j,j'}^{\y\y}&=\ell_j^{\rot}\ell_{j'}^{\rot},&(\ell\in T^*,j,j'\in[k_\ell],j\neq j'),\\
	\bar\gamma_{\ell,j}^{\rot\y}&=\bar\gamma_{\ell,j}^{\y\rot}=\ell_j^\rot(1-\ell_j^{\purpur})-\ell_j^{\rot}\sum_{j'\in[k_\ell]\setminus\{j\}}\ell_{j'}^{\rot}, &(\ell\in T^*,j\in[k_\ell]).
	\end{align*}
To accomplish this task, we are going to deal due to technical reasons with two cases separately.
\begin{definition}\label{Def_tame}
We call $(\omega,\gamma)$ 
\bemph{tame} if for all $\ell$ and all $j\in[k_\ell]$
the following conditions are satisfied.
\begin{description}
\item[TM1] $\omega^{\y\y}_{\ell,j}=\frac14+O_k(k^{-9})$.
\item[TM2] $\gamma^{\rot\cyan}_{\ell,j},\gamma^{\cyan\rot}_{\ell,j}=(1+O_k(k^{-9}))\bar\gamma^{\rot\cyan}_{\ell,j}$.
\item[TM3] $\gamma^{\y\y}_{\ell,j,j'}=(1+O_k(k^{-9}))\bar\gamma^{\y\y}_{\ell,j,j'}$.
\item[TM4] $\gamma^{\rot\rot}_{\ell,j}=(1+O_k(k^{-9}))\bar\gamma^{\rot\rot}_{\ell,j}$.
\end{description}
Otherwise, we call $(\omega,\gamma)$ \bemph{wild}.
\end{definition}

As a next step,
we estimate the expected number of pairs $(\xi_1,\xi_2)$ of good $\theta$-shades with a given overlap.
This task is of a similar nature as the derivation of the formula for the first moment in \Sec~\ref{Sec_TheFirstMoment}.

\subsection{The expected number of pairs with a given overlap}\label{Sec_f}
Let $\cZ(\omega,\gamma)$ be the number of pairs $(\xi_1,\xi_2)$ of $\theta$-shades with
$\omega(\xi_1,\xi_2)=\omega$ and $\gamma(\xi_1,\xi_2)=\gamma$.
Assuming that $\omega,\gamma$ are such that $\Erw_\cT[\cZ(\omega,\gamma)]>0$,
we aim to derive an asymptotic formula for $\frac1{n}\ln\Erw_\cT[\cZ(\omega,\gamma)]$.
More specifically, the aim in the following is to identify an explicit function $F(\omega,\gamma)$ such that
	$\Erw_\cT[\cZ(\omega,\gamma)]=O(\exp(nF(\omega,\gamma))).$
To this end, we follow the program that we used in \Sec~\ref{Sec_TheFirstMoment} to derive such a formula for the first moment, although the details are more involved.

\subsubsection{The entropy}
Let  $\hat\Xi(\omega)$ be the set of all pairs $(\zeta_1,\zeta_2)$
	such that $\zeta_1,\zeta_2:L'\ra\spins$ are maps that satisfy 
		$\zeta_1(\neg l)=\neg\zeta_1(l),\zeta_2(\neg l)=\neg\zeta_2(l)$ for all $l\in L'$ and such that
	 $|\zeta_1^{-1}(z_1)\cap\zeta_2^{-1}(z_2)\cap L_t'|=\omega_t^{z_1z_2}n_t$ for all $t\in T$, $ z_1,z_2\in\spins$.
Let
	\begin{align*}
	\Fent(\omega)&=\frac1{n}\ln\abs{\hat\Xi(\omega)}.
	\end{align*}
We have the following basic estimate of $\Fent(\omega)$.
Recall that $H(\,\cdot\,)$ denotes the entropy and that $\brk T=\cbc{\{t,\neg t\}:t\in T}$.

\begin{lemma}
For $\omega\in\Om$ let
	$$\fent(\omega)=\sum_{t\in T}\pi_t H(\omega_t^{z_1z_2})_{z_1,z_2\in\spins}.$$
Then $\Fent(\omega)=\fent(\omega)+o(1)$.
In fact, if 	$(\omega,\gamma)$ is tame, then  $\Fent(\omega)=\fent(\omega)-4\abs{\brk T}\ln n/n+O(1/n).$
\end{lemma}
\begin{proof}
Since $(\omega_t^{z_1z_2})_{z_1,z_2\in\spins}$ is a probability distribution,
the assertion follows from Fact~\ref{Fact_binomialLargeDev} (cf.\ the proof of Lemma~\ref{Lemma_theEntropy}).
\end{proof}

\subsubsection{The discrepancy}
To proceed, fix 	$(\zeta_1,\zeta_2)\in\hat\Xi(\omega)$.
Let 
	$$\omega_t^{\purpur\purpur}=\sum_{z_1,z_2\in\{1,*\}}\omega_t^{z_1z_2},\quad
		\omega_t^{\purpur\y}=\sum_{z\in\{1,*\}}\omega_t^{z0},\quad
		\omega_t^{\y\purpur}=\sum_{z\in\{1,*\}}\omega_t^{0z},\quad
			\omega_t^{\y\y}=\omega_t^{00}.$$
Further, let
	\begin{align*}
	\Fdisc(\omega)&=\frac1{n}\ln\pr_\cT\brk{\forall \ell\in T^*,j\in[k_\ell],z_1,z_2\in\{\purpur,\y\}:
		\abs{\cbc{i\in M_\ell:\zeta_1(\hat\PHI_{ij})=z_1,\zeta_2(\hat\PHI_{ij})=z_2}}=\omega_{\ell,j}^{z_1z_2}m_\ell},
	\end{align*}
i.e., the probability that for all clause types $\ell$ and all $j\in[k_\ell$] the distribution of the $(\purpur,\purpur),(\purpur,\y),(\y,\purpur),(\y,\y)$-dominos
over the clauses of type $\ell$ is as prescribed by $(\omega_{\ell,j}^{z_1z_2})_{z_1,z_2\in\{\purpur,\y\}}$.

\begin{lemma}
For $\omega\in\Om$ let
	\begin{align*}
		\fdisc(\omega)&=-\sum_{t\in T}\sum_{h\in[d_t]}\sum_{(\ell,j)\in\partial(t,h)}\frac{m_\ell}{n}
		\KL{\omega_{\ell,j}^{\purpur\purpur},\omega_{\ell,j}^{\purpur\y},\omega_{\ell,j}^{\y\purpur},\omega_{\ell,j}^{\y\y}}
			{\omega_{t}^{\purpur\purpur},
				\omega_{t}^{\purpur\y},\omega_{t}^{\y\purpur},
				\omega_{t}^{\y\y}}.
\end{align*}
Then $\Fdisc(\omega)=\fdisc(\omega)+o(1)$.
In fact, if $(\omega,\gamma)$ is tame, then
	$$\Fdisc(\omega)=\fdisc(\omega)-\sum_{t\in T}\sum_{h\in[d_t]}\frac{3(|\partial(t,h)|-1)\ln n}{2n}+O(1/n).$$
\end{lemma}
\begin{proof}
Once more, this is immediate from Fact~\ref{Fact_binomialLargeDev}.
\end{proof}

\subsubsection{The validity probability}\label{Sec_fval}
Fix a clause type  $\ell\in T^*$.
Let $\cX_\ell(\omega_\ell)$ be the set of all vectors
$(X_{\ell,j}(i,\omega_\ell))_{j\in[k_\ell],i\in[m_\ell]}$ with entries in $\cbc{\purpur,\y}$ such that
	$$\abs{\cbc{i\in[m_\ell]:X_{\ell,j}(i,\omega_\ell)=(z_1,z_2)}}\doteq\omega_{\ell,j}^{z_1z_2}m_\ell\qquad\mbox{for all }j\in[k_\ell],z_1,z_2\in\cbc{\purpur,\y}.$$
Further, let $\vec X_\ell(\omega_\ell)$ be a uniformly random element of $\cX_\ell(\omega_\ell)$.
For a given vector $X_\ell(\omega_\ell)\in\cX_\ell(\omega_\ell)$ let
	$G_{\ell,j}^{\rot\rot}(X_\ell(\omega_\ell))$ be set of indices $i\in[m_\ell]$ such that the ``domino sequence''
		$(X_{\ell,1}(i,\omega_\ell),\ldots,X_{\ell,k_\ell}(i,\omega_\ell))$ satisfies the condition for being a $(\rot,\rot,j)$-clause.
Define $G_{\ell,j}^{\rot\cyan}$ etc.\ analogously.
Further, let $\vec G_{\ell,j}^{\rot\rot}(\omega_\ell)=|G_{\ell,j}^{\rot\rot}(\vec X_\ell(\omega_\ell))|/m_\ell$ etc.\ and set
	\begin{align*}
	\Fvall(\omega_\ell,\gamma_\ell)&=\frac1{n}\ln\pr\brk{\vec G_{\ell}=\gamma_{\ell}}\quad\mbox{for }\ell\in T^*,&
	\Fval(\omega,\gamma)&	=\sum_{\ell\in T^*}\frac{m_\ell}{n}\Fvall(\omega_\ell,\gamma_\ell).
	\end{align*}

\begin{figure}
	\begin{align*}
	e_{\ell,j}^{\purpur\purpur}&=
	\gamma_{\ell,j}^{\rot\rot}+\gamma_{\ell,j}^{\rot\cyan}+\gamma_{\ell,j}^{\cyan\rot}+
	\frac{\gamma_\ell^{\cyan\cyan}q_{\ell,j}^{\purpur\purpur}}{g_\ell^{\cyan\cyan}}
		\brk{1-\prod_{j'\neq j}q^{\nix\y}_{\ell, j'}-\prod_{j'\neq j}q^{\y\nix}_{\ell, j'}+\prod_{j'\neq j}q^{\y\y}_{\ell,j'}},\\
e_{\ell,j}^{\purpur\y}&=
	\gamma_{\ell,j}^{\rot\y}+
	\frac{\gamma_\ell^{\cyan\cyan}q_{\ell,j}^{\purpur\y}}{g_{\ell}^{\cyan\cyan}}
		\bigg[1-\prod_{j'\neq j}q^{\nix\y}_{\ell, j'}-\prod_{j'\neq j}q^{\y\nix}_{\ell, j'}
			-\sum_{j'\neq j}q^{\nix\purpur}_{\ell, j'}\prod_{j''\neq j,j'}q^{\nix\y}_{\ell, j''}+\prod_{j'\neq j}q^{\y\y}_{\ell, j'}
					+\sum_{j'\neq j}q^{\y\purpur}_{\ell, j'}\prod_{j''\neq j,j'}q^{\y\y}_{\ell, j''}
					\bigg]\\
		&\qquad\qquad\qquad+\sum_{j'\neq j}\bigg[\frac{\gamma_{\ell, j'}^{\cyan\rot}q_{\ell,j}^{\purpur \y}}{g_{\ell,j'}^{\cyan\rot}}
				\prod_{j''\neq j,j'}q_{\ell,j''}^{\nix\y}
	+\frac{\gamma_{\ell, j'}^{\y\rot}q_{\ell,j}^{\purpur \y}}{g_{\ell,j'}^{\y\rot}}\brk{\prod_{j''\neq j'}q_{\ell,j''}^{\nix\y}-\prod_{j''\neq j'}q_{\ell,j''}^{\y\y}}
		+{\gamma_{\ell,j,j'}^{\y\y}}\bigg],\\
e_{\ell,j}^{\y\purpur}&=
	\gamma_{\ell,j}^{\y\rot}+
	\frac{\gamma_\ell^{\cyan\cyan}q_{\ell,j}^{\y\purpur}}{g_{\ell}^{\cyan\cyan}}
		\bigg[1-\prod_{j'\neq j}q^{\y\nix}_{\ell, j'}-\prod_{j'\neq j}q^{\nix\y}_{\ell, j'}
			-\sum_{j'\neq j}q^{\purpur\nix}_{\ell, j'}\prod_{j''\neq j,j'}q^{\y\nix}_{\ell, j''}+\prod_{j'\neq j}q^{\y\y}_{\ell, j'}
					+\sum_{j'\neq j}q^{\purpur\y}_{\ell, j'}\prod_{j''\neq j,j'}q^{\y\y}_{\ell, j''}
					\bigg]\\
		&\qquad\qquad\qquad+\sum_{j'\neq j}\bigg[\frac{\gamma_{\ell, j'}^{\rot\cyan}q_{\ell,j}^{ \y\purpur}}{g_{\ell,j'}^{\rot\cyan}}
				\prod_{j''\neq j,j'}q_{\ell,j''}^{\y\nix}
	+\frac{\gamma_{\ell, j'}^{\rot\y}q_{\ell,j}^{\y\purpur }}{g_{\ell,j'}^{\rot\y}}\brk{\prod_{j''\neq j'}q_{\ell,j''}^{\y\nix}-\prod_{j''\neq j'}q_{\ell,j''}^{\y\y}}
		{+\gamma_{\ell,j',j}^{\y\y}}\bigg].
	\end{align*}
\caption{The the vector $e_{\ell}$.}\label{Fig_ell}
\end{figure}

\begin{figure}
\begin{align*}
g_{\ell,j}^{\rot\rot}(q_\ell)&=q_{\ell,j}^{\purpur\purpur}\prod_{j'\neq j}q_{\ell,j'}^{\y\y},\qquad
	g_{\ell,j,j'}^{\y\y}(q_\ell)=q_{\ell,j}^{\purpur\y}q_{\ell,j'}^{\y\purpur}\prod_{j''\neq j,j'}q_{\ell,j'}^{\y\y},\\
g_{\ell,j}^{\rot\cyan}(q_\ell)&=q_{\ell,j}^{\purpur\purpur}
		\brk{\prod_{j'\neq j}q_{\ell,j'}^{\y \nix}-\prod_{j'\neq j}q_{\ell,j'}^{\y\y}},\qquad
	g_{\ell,j}^{\cyan\rot}(q_\ell)=q_{\ell,j}^{\purpur\purpur}
		\brk{\prod_{j'\neq j}q_{\ell,j'}^{ \nix\y}-\prod_{j'\neq j}q_{\ell,j'}^{\y\y}}\\
g_{\ell,j}^{\rot\y}(q_\ell)&=q_{\ell,j}^{\purpur\y}
	\brk{\prod_{j'\neq j}q_{\ell,j'}^{\y \nix}-\prod_{j'\neq j}q_{\ell,j'}^{\y\y}
		-\sum_{j'\neq j}q^{\y\purpur}_{\ell,j'}\prod_{j''\neq j,j'}q^{\y\y}_{\ell,j''}},\\
g_{\ell,j}^{\y\rot}(q_\ell)&=q_{\ell,j}^{\y\purpur}
				\brk{\prod_{j'\neq j}q_{\ell,j'}^{\nix\y }-\prod_{j'\neq j}q_{\ell,j'}^{\y\y}
					-\sum_{j'\neq j}q^{\purpur\y}_{\ell,j'}\prod_{j''\neq j,j'}q^{\y\y}_{\ell,j''}},\\
g_{\ell}^{\cyan\cyan}(q_\ell)&=
	1-\prod_{j\in[k_\ell]}q^{\y\nix}_{\ell,j}-\sum_{j\in[k_\ell]}q^{\purpur\nix}_{\ell,j}\prod_{j'\neq j}q^{\y\nix}_{\ell,j'}
					-\prod_{j\in[k_\ell]}q^{\nix\y}_{\ell,j}-\sum_{j\in[k_\ell]}q^{\nix\purpur}_{\ell,j}\prod_{j'\neq j}q^{\nix\y}_{\ell,j'}\\
			&\qquad\qquad\qquad\qquad	+\prod_{j\in[k_\ell]}q_{\ell,j}^{\y\y}+\sum_{j\in[k_\ell]}(1-q_{\ell,j}^{\y\y})\prod_{j'\neq j}q_{\ell,j'}^{\y\y}
				+\sum_{j_1\neq j_2}q_{\ell,j_1}^{\purpur\y}q_{\ell,j_2}^{\y\purpur}\prod_{j\neq j_1,j_2}q_{\ell,j}^{\y\y}.
	\end{align*}
\caption{The vector $g_\ell(q_\ell)$. }\label{Fig_gell}
\end{figure}

\begin{lemma}\label{Lemma_smValid}
Let $\ell\in\cT^*$ and let $q_\ell=(q_{\ell,j}^{z_1z_2})_{j\in[k_\ell],z_1,z_2\in\cbc{\purpur,\y}}$ be a vector with entries in $[0,1]$
such that	
	$$\sum_{z_1,z_2\in\{\purpur,\y\}}q_{\ell,j}^{z_1z_2}=1\quad\mbox{for all $j\in[k_\ell]$.}$$
Assume that with the expressions from Figure~\ref{Fig_ell} we have
	\begin{equation}\label{eqShiftVal}
	e_{\ell,j}^{\purpur\purpur}=\omega_{\ell,j}^{\purpur\purpur},
		e_{\ell,j}^{\purpur\y}=\omega_{\ell,j}^{\purpur\y},e_{\ell,j}^{\y\purpur}=\omega_{\ell,j}^{\y\purpur}\quad\mbox{for all $j\in[k_\ell]$.}
	\end{equation}
With $g_\ell=g_\ell(q_\ell)$ from Figure~\ref{Fig_gell}, let
	\begin{align*}
	\fvall(\omega_\ell,\gamma_\ell,q_\ell)&=
		-\KL{\gamma_\ell}{g_\ell}+
			\sum_{j \in [k_\ell]}\KL{
			\omega_{\ell,j}^{\purpur\purpur},\omega_{\ell,j}^{\purpur\y},\omega_{\ell,j}^{\y\purpur},\omega_{\ell,j}^{\y\y}}
			{q_{\ell,j}^{\purpur\purpur},q_{\ell,j}^{\purpur\y},q_{\ell,j}^{\y\purpur},q_{\ell,j}^{\y\y}}. 
	\end{align*}
Then $\Fvall(\omega_\ell,\gamma_\ell)=\fvall(\omega_\ell,\gamma_\ell,q_\ell)+o(1)$. 
Indeed, if  $(\omega,\gamma)$ is tame, then
	\begin{align*}
	\Fvall(\omega_\ell,\gamma_\ell)&=\fvall(\omega_\ell,\gamma_\ell,q_\ell)- \frac{(\bink{k_\ell}2+5k_\ell)\ln n}{{2n}}+O(1/n).
	\end{align*}
\end{lemma}
To prove \Lem~\ref{Lemma_smValid}, we consider a random vector $\vec\chi_\ell=(\chi_{\ell,j}(i))_{j\in[k_\ell],i\in[m_\ell]}$ 
whose entries $\chi_{\ell,j}(i)$ are independent random variables with values in $\cbc{(\purpur,\purpur),(\purpur,\y),(\y,\purpur),(\y,\y)}$ such that
	$$\pr[\chi_{\ell,j}(i)=(z_1,z_2)]=
		q_{\ell,j}^{z_1z_2}\qquad(j\in[k_\ell],i\in[m_\ell],z_1,z_2\in\{\purpur,\y\}).$$
Let $S_\ell$ be the event that $G_\ell(\vec\chi_\ell)=\gamma_\ell$.
Furthermore,  for $j\in[k_\ell]$ and $z_1,z_2\in\cbc{\purpur,\y}$ let
	\begin{equation}\label{eqFunnyB1}
	b_{\ell,j}^{z_1z_2}=\abs{\cbc{i\in[m_\ell]:\chi_{\ell,j}(i)=(z_1,z_2)}}.
	\end{equation}
Moreover, let $B_\ell$ be the event that $b_{\ell,j}^{z_1z_2}=\omega_{\ell,j}^{z_1z_2}m_\ell$  for all $j\in[k_\ell]$ and all $z_1,z_2\in\cbc{\purpur,\y}$. 
Given that $B_\ell$ occurs, $\vec\chi_\ell$ has the same distribution as the random vector $\vec X_\ell$.
Therefore,
	\begin{equation}\label{FactSB}
		\pr_\cT\brk{\vec G_\ell=\gamma_\ell}=\pr\brk{S_\ell|B_\ell}.
	\end{equation}
As in the previous instances where we used a similar approach, it turns out that $\pr\brk{S_\ell}$ and $\pr\brk{B_\ell}$ are easy to compute
due to the independence of the entries of $\vec\chi_\ell$.

\begin{claim}\label{Claim_smmVal}
We have
	$\frac1{m_\ell}\ln\pr\brk{S_\ell}=-\KL{\gamma_\ell}{g_\ell}+O(\ln n/n).$
Moreover, if $(\omega,\gamma)$ is tame, then 
	$$\frac1{m_\ell}\ln\pr\brk{S_\ell}=-\KL{\gamma_\ell}{g_\ell}-\frac{(\bink{k_\ell}2+5k_\ell)\ln n}{2m_\ell}+O(1/n).$$
\end{claim}
\begin{proof}
Because the entries of $\vec\chi_\ell$ are independent,
the entries of $g_\ell$ are  the probabilities that the sequence $(\vec\chi_{\ell,j}(i))_{j \in [k_\ell]}$
satisfies the various conditions from Definition~\ref{Def_clauseValidity}.
Thus, the assertion follows from Fact~\ref{Fact_binomialLargeDev}.
\end{proof}

\begin{claim}\label{Lemma_prBSimple}
We have
	$$\frac1{m_\ell}\ln\pr\brk{B_\ell}=-\sum_{j\in[k_\ell]}
		\KL{\omega_{\ell,j}^{\purpur\purpur},\omega_{\ell,j}^{\purpur\y},\omega_{\ell,j}^{\y\purpur},\omega_{\ell,j}^{\y\y}}
			{q_{\ell,j}^{\purpur\purpur},q_{\ell,j}^{\purpur\y},q_{\ell,j}^{\y\purpur},q_{\ell,j}^{\y\y}}+O(\ln n/n).$$
Moreover, if $(\omega,\gamma)$ is tame, then
	$$\frac1{m_\ell}\ln\pr\brk{B_\ell}=-\frac{3k_\ell\ln n}{2m_\ell}-\sum_{j\in[k_\ell]}
		\KL{\omega_{\ell,j}^{\purpur\purpur},\omega_{\ell,j}^{\purpur\y},\omega_{\ell,j}^{\y\purpur},\omega_{\ell,j}^{\y\y}}
			{q_{\ell,j}^{\purpur\purpur},q_{\ell,j}^{\purpur\y},q_{\ell,j}^{\y\purpur},q_{\ell,j}^{\y\y}}+O(1/n).$$
\end{claim}
\begin{proof}
This follows from Fact~\ref{Fact_binomialLargeDev} and the independence of the entries of $\vec\chi_\ell$.
\end{proof}

\begin{claim}\label{Lemma_prB}
For any $j\in[k_\ell]$ we have
	$\Erw[b_{\ell,j}^{\purpur\purpur}|S_\ell]= e_{\ell,j}^{\purpur\purpur}m_\ell,\quad
		\Erw[b_{\ell,j}^{\purpur\y}|S_\ell]= e_{\ell,j}^{\purpur\y}m_\ell,\quad
		\Erw[b_{\ell,j}^{\y\purpur}|S_\ell]= e_{\ell,j}^{\y\purpur}m_\ell.$
\end{claim}
\begin{proof}
Once more, this is immediate from the independence of the entries of $\vec\chi_\ell$.
\end{proof}

\begin{claim}\label{Cor_prB}
We have $\pr[B_\ell|S_\ell]=\exp(o(n))$.
Moreover, if $(\omega,\gamma)$ is tame, then	$$\frac1n\ln\pr[B_\ell|S_\ell]=-\frac{3k_\ell\ln n}{2n}+O(1/n).$$
\end{claim}
\begin{proof}
Given that $S_\ell$ occurs, the random variables $b_{\ell,j}^{z_1z_2}$ are sums of $\Theta(m_\ell)$ independent contributions.
Furthermore, by Claim~\ref{Lemma_prB} the expectation of each $b_{\ell,j}^{z_1z_2}$ is precisely the
value $\omega_{\ell,j}^{z_1z_2}m_\ell$ required by the event $B_\ell$.
Thus, the assertion follows from \Thm~\ref{Lemma_LLT}.
\end{proof}

\begin{proof}[Proof of \Lem~\ref{Lemma_smValid}]
\Lem~\ref{Lemma_smValid} is now immediate from Claims~\ref{Claim_smmVal}--\ref{Cor_prB} and Bayes' formula.
\end{proof}

\begin{lemma}\label{Lemma_implicit2}
Let $\ell\in T^*$ and assume that $|\omega_{\ell,j}^{z_1z_2}-1/4|\leq k^{-4}$ for all $z_1,z_2\in\{\purpur,\y\}$, $j\in[k_\ell]$.
Then there exists a unique $q_\ell=q_\ell(\omega_\ell,\gamma_\ell)$ such that~(\ref{eqShiftVal}) holds and 
	$|q_{\ell,j}^{z_1z_2}-\omega_{\ell,j}^{z_1z_2}|=O_k(2^{-k})$ for all $j\in[k_\ell],z_1,z_2\in\{\purpur,\y\}$.
Further,
	\begin{align}\label{eqLemma_implicit2_1}
	\frac{\partial q_{\ell,j}^{z_1z_2}}{\partial \omega_{\ell,j}^{z_1z_2}}&=1+\tilde O_k(2^{-k}),&
		\frac{\partial q_{\ell,j}^{z_1z_2}}{\partial \omega_{\ell,j'}^{z_1'z_2'}}&=\tilde O_k(2^{-k})\quad\mbox{ if }(j,z_1,z_2)\neq(j',z_1',z_2')
	\end{align}
and for all $j',j''\in[k_\ell]$, $(z_1,z_2)\in\{(\rot,\rot),(\rot,\cyan),(\rot,\y),(\cyan,\rot),(\y,\rot)\}$ we have
	\begin{align}\label{eqLemma_implicit2_2}
	\frac{\partial q_{\ell,j}^{z_1z_2}}{\partial \gamma_{\ell,j'}^{y_1y_2}},\frac{\partial q_{\ell,j}^{z_1z_2}}{\partial \gamma_{\ell,j',j''}^{\y\y}}
		=\tilde O_k(1).
	\end{align}
In addition,
${\partial^2q_{\ell,j}^{z_1z_2}}/{\partial x\partial y}=\tilde O_k(1)$ for all $x,y$ and
	\begin{equation}\label{eqSecondDeriv}
	\frac{\partial^2q_{\ell,j}^{z_1z_2}}{\partial\omega_{\ell,j'}^{z_1'z_2'}\partial\omega_{\ell,j''}^{z_1''z_2''}}
			=\tilde O_k(2^{-k})
			\qquad\mbox{ for all $j,j',j''\in[k_\ell]$ and all $z_1,z_1',z_1'',z_2,z_2',z_2''\in\cbc{\purpur,\y}$}.
	\end{equation}
\end{lemma} 
\begin{proof}
Let
	\begin{eqnarray*}
	e_\ell=(e_{\ell,j}^{\purpur\purpur}-\omega_{\ell,j}^{\purpur\purpur},e_{\ell,j}^{\purpur\y}-\omega_{\ell,j}^{\purpur\y},
		e_{\ell,j}^{\y\purpur}-\omega_{\ell,j}^{\y\purpur},e_{\ell,j}^{\y\y}-\omega_{\ell,j}^{\y\y},
		\gamma_{\ell,j}^{\rot\rot},\gamma_{\ell,j}^{\rot\cyan},\gamma_{\ell,j}^{\rot\y},\gamma_{\ell,j}^{\cyan\rot},\gamma_{\ell,j}^{\y\rot},
		\gamma^{\y\y}_{\ell,j,j'},\omega_{\ell,j}^{\purpur\purpur},\omega_{\ell,j}^{\purpur\y},\omega_{\ell,j}^{\y\purpur},\omega_{\ell,j}^{\y\y})_{j\neq j'}.
	\end{eqnarray*}
Then a solution $q_\ell$ to the equation
	\begin{align*}
	(q_{\ell},\gamma_{\ell,j}^{\rot\rot},\gamma_{\ell,j}^{\rot\cyan},\gamma_{\ell,j}^{\rot\y},\gamma_{\ell,j}^{\cyan\rot},\gamma_{\ell,j}^{\y\rot},
		\gamma^{\y\y}_{\ell,j,j'},\omega_{\ell,j}^{\purpur\purpur},\omega_{\ell,j}^{\purpur\y},\omega_{\ell,j}^{\y\purpur},\omega_{\ell,j}^{\y\y})\\
		&\hspace{-6cm}\,=e_\ell^{-1}(0,\ldots,0,
		\gamma_{\ell,j}^{\rot\rot},\gamma_{\ell,j}^{\rot\cyan},\gamma_{\ell,j}^{\rot\y},\gamma_{\ell,j}^{\cyan\rot},\gamma_{\ell,j}^{\y\rot},
		\gamma^{\y\y}_{\ell,j,j'},\omega_{\ell,j}^{\purpur\purpur},\omega_{\ell,j}^{\purpur\y},\omega_{\ell,j}^{\y\purpur},\omega_{\ell,j}^{\y\y}).
	\end{align*}
satisfies~(\ref{eqShiftVal}).
If we order the variables as 
	$q^{\purpur\purpur}_{\ell,j},
		q^{\purpur\y}_{\ell,j},q^{\y\purpur}_{\ell,j},\gamma_{\ell,j}^{\rot\rot},\gamma_{\ell,j}^{\rot\cyan},\gamma_{\ell,j}^{\cyan\rot},\gamma^{\y\y}_{\ell,j,j'},
			\omega_{\ell,j}^{\purpur\purpur},\omega_{\ell,j}^{\purpur\y},\omega_{\ell,j}^{\y\purpur},\omega_{\ell,j}^{\y\y}$, 
we find
	\begin{eqnarray*}
	De_\ell&=&\brk{\begin{array}{cc}
		D_1&D_2\\
		0&\id,
		\end{array}},
	\end{eqnarray*}
where $D_1=\id-D_3$ with $D_3$ a matrix with all entries $\tilde O_k(2^{-k})$. All entries of $D_2$ are $\tilde O_k(1)$.
Hence, 
	\begin{eqnarray}\label{eqmyCramer}
	(De_\ell)^{-1}&=&
		\brk{\begin{array}{cc}
		D_1^{-1}&-D_1^{-1}D_2\\
		0&\id
		\end{array}},\quad\mbox{and }\quad D_1^{-1}=\id+\sum_{\nu=1}^\infty D_3^\nu{=2\id-D_1+\sum_{\nu\geq2}D_3^\nu}.
	\end{eqnarray}
Therefore, (\ref{eqLemma_implicit2_1}) and~(\ref{eqLemma_implicit2_2}) follow from the inverse function theorem.
Further, a straightforward calculation yields
	\begin{equation}\label{eqSecondDeriv_pf}
	\frac{\partial^2 e_{\ell,j}^{z_1z_2}}{\partial q_{\ell,j'}^{z_1'z_2'}\partial q_{\ell,j''}^{z_1''z_2''}}=O_k(2^{-k})
		\quad
		\mbox{for all $j,j',j''\in[k_\ell],z_1,z_2,z_1',z_2',z_1'',z_2''\in\cbc{\purpur,\y}$, $(z_1,z_2)\neq(\y,\y)$}.
	\end{equation}
{Finally, 
combining (\ref{eqmyCramer}), (\ref{eqSecondDeriv_pf}) and applying the chain rule, we find
	\begin{align*}
	\frac{\partial^2q_{\ell,j}^{z_1z_2}}{\partial\omega_{\ell,j'}^{z_1'z_2'}\partial\omega_{\ell,j''}^{z_1''z_2''}}
		&=\frac{\partial}{\partial\omega_{\ell,j'}^{z_1'z_2'}}(D_1^{-1})_{(j,z_1,z_2),(j',z_1',z_2')}
		=-\frac{\partial}{\partial\omega_{\ell,j'}^{z_1'z_2'}}\frac{\partial e_{\ell,j}^{z_1z_2}}{\partial q_{\ell,j''}^{z_1'',z_2''}}
			+\sum_{\nu\geq2}	\frac{\partial}{\partial\omega_{\ell,j'}^{z_1'z_2'}}(D_3^\nu)_{(j,z_1,z_2),(j'',z_1'',z_2'')}\\
		&=-\sum_{j''',z_1''',z_2'''}\frac{\partial q_{\ell,j'''}^{z_1''',z_2'''}}{\partial\omega_{\ell,j'}^{z_1'z_2'}}
			\frac{\partial^2 e_{\ell,j}^{z_1z_2}}{\partial q_{\ell,j''}^{z_1'',z_2''}\partial q_{\ell,j'''}^{z_1''',z_2'''}}+
				\tilde O_k(2^{-k})
		=\tilde O_k(2^{-k}),
	\end{align*}
whence (\ref{eqSecondDeriv}) follows.}
\end{proof}

To deal with wild overlaps, it will be convenient to have a rough upper bound on $\Fvall(\omega_\ell,\gamma_\ell)$
without having to solve for $q_{\ell}$.
The following lemma provides such an upper bound.

\begin{lemma}\label{Lemma_smValid_rough}
For any $\ell\in T^*$ and any $(\omega,\gamma)$ we have
	$\Fvall(\omega_\ell,\gamma_\ell)\leq	-\KL{\gamma_\ell}{g_\ell(\omega_\ell)}+o(1)$.
\end{lemma}
\begin{proof}
Consider a random vector
	$\vec\eta_\ell=(\eta_{\ell,j}(i))_{j\in[k_\ell],i\in[m_\ell]}$
whose entries are independent with distribution
	\begin{equation}\label{FactSBeta0}
	\pr[\eta_{\ell,j}(i)=(z_1,z_2)]=
		\omega_{\ell,j}^{z_1z_2}\qquad(j\in[k_\ell],i\in[m_\ell],z_1,z_2\in\{\purpur,\y\}).
		\end{equation}
Let $S_\ell,B_\ell$ be as above.
Then
	\begin{equation}\label{FactSBeta}
	\pr_\cT\brk{\vec G_\ell=\gamma_\ell}=\pr_{\vec\eta_\ell}\brk{S_\ell|B_\ell}\leq \pr_{\vec\eta_\ell}\brk{S_\ell}/\pr_{\vec\eta_\ell}\brk{B_\ell}.
	\end{equation}
Furthermore, 
	\begin{equation}\label{FactSBeta1}
	\frac1{m_\ell}\ln\pr\brk{\vec\eta_\ell\in S_\ell}=	-\KL{\gamma_\ell}{g_\ell(\omega_\ell)}+o(1).
	\end{equation}
In addition, (\ref{FactSBeta0}) ensures that
	\begin{equation}\label{FactSBeta2}
	\Erw\abs{\cbc{i\in[m_\ell]:\eta_{\ell,j}(i)=(z_1,z_2)}}=\omega_{\ell,j}^{z_1z_2}m_\ell\quad\mbox{ for any $j\in[k_\ell]$, $z_1,z_2\in\{\purpur,\y\}$}.
	\end{equation}
Because the entries $\eta_{\ell,j}(i)$ are independent, (\ref{FactSBeta2}) and \Thm~\ref{Lemma_LLT} imply that
	$\pr \brk{B_\ell}=\exp(o(n))$.
Thus, the assertion follows from~(\ref{FactSBeta}) and~(\ref{FactSBeta1}).
\end{proof}

\subsubsection{The occupancy problem}\label{Sec_focc}

Fix two maps  $(\zeta_1,\zeta_2)\in\hat\Xi(\omega)$.
For a type $t\in T$ let $\cX_t(\omega)$ be the set of all vectors $X_t=(X_{t,h}(l))_{l\in L_t,h\in[d_t]}$
	with entries $X_{t,h}(l)\in\cbc{\rot,\cyan,\y}\times\cbc{\rot,\cyan,\y}$ that satisfy the following conditions.
	\begin{description}
	\item[OCC1] For each $h\in[d_t]$ and any $z_1,z_2\in\cbc{\rot,\cyan,\y}$ we have
			$\abs{\cbc{l\in L_t':X_{t,h}(l)=(z_1,z_2)}}=\omega_{t,h}^{z_1z_2}n_t$.
	\item[OCC2] Let $l\in L_t$.
		If $\zeta_1(l)=0$, then $X_{t,h}(l)\in\cbc\y\times\cbc{\rot,\cyan,\y}$ for all $h\in[d_t]$.
		Similarly, if $\zeta_2(l)=0$, then $X_{t,h}(l)\in\cbc{\rot,\cyan,\y}\times\cbc\y$ for all $h\in[d_t]$.
	\item[OCC3] Let $l\in L_t$.	
		If 	$\zeta_1(l)\neq0$, then $X_{t,h}(l)\in\cbc{\rot,\cyan}\times\cbc{\rot,\cyan,\y}$ for all $h\in[d_t]$.
		Moreover, if 	$\zeta_2(l)\neq0$, then $X_{t,h}(l)\in\cbc{\rot,\cyan,\y}\times\cbc{\rot,\cyan}$ for all $h\in[d_t]$.
	\end{description}
Let $\vec X_t$ be a uniformly random element of $\cX_t(\omega)$.
We are interested in the event that, in addition to {\bf OCC1--OCC3}, $\vec X_t$ also satisfies the following.
	\begin{description}
	\item[OCC4] If $l\in L_t$ is such that $\zeta_1(l)=*$, then $X_{t,h}(l)\in\cbc\cyan\times\cbc{\rot,\cyan,\y}$ for all $h\in[d_t]$.
		Moreover, if $l\in L_t$ is such that $\zeta_1(l)=*$, then $X_{t,h}(l)\in\cbc{\rot,\cyan,\y}\times\cbc\cyan$ for all $h\in[d_t]$.
	\item[OCC5] If $l\in L_t$ is such that $\zeta_1(l)=1$, then there exists $h\in[d_t]$ such that $X_{t,h}(l)\in\cbc\rot\times\cbc{\rot,\cyan,\y}$.
		Analogously, if $l\in L_t$ is such that $\zeta_2(l)=1$, then there exists $h\in[d_t]$ such that $X_{t,h}(l)\in\cbc{\rot,\cyan,\y}\times\cbc\rot$.
	\end{description}
Let 
	$$\Foct(\omega_t)=\frac1{n_t}\ln\pr\brk{\vec X_t\mbox{ satisfies {\bf OCC4--OCC5}}}
	\quad \text{and} \quad \Focc(\omega)=\sum_{t\in T}\pi_t\Foct(\omega_t).$$
Wer will show the following.
\begin{figure}
\begin{align*}
	s_t^{11}&=1-\prod_{h\in[d_t]}\bc{1-q_{t,h}^{\rot\rot}-q_{t,h}^{\rot\cyan}}
		-\prod_{h\in[d_t]}\bc{1-q_{t,h}^{\rot\rot}-q_{t,h}^{\cyan\rot}}
		+\prod_{h\in[d_t]}(1-q_{t,h}^{\rot\rot}-q_{t,h}^{\rot\cyan}-q_{t,h}^{\cyan\rot}),\\
	s_t^{1*}&=\prod_{h\in[d_t]}(1-q_{t,h}^{\rot\rot}-q_{t,h}^{\cyan\rot})-\prod_{h\in[d_t]}(1-q_{t,h}^{\rot\rot}-q_{t,h}^{\rot\cyan}-q_{t,h}^{\cyan\rot}),\\
		s_t^{*1}&=\prod_{h\in[d_t]}\bc{1-q_{t,h}^{\rot\rot}-q_{t,h}^{\rot\cyan}}-\prod_{h\in[d_t]}(1-q_{t,h}^{\rot\rot}-q_{t,h}^{\rot\cyan}-q_{t,h}^{\cyan\rot}),\qquad
		s_t^{**}=\prod_{h\in[d_t]}(1-q_{t,h}^{\rot\rot}-q_{t,h}^{\rot\cyan}-q_{t,h}^{\cyan\rot}),\\
	s_t^{10}&=1-\prod_{h\in[d_t]}(1-q_{t,h}^{\rot\y}),\qquad
		s_t^{*0}=\prod_{h\in[d_t]}(1-q_{t,h}^{\rot\y}),\qquad
		s_t^{01}=1-\prod_{h\in[d_t]}(1-q_{t,h}^{\y\rot}),\qquad
		s_t^{0*}=\prod_{h\in[d_t]}(1-q_{t,h}^{\y\rot}).
\end{align*}
\caption{The expression for \Lem~\ref{Prop_f}.}\label{Fig_occs}
\end{figure}

\begin{figure}
\begin{align*}
	e_{t,h}^{\rot\rot}&=\frac{\omega_t^{11}{q_{t,h}^{\rot\rot}}}{s_t^{11}},\qquad
	e_{t,h}^{\rot\cyan}=\frac{\omega_t^{11} q_{t,h}^{\rot\cyan}}{s_t^{11}}
			\brk{1-\prod_{h'\neq h}\bc{1-q_{t,h'}^{\rot\rot}-q_{t,h'}^{\cyan\rot}}}
			+\frac{\omega_t^{1*} q_{t,h}^{\rot\cyan}}{s_t^{1*}}\prod_{h'\neq h}(1-q_{t,h'}^{\rot\rot}-q_{t,h'}^{\cyan\rot}),\\
	e_{t,h}^{\cyan\rot}&=\frac{\omega_t^{11} q_{t,h}^{\cyan\rot}}{s_t^{11}}
			\brk{1-\prod_{h'\neq h}\bc{1-q_{t,h'}^{\rot\rot}-q_{t,h'}^{\rot\cyan}}}
			+\frac{\omega_t^{*1} q_{t,h}^{\cyan\rot}}{s_t^{*1}}\prod_{h'\neq h}(1-q_{t,h'}^{\rot\rot}-q_{t,h'}^{\rot\cyan}),\quad
	e_{t,h}^{\rot\y}=\frac{\omega_t^{10}q_{t,h}^{\rot\y}}{s_t^{10}},\quad
		e_{t,h}^{\y\rot}=\frac{\omega_t^{10} q_{t,h}^{\y\rot}}{s_t^{10}}.
\end{align*}
\caption{The expressions for \Lem~\ref{Prop_f}.}\label{Fig_occe}
\end{figure}

\begin{lemma}\label{Prop_f}
Let $t\in T$.
Assume that for any $h\in[d_t]$ there exist
 $q_{t,h}^{\rot\rot},q_{t,h}^{\rot\cyan},q_{t,h}^{\rot\y},q_{t,h}^{\cyan\rot},q_{t,h}^{\y\rot},q_{t,h}^{\cyan\cyan}\in[0,1]$ such that
 $q_{t,h}^{\cyan\cyan}=1-q_{t,h}^{\rot\rot}-q_{t,h}^{\rot\cyan}-q_{t,h}^{\cyan\rot}$ and such that
 with the expressions from Figure~\ref{Fig_occe} we have
	\begin{equation}\label{eqProp_f_shift}
	e_{t,h}^{z_1z_2}=\omega_{t,h}^{z_1z_2}\qquad\mbox{for all }(z_1,z_2)\in\cbc{(\rot,\rot),(\rot,\cyan),(\rot,\y),(\cyan,\rot),(\y,\rot)}.
	\end{equation}
With the expressions from Figure~\ref{Fig_occs}, let
	\begin{align*}
	\focct(\omega_t,q_t)&=\hspace{-3mm}\sum_{(z_1,z_2)\in\cbc{0,1,*}^2\setminus\cbc{(0,0)}}\hspace{-5mm}\omega_t^{z_1z_2}\ln s_t^{z_1z_2}
						+\sum_{h\in[d_t]}\omega_{t,h}^{\purpur\y}\KL{\frac{\omega_{t,h}^{\rot\y}}{\omega_{t,h}^{\purpur\y}}}{q_{t,h}^{\rot\y}}
						+\omega_{t,h}^{\y\purpur}\KL{\frac{\omega_{t,h}^{\y\rot}}{\omega_{t,h}^{\y\purpur}}}{q_{t,h}^{\y\rot}}\\
		&\qquad\qquad\qquad\quad		\qquad\qquad\qquad\ \ \,
			+\omega_{t,h}^{\purpur\purpur}\KL{\frac{\omega_{t,h}^{\rot\rot}}{\omega_{t,h}^{\purpur\purpur}},
					\frac{\omega_{t,h}^{\rot\cyan}}{\omega_{t,h}^{\purpur\purpur}},
						\frac{\omega_{t,h}^{\cyan\rot}}{\omega_{t,h}^{\purpur\purpur}}}
							{q_{t,h}^{\rot\rot},q_{t,h}^{\rot\cyan},q_{t,h}^{\cyan\rot}}.
	\end{align*}
Then $\Foct(\omega_t)=\focct(\omega_t,q_t)+o(1)$.
In fact, if $(\omega,\gamma)$ is tame, then  $\Foct(\omega_t)=\focct(\omega_t,q_t)+O(1/n).$
\end{lemma}

To prove \Lem~\ref{Prop_f} we introduce an auxiliary probability space.
Namely, let $\vec\chi_t=(\vec\chi_{t,h}(l))_{l\in L_t,h\in[d_t]}$ be a random vector with mutually independent
	entries $\vec\chi_{t,h}(l)\in\cbc{\rot,\cyan,\y}\times\cbc{\rot,\cyan,\y}$ that are distributed as follows.
\begin{itemize}
\item If $\zeta_1(l),\zeta_2(l)\in\{1,*\}$,
		then $\pr\brk{\vec\chi_{t,h}(l)=(z_1,z_2)}=q_{t,h}^{z_1z_2}$ for all $h\in[d_t]$ and $z_1,z_2\in\{\rot,\cyan\}$.
\item If $\zeta_1(l)\in\{1,*\},\zeta_2(l)=0$,
		then $\pr\brk{\vec\chi_{t,h}(l)=(z,\y)}=q_{t,h}^{z\y}$ for all $h\in[d_t]$ and $z\in\{\rot,\cyan\}$.
\item If $\zeta_1(l)=0,\zeta_2(l)\in\{1,*\}$,
		then $\pr\brk{\vec\chi_{t,h}(l)=(\y,z)}=q_{t,h}^{\y z}$ for all $h\in[d_t]$ and $z\in\{\rot,\cyan\}$.
\item If $\zeta_1(l)=\zeta_2(l)=0,\zeta_2(l)=\in\{1,*\}$,
		then $\vec\chi_{t,h}(l)=(\y,\y)$ with certainty.
\end{itemize}
Let $S_t$ be the event that the following four conditions hold.
	\begin{enumerate}[(i)]
	\item If $\zeta_1(l)=1$, then there exists $h\in[d_t]$ such that $\vec\chi_{t,h}(l)\in\cbc\rot\times\{\rot,\cyan,\y\}$.
	\item If $\zeta_2(l)=1$, then there exists $h\in[d_t]$ such that $\vec\chi_{t,h}(l)\in\{\rot,\cyan,\y\}\times\cbc\rot$.
	\item If $\zeta_1(l)=*$, then $\vec\chi_{t,h}(l)\in\cbc\cyan\times\{\rot,\cyan,\y\}$ for all $h\in[d_t]$.
	\item If $\zeta_2(l)=*$, then $\vec\chi_{t,h}(l)\in\{\rot,\cyan,\y\}\times\cbc\cyan$ for all $h\in[d_t]$.
	\end{enumerate}

Further, for $z_1,z_2\in\cbc{\rot,\cyan,\y}$ and $h\in[d_t]$ define
	$b_{t,h}^{z_1z_2}=\abs{\cbc{l\in L_t':\vec\chi_{t,h}(l)=(z_1,z_2)}}$.
Let $B_t$ be the event that for all $z\in\cbc{\cyan,\y}$ we have
	$b_{t,h}^{\rot\rot}=\omega_{t,h}^{\rot\rot},\ 
		b_{t,h}^{\rot z}=\omega_{t,h}^{\rot z},\ 
		b_{t,h}^{z\rot}=\omega_{t,h}^{\rot z}.$
Then
	\begin{equation}\label{eqfequivalence}
	\pr\brk{\vec X_t\mbox{ satisfies {\bf OCC4--OCC5}}}=\pr\brk{S_t|B_t}.
	\end{equation}

\begin{claim}\label{Claim_f1}
We have $\frac1{n_t}\ln\pr\brk{S_t}=\sum_{(z_1,z_2)\in\cbc{0,1,*}^2\setminus\cbc{(0,0)}}\omega_t^{z_1z_2}\ln s_t^{z_1z_2}$. 
\end{claim}
\begin{proof}
This is immediate from the independence of the entries of $\vec X_t$.
\end{proof}

\begin{claim}\label{Claim_f2}
We have $-\frac{\ln\pr\brk{B_t}}{n_t}=\Delta+O(\ln n/n)$, where
	\begin{align*}
	\Delta=	\sum_{h\in[d_t]}\omega_{t,h}^{\purpur\y}\KL{\frac{\omega_{t,h}^{\rot\y}}{\omega_{t,h}^{\purpur\y}}}{q_{t,h}^{\rot\y}}
						+\omega_{t,h}^{\y\purpur}\KL{\frac{\omega_{t,h}^{\y\rot}}{\omega_{t,h}^{\y\purpur}}}{q_{t,h}^{\y\rot}}
			+\omega_{t,h}^{\purpur\purpur}\KL{\frac{\omega_{t,h}^{\rot\rot}}{\omega_{t,h}^{\purpur\purpur}},
					\frac{\omega_{t,h}^{\rot\cyan}}{\omega_{t,h}^{\purpur\purpur}},
						\frac{\omega_{t,h}^{\cyan\rot}}{\omega_{t,h}^{\purpur\purpur}}}
							{q_{t,h}^{\rot\rot},q_{t,h}^{\rot\cyan},q_{t,h}^{\cyan\rot}
								}.
	\end{align*}
In fact, if $(\omega,\gamma)$ is tame, then $-\frac{\ln\pr\brk{B_t}}{n_t}=\Delta-\frac{5d_t}{2n}\ln n+O(1/n)$.
\end{claim}
\begin{proof}
Once more, this is immediate from the independence of the entries of $\vec X_t$ and Fact~\ref{Fact_binomialLargeDev}.
\end{proof}

\begin{claim}\label{Claim_f3}
We have $\pr\brk{B_t|S_t}=\exp(o(n))$.
In fact, if $(\omega,\gamma)$ is tame, then $\frac1{n_t}\ln\pr\brk{B_t|S_t}=-\frac{5d_t}{2n}\ln n+O(1/n)$.
\end{claim}
\begin{proof}
For any $h\in[d_t]$, $(z_1,z_2)\in\{(\rot,\rot),(\rot,\cyan),(\rot,\y),(\cyan,\rot),(\y,\rot)\}$  we have
	$\Erw[b_{t,h}^{z_1z_2}|S_t]=e_{t,h}^{z_1z_2}n_t$.
Moreover, being sums of independent contributions, the vectors $(b_{t,h}^{z_1z_2})_{z_1,z_2}$ satisfy the assumtions
of \Thm~\ref{Lemma_LLT}, whence the assertion follows.
\end{proof}

\begin{proof}[Proof of \Lem~\ref{Prop_f}]
The assertion is immediate from~(\ref{eqfequivalence}) and Claims~\ref{Claim_f1}--\ref{Claim_f3}.
\end{proof}

We conclude this section by showing that under certain conditions the equation~(\ref{eqProp_f_shift})
has a solution.

\begin{lemma}\label{Lemma_occImplicit}
Let $\omega\in\Om$, $t\in T$ and assume that there are no more than $d_t/k^4$ indices
$h\in[d_t]$ such that
	$$\max\{|\omega_{t,h}^{\rot\cyan}-\bar\omega_{t,h}^{\rot\cyan}|,|\omega_{t,h}^{\cyan\rot}-\bar\omega_{t,h}^{\cyan\rot}|,
		|\omega_{t,h}^{\rot\y}-\bar\omega_{t,h}^{\rot\y}|,|\omega_{t,h}^{\y\rot}-\bar\omega_{t,h}^{\y\rot}|\}>k^{-5}2^{-k}.$$
Further, assume that $|\omega_t^{z_1z_2}-\frac14|\leq1/k$ for all $z_1,z_2\in\{0,1\}$.
Then there exists a unique vector $q_t=q_t(\omega_t)$ such that~(\ref{eqProp_f_shift}) is satisfied and
	\begin{align}\label{eqLemma_occImplicit1}
	q_{t,i}^{\rot\rot}&=(1+\tilde O_k(2^{-k}))\frac{\omega_{t,i}^{\rot\rot}}{\omega_{t}^{11}},&
		q_{t,i}^{\rot\cyan}&=(1+\tilde O_k(2^{-k}))\frac{\omega_{t,i}^{\rot\cyan}}{\omega_{t}^{11}},&
		q_{t,i}^{\cyan\rot}&=(1+\tilde O_k(2^{-k}))\frac{\omega_{t,i}^{\cyan\rot}}{\omega_{t}^{11}},\\
	q_{t,i}^{\rot\y}&=(1+\tilde O_k(2^{-k}))\frac{\omega_{t,i}^{\rot\y}}{\omega_{t}^{10}},&
		q_{t,i}^{\y\rot}&=(1+\tilde O_k(2^{-k}))\frac{\omega_{t,i}^{\y\rot}}{\omega_{t}^{01}}.
			\label{eqLemma_occImplicit2}
	\end{align}
Moreover,
	\begin{align*}
	\frac{\partial q_{t,i}^{\rot\rot}}{\partial\omega_{t,i}^{\rot\rot}},
	\frac{\partial q_{t,i}^{\rot\cyan}}{\partial\omega_{t,i}^{\rot\cyan}},
	\frac{\partial q_{t,i}^{\cyan\rot}}{\partial\omega_{t,i}^{\cyan\rot}}&=\frac1{\omega_t^{11}}+\tilde O_k(2^{-k}),\qquad
	\frac{\partial q_{t,i}^{\rot\y}}{\partial\omega_{t,i}^{\rot\y}}=\frac1{\omega_t^{10}}+\tilde O_k(2^{-k}),&
	\frac{\partial q_{t,i}^{\y\rot}}{\partial\omega_{t,i}^{\y\rot}}&=\frac1{\omega_t^{01}}+\tilde O_k(2^{-k}),\\
	\frac{\partial q_{t,i}^{z_1z_2}}{\partial\omega_{t,i'}^{z_1'z_2'}}&=\tilde O_k(4^{-k})\quad\mbox{if }(i,z_1,z_2)\neq(i',z_1',z_2'),&
	\frac{\partial q_{t,i}^{z_1z_2}}{\partial\omega_{t}^{y_1y_2}}&=\tilde O_k(2^{-k})\quad\mbox{for all }y_1,y_2\in\spins.
	\end{align*}
In addition, if $(\omega,\gamma)$ is tame, then
	\begin{align*}
		\frac{\partial^2 q_{t,h}^{z_1z_2}}{\partial\omega_{t,h'}^{z_1'z_2'}\partial\omega_{t,h''}^{z_1''z_2''}}&=
			\begin{cases}
			\tilde O_k(2^{-k})&\mbox{ if }(h,z_1,z_2)\in\{(h',z_1',z_2'),(h'',z_1'',z_2'')\},\\
			\tilde O_k(4^{-k})&\mbox{ otherwise},
			\end{cases}
			\\
				\frac{\partial^2 q_{t,h}^{\rot z}}{\partial\omega_t^{z_1z_2}\partial\omega_t^{z_1'z_2'}},
					\frac{\partial^2 q_{t,h}^{z\rot}}{\partial\omega_t^{z_1z_2}\partial\omega_t^{z_1'z_2'}}&=\tilde O_k(2^{-k})\mbox{ for }z_1,z_1',z_2,z_2'	
							\in\spins,z\in\{\rot,\cyan,\y\},\\
		\frac{\partial^2 q_{t,h}^{z_1z_2}}{\partial\omega_{t,h'}^{z_1'z_2'}\partial\omega_{t}^{z_1''z_2''}}&=\begin{cases}
			\tilde O_k(1)&\mbox{ if }h=h',\\
			\tilde O_k(2^{-k})&\mbox{ otherwise}.
			\end{cases}
	\end{align*}
\end{lemma}
\begin{proof}
Consider $q_t=(q_{t,h}^{\rot\rot},q_{t,h}^{\rot\cyan},q_{t,h}^{\cyan\rot},q_{t,h}^{\rot\y},q_{t,h}^{\y\rot})_{h\in[d_t]}$ such
that $0\leq q_{t,h}^{z_1z_2}\leq O_k(2^{-k})$ for all $h$, $z_1,z_2$ and such that for no more than $d_t/k^4$ indices $h\in[d_t]$ we have
	$\max\{|q_{t,h}^{\rot\cyan}-2^{-k}|,|q_{t,h}^{\cyan\rot}-2^{-k}|,
		|q_{t,h}^{\rot\y}-2^{-k}|,|q_{t,h}^{\y\rot}-2^{-k}|\}>k^{-5}2^{-k}.$
A straightforward and tedious calculation reveals that
	\begin{align*}
	\frac{\partial e_{t,i}^{\rot\rot}}{\partial q_{t,i}^{\rot\rot}},
	\frac{\partial e_{t,i}^{\rot\cyan}}{\partial q_{t,i}^{\rot\cyan}},
	\frac{\partial e_{t,i}^{\cyan\rot}}{\partial q_{t,i}^{\cyan\rot}}&=\frac1{\omega_t^{11}}+\tilde O_k(2^{-k}),&
	\frac{\partial e_{t,i}^{\rot\y}}{\partial q_{t,i}^{\rot\y}}&=\frac1{\omega_t^{10}}+\tilde O_k(2^{-k}),&
	\frac{\partial e_{t,i}^{\y\rot}}{\partial q_{t,i}^{\y\rot}}&=\frac1{\omega_t^{01}}+\tilde O_k(2^{-k}),\\
	\frac{\partial e_{t,i}^{z_1z_2}}{\partial q_{t,i'}^{z_1'z_2'}}&=\tilde O_k(4^{-k})&&\mbox{if }(i,z_1,z_2)\neq(i',z_1',z_2'),\\
	\frac{\partial e_{t,i}^{z_1z_2}}{\partial \omega_{t}^{y_1y_2}}&=\tilde O_k(2^{-k})&&\mbox{for all }y_1,y_2\in\spins.
	\end{align*}
Hence, the inverse function theorem yields the existence of a unique $q_{t}$ 
that satisfies~(\ref{eqProp_f_shift}) and~(\ref{eqLemma_occImplicit1})--(\ref{eqLemma_occImplicit2})
along with the
bounds on the first partial derivatives of $q_{t,i}^{z_1z_2}$.
Finally, the bounds on the second derivatives follow by calculating the second differentials of $e_{t,i}^{z_1z_2}$ and  using the chain rule.
\end{proof}

\subsubsection{Putting things together}
Letting
	$$F(\omega,\gamma)=\Fent(\omega)+\Fdisc(\omega)+\Fval(\omega,\gamma)+\Focc(\omega),$$
we finally arrive at the following statement.

\begin{fact}\label{Fact_F}
For any $(\omega,\gamma)$ we have $\Erw_\cT[\cZ(\omega,\gamma)]\leq \exp(nF(\omega,\gamma)+o(n))$.
Moreover, if $(\omega,\gamma)$ is tame, then $\Erw[\cZ(\omega,\gamma)|\cT]\leq O(\exp(nF(\omega,\gamma)))$.
\end{fact}

\noindent
In the following two sections we are going to estimate $F(\omega,\gamma)$.
In \Sec~\ref{Sec_tame} we deal with the case that $(\omega,\gamma)$ is tame.
Then, in \Sec~\ref{Sec_wild} we will deal with wild $(\omega,\gamma)$
and complete the proof of \Prop~\ref{Prop_secondMoment}.

\subsection{Tame overlaps}\label{Sec_tame} 
In this section we estimate the contribution of tame $(\omega,\gamma)$ to the second moment.

\begin{lemma}\label{Lemma_tame}
Let $\Omega'$ be the set of all tame $(\omega,\gamma)$.
Then
	$\sum_{(\omega,\gamma)\in\Omega'}\exp(nF(\omega,\gamma))\leq O(\Erw_\cT[\cZ]^2).$
\end{lemma}

To prove \Lem~\ref{Lemma_tame} we approximate $F(\omega,\gamma)$ by means of the functions
$\fent,\fdisc,\fval,\focc$ from \Sec~\ref{Sec_f}.
Indeed, assume that $(\omega,\gamma)$ is tame. 
Then  \Lem s~\ref{Lemma_implicit2} and~\ref{Lemma_occImplicit} provide canonical 
vectors $q_t,q_\ell$ for $t\in T$, $\ell\in T^*$.
For the sake of brevity, we write $\fvall(\omega_\ell,\gamma_\ell)=\fvall(\omega,\gamma,q_\ell)$,
$\focct(\omega_t)=\focct(\omega_t,q_t)$ and
	$$\fval(\omega,\gamma)=\sum_{\ell\in T^*}\frac{m_\ell}n\fvall(\omega,\gamma,q_\ell),
		\qquad\focc(\omega)=\sum_{t\in T}\pi_t\focct(\omega_t,q_t).$$
Let
	$$f(\omega,\gamma)=\fent(\omega)+\fdisc(\omega)+\fval(\omega,\gamma)+\focc(\omega).$$
The lemmas from the previous section show that
	$F(\omega,\gamma)=f(\omega,\gamma)+o(1)$.
Thus, we need to study $f$.
We are going to show that on the set of tame overlaps, $f$ is strictly concave with its maximum attained at $(\bar\omega,\bar\gamma)$.
Throughout, it is understood that we take differentials within the polytope defined by the affine relations from Fact~\ref{Fact_affine}.

\subsubsection{The first derivative}
Here we calculate the first derivative of the function $f$ to prove 

\begin{lemma}\label{Lemma_Df}
We have $Df(\bar\omega,\bar\gamma)=0$.
\end{lemma}

\noindent
Indeed, we are going to show that $D\fent(\bar\omega),D\fdisc(\bar\omega),D\fval(\bar\omega,\bar\gamma),D\focc(\bar\omega,\bar\gamma)=0$.

\begin{claim}\label{Claim_Dfent}
We have $D\fent(\bar\omega)=D\fdisc(\bar\omega)=0$.
\end{claim}
\begin{proof}
Each component of $\bar\omega$ is a product measure.
Indeed,  for any $t\in T$, $z_1,z_2\in\spins$ we have $\bar\omega_t^{z_1z_2}=t^{z_1}t^{z_2}$.
Therefore,  subject to the relations from Fact~\ref{Fact_affine}, $(\bar\omega_t^{z_1z_2})_{z_1,z_2\in\spins}$ is the maximizer of the entropy term $\fent$.
Hence, $D\fent(\bar\omega)=0$.
In addition,  since for any $\ell\in T^*,j\in[k_\ell]$, $z_1,z_2\in\{\purpur,\y\}$ we have $\omega_{\ell,j}^{z_1z_2}=\ell_j^{z_1}\ell_j^{z_2}$,
we see that $\fdisc(\bar\omega)=0$.
Since $0$ is the global maximum of $\fdisc$, we conclude that $D\fdisc(\bar\omega)=0$. 
\end{proof}

\begin{claim}\label{Claim_Dfval}
We have $D\fval(\bar\omega,\bar\gamma)=0$.
\end{claim}
\begin{proof}
We are going to show that $D\fvall(\bar\omega_\ell,\bar\gamma_\ell)=0$ for all $\ell\in T^*$.
While we could directly calculate $D\fvall(\bar\omega_\ell,\bar\gamma_\ell)$, it is more elegant to argue by way of the combinatorial interpretation of $\fvall$.
Thus, let $\cX_\ell(\omega_\ell)$ be as in \Sec~\ref{Sec_fval}.
Furthermore, again with the notation from \Sec~\ref{Sec_fval}, let $\cS_\ell(\omega_\ell,\gamma_\ell)$ be the set of all $X_\ell\in\cX_\ell(\omega)$ such that
	$G_\ell(X_\ell)=\gamma_\ell$.
Then by \Lem~\ref{Lemma_smValid} for tame $(\omega,\gamma)$ we have
	\begin{equation}\label{eqClaim_Dfval0}
	\fvall(\omega_\ell,\gamma_\ell)=\Fvall(\omega_\ell,\gamma_\ell)+o(1)=\frac1{m_\ell}\ln\frac{|\cS_\ell(\omega_\ell,\gamma_\ell)|}{|\cX_\ell(\omega_\ell)|}+o(1)
		=\frac{\ln|\cS_\ell(\omega_\ell,\gamma_\ell)|-\ln|\cX_\ell(\omega_\ell)|}{m_\ell}+o(1).
	\end{equation}
By Fact~\ref{Fact_entropyFunction},
	\begin{equation}\label{eqClaim_Dfval1}
	\frac1{m_\ell}\ln|\cX_\ell(\omega_\ell)|
		=			\sum_{j\in[k_\ell]}H(\omega_{\ell,j})+o(1).
	\end{equation}
Let  $\tilde\cX_\ell$ be the set of all maps $\chi_\ell:[m_\ell]\times[k_\ell]\ra\cbc{\purpur,\y}$, $(i,j)\mapsto\chi_{\ell,j}(i)$ such that 
for any $j\in[k_\ell]$ we have $\abs{\cbc{i\in[m_\ell]:\chi_{\ell,j}(i)=\y}}\doteq\ell_j^\y m_\ell$.
Then $\cX_\ell(\omega_\ell)\subset\tilde\cX_\ell\times\tilde\cX_\ell$ for all $\omega_\ell$.
In effect,
	\begin{equation}\label{eqClaim_Dfval1a}
	\frac1{m_\ell}\ln|\cX_\ell(\omega_\ell)|
		\leq
			\frac2{m_\ell}\ln|\tilde\cX_\ell|
		=\sum_{j\in[k_\ell]}H(\bar\omega_{\ell,j}^\y)+o(1).
	\end{equation}
Analogously, let $\tilde\cS_\ell$ be the set of $\chi_\ell\in\tilde\cX_\ell$ such that
for any $i\in[m_\ell]$ there is $j\in[k_\ell]$ such that $\chi_{\ell,j}(i)=\purpur$ and
such that for any $j\in[k_\ell]$ we have
		$\abs{\cbc{i\in[m_\ell]:\chi_{\ell,j}(i)=\purpur\wedge\forall j'\neq j:\chi_{\ell,j'}(i)=\y}}\doteq\ell_j^\rot m_\ell.$
Then $\cS_\ell(\omega_\ell,\gamma_\ell)\subset\tilde\cS_\ell\times\tilde\cS_\ell$ for any $(\omega_\ell,\gamma_\ell)$.
Hence, \Prop~\ref{Prop_firstMomentFormula} and \Lem~\ref{Lemma_firstMomentSell} show that
	$$\frac1{m_\ell}\ln\frac{|\cS_\ell(\omega_\ell,\gamma_\ell)|}{|\tilde\cX_\ell\times\tilde\cX_\ell|}\leq
		\frac2{m_\ell}\ln\frac{|\tilde\cS_\ell|}{|\tilde\cX_\ell|}=2\varphi_\ell+o(1)=\fvall(\bar\omega_\ell,\bar\gamma_\ell)+o(1);$$
to obtain the last equation, we verify that at the point $\bar\omega_\ell,\bar\gamma_\ell$,
the implicit parameters in \Lem~\ref{Lemma_firstMomentSell} and \Lem~\ref{Lemma_smValid} satisfy the relation
$q_{\ell,j}^{z_1z_2}=q_{\ell,j}^{z_1}q_{\ell,j}^{z_2}$ for all $j\in[k_\ell]$, $z_1,z_2\in\{\purpur,\y\}$.
Hence,
	$$\frac1{m_\ell}\ln|\cS_\ell(\omega_\ell,\gamma_\ell)|
		\leq \fvall(\bar\omega_\ell,\bar\gamma_\ell)+\sum_{j\in[k_\ell]}H(\bar\omega_{\ell,j}^\y)+o(1).$$
Combining~(\ref{eqClaim_Dfval1}) and~(\ref{eqClaim_Dfval1a}), 
we see that $\omega_\ell\mapsto\frac1{m_\ell}\ln|\cX_\ell(\omega_\ell)|=\sum_{j\in[k_\ell]}H(\omega_{\ell,j}^\y)+o(1)$ attains
its global maximum at a point $\hat\omega_\ell$ such that $\norm{\hat\omega_\ell-\bar\omega_\ell}_\infty=o(1)$.
Analogously, there is $(\tilde\omega_\ell,\tilde\gamma_\ell)$ such that
$\norm{\tilde\omega_\ell-\bar\omega_\ell}_\infty,\norm{\tilde\gamma_\ell-\bar\gamma_\ell}_\infty=o(1)$ where
$(\omega_\ell,\gamma_\ell)\mapsto\frac1{m_\ell}\ln|\cS_\ell(\omega_\ell,\gamma_\ell)|=\fvall(\bar\omega_\ell,\bar\gamma_\ell)+\sum_{j\in[k_\ell]}H(\bar\omega_{\ell,j}^\y)+o(1)$
attains its maximum.
Because their difference $\fvall$ has continuous derivatives, (\ref{eqClaim_Dfval0}) implies $D\fvall(\bar\omega_\ell,\bar\gamma_\ell)=0$.
\end{proof}

\begin{claim}\label{Claim_Dfocc}
We have $D\focc(\bar\omega)=0$.
\end{claim}
\begin{proof}
We are going to show that $D\focct(\bar\omega_\ell)=0$ for all $t\in T$.
Once more we use a combinatorial argument.
We use the notion from \Sec~\ref{Sec_focc}.
Let $\cS_t(\omega_t)$ be the set of all $X_t\in\cX_t(\omega)$ that satisfy {\bf OCC4--OCC5}.
Then for tame $(\omega,\gamma)$ we have
	\begin{equation}\label{eqClaim_Dfocc0}
	\focct(\omega_t)=\Foct(\omega_t)+o(1)=\frac1{n_t}\ln\frac{|\cS_t(\omega_t)|}{|\cX_t(\omega_t)|}+o(1)
		=\frac{\ln|\cS_t(\omega_t)|-\ln|\cX_t(\omega_t)|}{n_t}+o(1).
	\end{equation}
As in the proof of Claim~\ref{Claim_Dfval}, by considering the entropy we see that the maximizer
$\hat\omega_t$ of $|\cX_t(\nix)|$ satisfies $\hat\omega_t\doteq\bar\omega_t$.
Similarly, if $\tilde\omega_t$ is such that $|\cS_t(\nix)|$ is maximum, then $\tilde\omega_t\doteq\bar\omega_t$.
Hence, (\ref{eqClaim_Dfocc0}) implies that $D\focct(\bar\omega_t)=0$.
\end{proof}

\noindent
Finally, \Lem~\ref{Lemma_Df} is immediate from Claims~\ref{Claim_Dfent}--\ref{Claim_Dfocc}.

\subsubsection{The second derivative}
In this section we establish the following statement about the second derivative of $f$.

\begin{lemma}\label{Prop_DDF}
There is a number $\beta=\Omega(1)$ such that
for all tame $(\omega,\gamma)$ we have $D^2f(\omega,\gamma)\preceq-\beta\id$.
\end{lemma}

\noindent
{\bf\em In the rest of this section we tacitly assume that $(\omega,\gamma)$ is tame.}
As a first step we estimate the second derivative of $\fent$,
which is a function of $(\omega_t^{z_1z_2})_{t\in T,z_1,z_2\in\spins}$.

\begin{lemma}\label{Lemma_D2ent} 
We have
	$D^2\fent\preceq-\cJ,$
where $\cJ$ is a diagonal matrix with entries
	\begin{align*}
	\cJ_{\omega_{t}^{z_1z_2}\omega_{t}^{z_1z_2}}&=\begin{cases}
			\Omega_k(1)&\mbox{ if }z_1,z_2\in\cbc{0,1},\\
			\Omega_k(2^k)&\mbox{ if $z_1=*$ or $z_2=*$}.
			\end{cases}
	\end{align*}
\end{lemma}
\begin{proof}
The second derivative of the generic summand of the entropy function is
	$\frac{\partial^2}{\partial p^2}p\ln p=1/p$.
Furthermore, together with the affine relations from Fact~\ref{Fact_affine}, the assumption that $(\omega,\gamma)$ is tame implies that
$\omega_t^{z_1z_2}=\frac14+o_k(1)$ if $z_1,z_2\in\{0,1\}$ and
$\omega_t^{z_1z_2}\leq O_k(2^{-k})$ if $z_1=*$ or $z_2=*$.
\end{proof}

\begin{lemma}\label{Lemma_D2disc} 
We have $D^2\fdisc(\omega,\gamma)\preceq0$
\end{lemma}
\begin{proof}
This is immediate from the fact that the Kullback-Leibler divergence is convex.
\end{proof}

As a next step we estimate the second derivative of $\fvall$ for any $\ell\in T^*$.
We can view $\fvall$ as a function of 
	$\gamma_{\ell,j}^{\rot\rot},\gamma_{\ell,j}^{\rot\cyan},\gamma_{\ell,j}^{\cyan\rot},
			\gamma_{\ell,j,j'}^{\y\y},
			\omega_{\ell,j}^{\purpur\purpur}$ with $j,j'\in[k_\ell]$, $j\neq j'$.
Indeed, let $\cV_\ell$ be the set containing these variables.
Then the variables $\cV_\ell$  determine the remaining components of $\omega_\ell,\gamma_\ell$ via the affine relations from Fact~\ref{Fact_affine}.

\begin{lemma}\label{Lemma_D2val}
Let $\ell\in T^*$.
There is a matrix $\cJ=(\cJ_{xy})_{x,y\in\cV_\ell}$ with diagonal entries
	\begin{equation}\label{eqLemma_D2valJ}
	\cJ_{\gamma_{\ell,j}^{\rot\rot}\gamma_{\ell,j}^{\rot\rot}}=\tilde\Omega_k(4^{k}),
		\cJ_{\gamma_{\ell,j,j'}^{\y\y}\gamma_{\ell,j,j'}^{\y\y}}=\tilde\Omega_k(4^{k}),
		\cJ_{\gamma_{\ell,j}^{\rot\cyan}\gamma_{\ell,j}^{\rot\cyan}},
		\cJ_{\gamma_{\ell,j}^{\cyan\rot}\gamma_{\ell,j}^{\cyan\rot}}=\tilde\Omega_k(2^{k}),
		\cJ_{\omega_{\ell,j}^{\purpur\purpur}\omega_{\ell,j}^{\purpur\purpur}}=O_k(k^{-99}2^{-k})
	\end{equation}
and with all off-diagonal entries equal to $0$ such that for all tame $(\omega,\gamma)$ we have
	$D^2\fvall(\omega_\ell,\gamma_\ell)\preceq-\cJ$.
\end{lemma}

\noindent
To prove \Lem~\ref{Lemma_D2val} we determine $D^2\KL{\gamma_\ell}{g_\ell}$ and $D^2\KL{\omega_\ell}{q_\ell}$ separately.

\begin{claim}\label{Lemma_D2FR}
There is a diagonal matrix $\cJ=(\cJ_{xy})_{x,y\in\cV_\ell}$ with entries as in~(\ref{eqLemma_D2valJ}) such that $-D^2\KL{\gamma_\ell}{g_\ell}\preceq-\cJ.$
\end{claim}
\begin{proof}
Let $\cG$ be the set of variables 
	$g_\ell^{\cyan\cyan},g^{\rot\rot}_{\ell,j},g^{\rot\cyan}_{\ell,j},g^{\rot\y}_{\ell,j},
		g^{\cyan\rot}_{\ell,j},g^{\y\rot}_{\ell,j},g^{\y\y}_{\ell,j,j'},
			\gamma_\ell^{\cyan\cyan},\gamma^{\rot\rot}_{\ell,j},\gamma^{\rot\cyan}_{\ell,j},g^{\rot\y}_{\ell,j},
		\gamma^{\cyan\rot}_{\ell,j},\gamma^{\y\rot}_{\ell,j},\gamma^{\y\y}_{\ell,j,j'}$.
To compute the second derivative with respect to $x,y\in\cV_\ell$, we use the chain rule: 
	\begin{eqnarray}\label{eqLemma_D2FR1}
	-\frac{\partial^2\,\KL{\gamma_\ell}{g_\ell}}{\partial x\partial y}&=&
		-\sum_{G\in\cG}\frac{\partial \KL{\gamma_\ell}{g_\ell}}{\partial G}\frac{\partial^2 G}{\partial x\partial y}-
		\sum_{G,G'\in\cG}\frac{\partial^2 \KL{\gamma}{g}}{\partial G\partial G'}
			\frac{\partial G}{\partial x}\frac{\partial G'}{\partial y}.
	\end{eqnarray}
Letting $M=(M_{xy})$ signify the matrix with entries
	$M_{xy}=-\sum_{G\in\cG}\frac{\partial \KL{\gamma}{g}}{\partial G}\frac{\partial^2 G}{\partial x\partial y}$, we obtain from~(\ref{eqD2psi0})
	\begin{align*}
	M_{xy}&=
		\frac{\gamma^{\cyan\cyan}_{\ell}}{g^{\cyan\cyan}_{\ell}}\frac{\partial^2 g^{\cyan\cyan}_\ell}{\partial x\partial y}+
			\sum_{j=1}^{k_\ell}\frac{\gamma^{\rot\rot}_{\ell,j}}{g^{\rot\rot}_{\ell,j}}\frac{\partial^2 g^{\rot\rot}_{\ell,j}}{\partial x\partial y}
			+\frac{\gamma^{\rot\cyan}_{\ell,j}}{g^{\rot\cyan}_{\ell,j}}\frac{\partial^2 g^{\rot\cyan}_{\ell,j}}{\partial x\partial y}
			+\frac{\gamma^{\rot\y}_{\ell,j}}{g^{\rot\y}_{\ell,j}}\frac{\partial^2 g^{\rot\y}_{\ell,j}}{\partial x\partial y}
			+\frac{\gamma^{\cyan\rot}_{\ell,j}}{g^{\cyan\rot}_{\ell,j}}\frac{\partial^2 g^{\cyan\rot}_{\ell,j}}{\partial x\partial y}
			+\frac{\gamma^{\y\rot}_{\ell,j}}{g^{\y\rot}_{\ell,j}}\frac{\partial^2 g^{\y\rot}_{\ell,j}}{\partial x\partial y}
			+\sum_{j'\neq j}
			\frac{\gamma^{\y\y}_{\ell,j,j'}}{g^{\y\y}_{\ell,j,j'}}\frac{\partial^2 g^{\y\y}_{\ell,j,j'}}{\partial x\partial y}.
	\end{align*}
Since $(\omega,\gamma)$ is tame, we verify that
	\begin{eqnarray}\label{eqBoundsOngQuotients}
	\frac{\gamma^{\rot\rot}_{\ell,j}}{g^{\rot\rot}_{\ell,j}},\frac{\gamma^{\y\y}_{\ell,j,j'}}{g^{\y\y}_{\ell,j,j'}},
	\frac{\gamma^{\rot\cyan}_{\ell,j}}{g^{\rot\cyan}_{\ell,j}},\frac{\gamma^{\cyan\rot}_{\ell,j}}{g^{\cyan\rot}_{\ell,j}},
	\frac{\gamma^{\rot\y}_{\ell,j}}{g^{\rot\y}_{\ell,j}},\frac{\gamma^{\y\rot}_{\ell,j}}{g^{\y\rot}_{\ell,j}},
	\frac{\gamma^{\cyan\cyan}_\ell}{g^{\cyan\cyan}_\ell}=1+o_k(1).
	\end{eqnarray}
Furthermore, 
a direct calculation  reveals that
	\begin{align*}
	\frac{\partial^2 g^{\rot\rot}_{\ell,j}}{{\partial\omega_{\ell,j_1}^{\purpur\purpur}\partial\omega_{\ell,j_2}^{\purpur\purpur}}},\ldots,
		\frac{\partial^2 g^{\cyan\cyan}_{\ell}}{{\partial\omega_{\ell,j_1}^{\purpur\purpur}\partial\omega_{\ell,j_2}^{\purpur\purpur}}}
		&=\tilde O_k(4^{-k}),&
	\frac{\partial^2 g^{\rot\rot}_{\ell,j}}{{\partial x\partial y}},\frac{\partial^2 g^{\rot\cyan}_{\ell,j}}{{\partial x\partial y}},\ldots,
		\frac{\partial^2 g^{\cyan\cyan}_{\ell}}{{\partial x\partial y}}
		&=\tilde O_k(2^{-k})\quad\mbox{for all $x,y$.}
	\end{align*}
Thus, we obtain
	\begin{align}\label{eqMxybound}
	M_{xy}&=\tilde O_k(2^{-k})\quad\mbox{for all }x,y,\mbox{ and in fact }&
	M_{\omega_{\ell,j}^{\purpur\purpur}\,\omega_{\ell,j'}^{\purpur\purpur}}=\tilde O_k(4^{-k})\quad\mbox{for all }j,j'\in[k_\ell].
	\end{align}
Further, let
	$\cM=-\brk{\sum_{g,g'\in\cG}\frac{\partial g}{\partial x}\frac{\partial g'}{\partial y}\frac{\partial^2 \KL{\gamma}{g}}{\partial g\partial g'}}_{x,y}$
denote the second summand in (\ref{eqLemma_D2FR1}).
We find that
	\begin{align}\label{eqD2_g_gamma_1}
	\frac{\partial g}{\partial x}&=\tilde O_k(2^{-k})\qquad\mbox{ for all $x$, }
			g\in\{g_\ell^{\cyan\cyan},g^{\rot\rot}_{\ell,j},g^{\rot\cyan}_{\ell,j},g^{\rot\y}_{\ell,j},
				g^{\cyan\rot}_{\ell,j},g^{\y\rot}_{\ell,j},g^{\y\y}_{\ell,j,j'}:j\neq j'\},&
	\frac{\partial g}{\partial \omega_{\ell,j}^{\purpur\purpur}}=\tilde O_k(4^{-k})\quad\mbox{ for }g\in\cG.
	\end{align}
Moreover, due to our assumption that $(\omega,\gamma)$ is tame and~(\ref{eqD2psi}),
	\begin{align*}
	\frac{\partial^2 \KL{\gamma}{g}}{\partial\gamma_{\ell,j}^{\rot\rot\,2}}=\frac1{\gamma_{\ell,j}^{\rot\rot}}&=\Theta_k(4^k),&
		\frac{\partial^2 \KL{\gamma}{ g}}{\partial\gamma_{\ell,j,j'}^{\y\y\,2}}=\frac1{\gamma_{\ell,j,j'}^{\y\y}}&=\Theta_k(4^k),\\
		\frac{\partial^2 \KL{\gamma}{ g}}{\partial\gamma_{\ell,j}^{\rot\cyan\,2}}=\frac1{\gamma_{\ell,j}^{\rot\cyan}}&=\Theta_k(2^k),&
		\frac{\partial^2 \KL{\gamma}{g}}{\partial\gamma_{\ell,j}^{\cyan\rot\,2}}=\frac1{\gamma_{\ell,j}^{\cyan\rot}}&=\Theta_k(2^k),\\
	\frac{\partial^2 \KL{\gamma}{g}}{\partial g_{\ell,j}^{\rot\rot\,2}},
		\frac{\partial^2 \KL{\gamma}{g}}{\partial \gamma_{\ell,j}^{\rot\rot}\partial g_{\ell,j}^{\rot\rot}}&=O_k(4^k),&
		\frac{\partial^2 \KL{\gamma}{ g}}{\partial g_{\ell,j,j'}^{\y\y\,2}},
			\frac{\partial^2 \KL{\gamma}{ g}}{\partial \gamma_{\ell,j,j'}^{\y\y}\partial g_{\ell,j,j'}^{\y\y}}&=O_k(4^k),\\
	\frac{\partial^2 \KL{\gamma}{ g}}{\partial g_{\ell,j}^{\rot\cyan\,2}},
		\frac{\partial^2 \KL{\gamma}{ g}}{\partial \gamma_{\ell,j}^{\rot\cyan}\partial g_{\ell,j}^{\rot\cyan}}&=O_k(2^k),&
		\frac{\partial^2 \KL{\gamma}{g}}{\partial g_{\ell,j}^{\cyan\rot\,2}},
			\frac{\partial^2 \KL{\gamma}{g}}{\partial\gamma_{\ell,j}^{\cyan\rot}\partial g_{\ell,j}^{\cyan\rot}}&=\Theta_k(2^k).
	\end{align*}
Combining these bounds with~(\ref{eqD2_g_gamma_1}), we see that there is a diagonal matrix $J$ with entries
	$$J_{\gamma_{\ell,j}^{\rot\rot}\gamma_{\ell,j}^{\rot\rot}},
		J_{\gamma_{\ell,j,j'}^{\y\y}\gamma_{\ell,j,j'}^{\y\y}}=\tilde\Omega_k(4^k),\ 
		J_{\gamma_{\ell,j}^{\rot\cyan}\gamma_{\ell,j}^{\rot\cyan}},
			J_{\gamma_{\ell,j}^{\cyan\rot}\gamma_{\ell,j}^{\cyan\rot}}=\tilde\Omega_k(2^k),\
			J_{\omega_{\ell,j}^{\purpur\purpur}\omega_{\ell,j}^{\purpur\purpur}}=\tilde O_k(4^{-k})$$
such that $\cM\preceq-J$.
Together with~(\ref{eqMxybound}), this bound implies the assertion.
\end{proof}

\begin{claim}\label{Lemma_D2FB}
If $(\omega,\gamma)$ is tame, then
	$$D^2\KL{\omega_{\ell,j}^{\purpur\purpur},\omega_{\ell,j}^{\purpur\y},\omega_{\ell,j}^{\y\purpur},\omega_{\ell,j}^{\y\y}}
		{q_{\ell,j}^{\purpur\purpur},q_{\ell,j}^{\purpur\y},q_{\ell,j}^{\y\purpur},q_{\ell,j}^{\y\y}}\preceq J,$$
where $J$ is a diagonal matrix with entries
	$J_{\omega_{\ell,j}^{\purpur\purpur}\omega_{\ell,j}^{\purpur\purpur}}=\tilde O_k(2.1^{-k})\mbox{ and }J_{xx}=\tilde O_k(1.9^k)\mbox{ for all other }x\in\cV_\ell.$
\end{claim}
\begin{proof}
Let $\cQ=\KL{\omega_{\ell,j}^{\purpur\purpur},\omega_{\ell,j}^{\purpur\y},\omega_{\ell,j}^{\y\purpur},\omega_{\ell,j}^{\y\y}}
		{q_{\ell,j}^{\purpur\purpur},q_{\ell,j}^{\purpur\y},q_{\ell,j}^{\y\purpur},q_{\ell,j}^{\y\y}}$ for brevity.
By the chain rule,
	\begin{align}
	D^2\cQ&=T_1+T_2,\qquad\mbox{where}&
	T_1&=\bc{\sum_{y}\frac{\partial\cQ}{\partial y}\frac{\partial^2 y}{\partial x\partial x'}}_{x,x'},&
	T_2&=\bc{\sum_{y,y'}\frac{\partial^2\cQ}{\partial y\partial y'}
		\frac{\partial y}{\partial x}\frac{\partial y'}{\partial x'}}_{x,x'}.\label{eqChainRuleInAction}
	\end{align}
Because \Lem~\ref{Lemma_implicit2} ensures that $|\omega_{\ell,j}^{z_1z_2}-q_{\ell,j}^{z_1z_2}|\leq\tilde O_k(2^{-k})$ for all $z_1,z_2\in\cbc{\purpur,\y}$
and as $q_{\ell,j}^{\y\y}=1-q_{\ell,j}^{\purpur\purpur}-q_{\ell,j}^{\purpur\y}-q_{\ell,j}^{\y\purpur}$,
we see that $|\frac{\partial\cQ}{\partial y}|=\tilde O_k(2^{-k})$ for all $y$.
Hence, (\ref{eqSecondDeriv}) implies that
	$T_1\preceq J'$ for a diagonal matrix $J'$ such that
	$$J'_{\omega_{\ell,j}^{\purpur\purpur}\omega_{\ell,j}^{\purpur\purpur}}=\tilde O_k(2.1^{-k}),J'_{xx}=\tilde O_k(1.9^k)\mbox{ for all other }x\in\cV_\ell.$$
With respect to $T_2$, we obtain from~(\ref{eqD2psi}) and (\ref{eqLemma_implicit2_1})--(\ref{eqLemma_implicit2_2}) that
	\begin{align*}
	\sum_{y,y'}\frac{\partial^2\cQ}{\partial y\partial y'}
		\frac{\partial y}{\partial x}\frac{\partial y'}{\partial x'}=\begin{cases}
			\tilde O_k(4^{-k})&\mbox{ if }x,x'\in\cbc{\omega_{\ell,j}^{\purpur\purpur}:j\in[k_\ell]},\\
			\tilde O_k(2^{-k})&\mbox{ if }x'\in\cbc{\omega_{\ell,j}^{\purpur\purpur}:j\in[k_\ell]},\\
			\tilde O_k(1)&\mbox{ otherwise.}
			\end{cases}
	\end{align*}
Hence, there is a diagonal matrix $J''$ 
with $$J''_{\omega_{\ell,j}^{\purpur\purpur}\omega_{\ell,j}^{\purpur\purpur}}=\tilde O_k(2.1^{-k}),J''_{xx}=\tilde O_k(1.9^k)\mbox{ for all other }x\in\cV_\ell$$
 such that $T_2\preceq J_2$.
Setting $J=J'+J''$ completes the proof.
\end{proof}

\noindent
Finally,  \Lem~\ref{Lemma_D2val} follows from Claims~\ref{Lemma_D2FR}--\ref{Lemma_D2FB}.

Next, we estimate the second derivative of $\focc$.
For any $t\in T$, $\focct$ is a function of $\omega_{t}^{z_1z_2}$ with $z_1,z_2\in\spins$ and
of $\omega_{t,h}^{z_1z_2}$ with $h\in[d_t]$ and $(z_1,z_2)\in\cbc{(\rot,\rot),(\rot,\cyan),(\rot,\y),(\cyan,\rot),(\y,\rot)}$.
Let $\cV_t$ be the set containing these variables.

\begin{lemma}\label{Lemma_D2focc} 
Let $t\in T$ and $h\in[d_t]$.
Then $D^2\focct(\omega)\preceq\cJ$, where $\cJ$ is a diagonal matrix with entries
	\begin{align}\label{eqClaim_D2occProb1}
	\cJ_{\omega_t^{z_1z_2}\omega_t^{z_1z_2}}&=\tilde O_k(2^{-k/64})\qquad\mbox{ for $z_1,z_2\in\spins$},\\
	\cJ_{\omega_{t,h}^{\rot\rot}\omega_{t,h}^{\rot\rot}}&=\tilde O_k(4^{15k/16})\mbox{ and }
	\cJ_{\omega_{t,h}^{z_1z_2}\omega_{t,h}^{z_1z_2}}=\tilde O_k(2^{15k/16})\mbox{ for }(z_1,z_2)\in\cbc{(\rot,\cyan),(\cyan,\rot),(\rot,\y),(\y,\rot)},h\in[d_t].
		\label{eqClaim_D2occProb2}
	\end{align}
\end{lemma}

\noindent
The proof of \Lem~\ref{Lemma_D2focc} consists of several steps.

\begin{claim}\label{Claim_D2occProb}
There is a diagonal matrix $\cJ$ with entries as in~(\ref{eqClaim_D2occProb1})--(\ref{eqClaim_D2occProb2}) such that
	$$D^2\sum_{z_1,z_2}\omega_t^{z_1z_2}\ln s_t^{z_1z_2}\preceq\cJ.$$
\end{claim}
\begin{proof}
Let $\cY=\sum_{z_1,z_2}\omega_t^{z_1z_2}\ln s_t^{z_1z_2}$.
Because the function $(a,b)\in\RRpos\mapsto a\ln b$ is concave, we have
	\begin{equation}\label{eqClaim_D2occProb1.2}
	D^2\cY\preceq\cM=\bc{\sum_{z_1,z_2\in\spins:(z_1,z_2)\neq(0,0)}
		\frac{\partial\cY}{\partial s_t^{z_1z_2}}\frac{\partial^2 s_t^{z_1z_2}}{\partial x\partial y}}_{x,y\in\cV_t}.
	\end{equation}
Further, an elementary calculation based on \Lem~\ref{Lemma_occImplicit} and our assumption that $(\omega,\gamma)$ is tame yields
	\begin{align}\label{eqClaim_D2occProb2.2}
	\frac{\partial\cY}{\partial s_t^{z_1z_2}}\frac{\partial^2 s_t^{z_1z_2}}{\partial x\partial y}&\leq\tilde O_k(2^{-k})\quad\mbox{for all }z_1,z_2,x,y.
	\end{align}
The bound~(\ref{eqClaim_D2occProb2.2}) implies bounds on the Frobenius norms of the four blocks of $\cM$.
Namely, the Frobenius norm of the diagonal block corresponding to the variables $\omega_t^{z_1z_2}$, $z_1,z_2\in\spins$, is $\tilde O_k(2^{-k})$.
Moreover, the Frobenius norm of the diagonal block $\omega_{t,h}^{z_1z_2}$ with $h\in[d_t]$ and $z_1,z_2\in\cbc{\rot,\cyan,\y}$ is $\tilde O_k(1)$.
Finally, the Frobenius norm of the off-diagonal blocks comes to $\tilde O_k(2^{-k/2})$.
Because the Frobenius norm is an upper bound on the spectral norm, these estimates and~(\ref{eqClaim_D2occProb1.2}) yield the assertion.
\end{proof}

\begin{claim}\label{Claim_D2corrTermocc}
There exists a diagonal matrix $\cJ$ such that~(\ref{eqClaim_D2occProb1}) and~(\ref{eqClaim_D2occProb2}) are satisfied
and such that
	$$D^2\sum_{h\in[d_t]}\omega_{t}^{\purpur\y}\KL{\frac{\omega_{t,h}^{\rot\y}}{\omega_{t}^{\purpur\y}}}{q_{t,h}^{\rot\y}}
						+\omega_{t}^{\y\purpur}\KL{\frac{\omega_{t,h}^{\y\rot}}{\omega_{t}^{\y\purpur}}}{q_{t,h}^{\y\rot}}
			+\omega_{t}^{\purpur\purpur}
				\KL{\frac{\omega_{t,h}^{\rot\rot}}{\omega_{t}^{\purpur\purpur}},\frac{\omega_{t,h}^{\rot\cyan}}{\omega_{t}^{\purpur\purpur}},
						\frac{\omega_{t,h}^{\cyan\rot}}{\omega_{t}^{\purpur\purpur}}}
							{q_{t,h}^{\rot\rot},q_{t,h}^{\rot\cyan},q_{t,h}^{\cyan\rot}}
		\preceq \cJ.$$
\end{claim}
\begin{proof}
Let $\cX=\cbc{(\rot,\rot),(\rot,\cyan),(\cyan,\rot)}$. 
For $(z_1,z_2)\in\cX$ let $p_{t,h}^{z_1z_2}=\omega_{t,h}^{z_1z_2}/\omega_t^{\purpur\purpur}$.
Further, let 
	$$\cQ_{t,h}=\KL{(p_{t,h}^{z_1z_2})_{(z_1,z_2)\in\cX}}{(q_{t,h}^{z_1z_2})_{(z_1,z_2)\in\cX}}\quad\mbox{and}\quad
		\cQ_t=\sum_{h\in[d_t]}\cQ_{t,h}.$$
Let $\cA_{t,h}$ be the set of variables $p_{t,h}^{z_1z_2},q_{t,h}^{z_1z_2}$ with  $(z_1,z_2)\in\cX$.
Then by the chain rule,	$D^2\cQ_t=\cM+\cN$, where
	\begin{align*}
	\cM&=\bc{\sum_{h\in[d_t]}\sum_{a\in\cA_{t,h}}\frac{\partial\cQ_{t,h}}{\partial a}\frac{\partial^2a}{\partial x\partial y}}_{x,y\in\cV_t},&
	\cN&=\bc{\sum_{h\in[d_t]}
		\sum_{a,b\in\cA_{t,h}}\frac{\partial^2\cQ_{t,h}}{\partial a\partial b}\frac{\partial a}{\partial x}\frac{\partial b}{\partial y}}_{x,y\in\cV_t}.
	\end{align*}
To bound $\cM$ we consider three cases.
For starters, we note that because $(\omega,\gamma)$ is tame,  \Lem~\ref{Lemma_occImplicit}, the affine relations
	$$\frac{\omega_{t,h}^{\rot\rot}}{\omega_{t}^{\purpur\purpur}}+
		\frac{\omega_{t,h}^{\rot\cyan}}{\omega_{t}^{\purpur\purpur}}+
		\frac{\omega_{t,h}^{\cyan\rot}}{\omega_{t}^{\purpur\purpur}}+
		\frac{\omega_{t,h}^{\cyan\cyan}}{\omega_{t}^{\purpur\purpur}}=1,
		\quad q_{t,h}^{\rot\rot}+q_{t,h}^{\rot\cyan}+q_{t,h}^{\cyan\rot}+q_{t,h}^{\cyan\cyan}=1$$
and (\ref{eqD2psi0}) yield $\partial\cQ_{t,h}/\partial a=\tilde O_k(2^{-k})$ for any $a\in\cA_{t,h}$.
\begin{description}
\item[Case 1: $x,y\in\{\omega_{t,h}^{z_1z_2}:h\in\brk{d_t},(z_1,z_2) \in\cX\}$]
	suppose $x=\omega_{t,h_1}^{z_{11}z_{12}}$, $y=\omega_{t,h_2}^{z_{21}z_{22}}$.
	For summands $h\not\in\{h_1,h_2\}$ 
	\Lem~\ref{Lemma_occImplicit} yields $\partial^2a/\partial x\partial y=\tilde O_k(4^{-k})$,
	whilst $\partial^2a/\partial x\partial y=\tilde O_k(2^{-k})$ if $h\in\{h_1,h_2\}$.
	Hence, $\cM_{xy}=\tilde O_k(4^{-k})$.
\item[Case 2: $x\in\{\omega_t^{z_1z_2}:z_1,z_2\in\spins\}$, $y\in\{\omega_{t,h}^{z_1z_2}:h\in\brk{d_t},(z_1,z_2) \in\cX\}$]
	suppose that $x=\omega_t^{z_{11}z_{12}}$, $y=\omega_{t,h_2}^{z_{21}z_{22}}$.
	For $a\in\{p_{t,h}^{z_1z_2}:z_1,z_2\in\cX\}$ and $h=h_2$ we have $\partial^2a/\partial x\partial y=\tilde O_k(1)$,
	while $\partial^2a/\partial x\partial y=0$ if $h\neq h_2$.
	Further, if $a\in\{q_{t,h}^{z_1z_2}:z_1,z_2\in\cX\}$, then \Lem~\ref{Lemma_occImplicit} yields
	$\partial^2a/\partial x\partial y=\tilde O_k(1)$ if $h=h_2$ and $\partial^2a/\partial x\partial y=\tilde O_k(2^{-k})$ otherwise.
	Hence,  $\cM_{xy}=\tilde O_k(2^{-k})$.
\item[Case 3: $x,y\in\{\omega_t^{z_1z_2}:z_1,z_2\in\spins\}$]	
	 \Lem~\ref{Lemma_occImplicit} yields $\partial^2a/\partial x\partial y=\tilde O_k(2^{-k})$ for all $a$.
	Therefore, the bound $\partial\cQ_{t,h}/\partial a=\tilde O_k(2^{-k})$ entails that $\cM_{xy}=\tilde O_k(2^{-k})$.
\end{description}
Combining these three estimate, we see that $\cM\preceq\cJ$ for a diagonal matrix $\cJ$
with entries as detailed in~(\ref{eqClaim_D2occProb1}) and~(\ref{eqClaim_D2occProb2}).

With respect to $\cN$ \Lem~\ref{Lemma_occImplicit} and (\ref{eqD2psi}) yield
	\begin{equation}\label{eqClaim_D2corrTermocc3}
	\frac{\partial^2\cQ_{t,h}}{\partial q_{t,h}^{z_1z_2\,2}}+
		\frac{\partial^2\cQ_{t,h}}{\partial p_{t,h}^{z_1z_2\,2}}
		+2\frac{\partial^2\cQ_{t,h}}{\partial p_{t,h}^{z_1z_2}\partial q_{t,h}^{z_1z_2}}=\tilde O_k(1).
	\end{equation}
To estimate the entries $\cN_{xy}$ we treat three cases separately.
\begin{description}
\item[Case 1: $x,y\in\{\omega_{t,h}^{z_1z_2}:h\in\brk{d_t},(z_1,z_2) \in\cX\}$]
	let $x=\omega_{t,h_1}^{z_{11}z_{12}}$, $y=\omega_{t,h_2}^{z_{21}z_{22}}$.
	\Lem~\ref{Lemma_occImplicit} shows that for the summand $h=h_1=h_2$ we have $\partial a/\partial x,\partial b/\partial y=\tilde O_k(1)$, whilst
	$(\partial a/\partial x)(\partial b/\partial y)=\tilde O_k(4^{-k})$ if either $h\neq h_1$ or $h\neq h_2$.
	Therefore, (\ref{eqClaim_D2corrTermocc3}) yields $\cN_{xy}=\tilde O_k(1)\vecone\{h_1=h_2\}+\tilde O_k(2^{-k})$.
\item[Case 2: $x\in\{\omega_t^{z_1z_2}:z_1,z_2\in\spins\}$, $y\in\{\omega_{t,h}^{z_1z_2}:h\in\brk{d_t},(z_1,z_2) \in\cX\}$]
	suppose that $x=\omega_t^{z_{11}z_{12}}$, $y=\omega_{t,h_2}^{z_{21}z_{22}}$.
	Then by \Lem~\ref{Lemma_occImplicit} the summand $h=h_2$ is $\tilde O_k(2^{-k})$, while all other summands are $\tilde O_k(4^{-k})$.
	Hence, $\cN_{xy}=\tilde O_k(2^{-k})$.
\item[Case 3: $x,y\in\{\omega_t^{z_1z_2}:z_1,z_2\in\spins\}$]	
	then \Lem~\ref{Lemma_occImplicit} yields $\cN_{xy}=\tilde O_k(2^{-k})$.
\end{description}
Hence, $\cN\preceq\cJ$ for a  diagonal matrix $\cJ$ that satisfies~(\ref{eqClaim_D2occProb1}) and~(\ref{eqClaim_D2occProb2}).

A similar argument applies to the other two terms
	 $\KL{\omega_{t,h}^{\rot\y}/\omega_{t}^{\purpur\y}}{q_{t,h}^{\rot\y}},\KL{\omega_{t,h}^{\y\rot}/\omega_{t}^{\y\purpur}}{q_{t,h}^{\y\rot}}$.
\end{proof}

\noindent
\Lem~\ref{Lemma_D2focc} is immediate from Claims~\ref{Claim_D2occProb}--\ref{Claim_D2corrTermocc}.

\begin{proof}[Proof of \Lem~\ref{Prop_DDF}]
This follows from \Lem s~\ref{Lemma_D2ent}, \ref{Lemma_D2disc}, \ref{Lemma_D2val} and~\ref{Lemma_D2focc}
and the affine relations  from Fact~\ref{Fact_affine}.
\end{proof}

\begin{proof}[Proof of \Lem~\ref{Lemma_tame}]
\Lem~\ref{Lemma_tame} follows from \Lem s~\ref{Lemma_Df} and~\ref{Prop_DDF} via a standard application of the Laplace method.
More specifically,  let $\Omega'$ be the set of all tame overlaps $(\omega,\gamma)$.
Moreover, for a large enough number $c''=c''(k)$ let $\Omega''$ be the set
of all $(\omega,\gamma)\in\Omega'$ such that for all $t\in T,\ell\in T^*,j\in[k_\ell],h\in[d_t]$, 
	$$\norm{\omega_t-\bar\omega_t}_\infty,\norm{\omega_{t,h}-\bar\omega_{t,h}}_\infty,\norm{\omega_{\ell,j}-\bar\omega_{\ell,j}}_\infty,
		\norm{\gamma_{\ell}-\bar\gamma_{\ell}}_\infty\leq c''n^{-1/2}.$$
\Lem s~\ref{Lemma_Df} and~\ref{Prop_DDF} imply that
	\begin{equation}\label{eqLemma_tame1}
	S'=\sum_{(\omega,\gamma)\in\Omega'}\exp(nF(\omega,\gamma))
		\leq O(1)\sum_{(\omega,\gamma)\in\Omega''}\exp(nF(\omega,\gamma)).
	\end{equation}
Further, let $C'=4\abs{\brk T}+\sum_{\ell\in T^*}4k_\ell+\bink{k_\ell}2$.
Then the affine relations from Fact~\ref{Fact_affine} imply that the set $\Omega''$ is contained in the affine image of the set of integer lattice points in a
	$C'$-dimensional cube with side lengths $O(\sqrt n)$.
(Indeed, once we fix for each $t\in T$ the parameters $\omega_t^{z_1z_2}$ with $z_1,z_2\in\{0,1\}$
and for every $\ell\in T^*$, $j,j'\in[k_\ell]$, $j\neq j'$ the parameters $\omega_{\ell,j}^{\purpur\purpur}$, $\gamma_{\ell,j}^{\rot\rot}$,
$\gamma_{\ell,j}^{\rot\cyan}$, $\gamma_{\ell,j}^{\cyan\rot}$, $\gamma_{\ell,j,j'}^{\y\y}$, the remaining components of $(\omega,\gamma)$ are implied.)
Therefore, with $C$ the number from~(\ref{eqC}), \Lem s~\ref{Lemma_Df} and~\ref{Prop_DDF} and the Laplace method yield
	\begin{align}\label{eqLemma_tame2}
	S''=\sum_{(\omega,\gamma)\in\Omega''}\exp(nF(\omega,\gamma))&\leq O(n^{-2C})\exp(nf(\bar\omega,\bar\gamma)).
	\end{align}
Hence, we need to compare $f(\bar\omega,\bar\gamma)$ with the formula from \Prop~\ref{Prop_firstMomentFormula}.
To this end, we observe that at the point $(\bar\omega,\bar\gamma)$
the parameters $(q_{\ell,j}^z)_{z\in\{\purpur,\y\}}$ and $q_{t,h}^\rot$ from \Prop~\ref{Prop_firstMomentFormula}
and the implicit parameters $(q_{\ell,j}^{z_1z_2})_{z_1,z_2\in\{\purpur,\y\}}$, $(q_{t,h}^{z_1z_2})_{z_1,z_2}$ from \Lem s~\ref{Lemma_smValid} and~\ref{Prop_f}
satisfy
	\begin{align*}
	q_{\ell,j}^{z_1z_2}&=q_{\ell,j}^{z_1}q_{\ell,j}^{z_2},&
	q_{t,h}^{\rot\rot}&=(q_{t,h}^{\rot})^2,&
	q_{t,h}^{\rot\cyan}&=q_{t,h}^{\cyan\rot}=q_{t,h}^{\rot}(1-q_{t,h}^{\rot}),&
	q_{t,h}^{\rot\y}&=q_{t,h}^{\y\rot}=q_{t,h}^{\rot}.
	\end{align*}
As a consequence, it is straightforward to check that 
	\begin{align}\label{eqLemma_tame3}
	S''&=O(\Erw_\cT[\cZ]^2).
	\end{align}
Combining~(\ref{eqLemma_tame1})--(\ref{eqLemma_tame3}), we conclude
that $S'\leq O(\Erw[\cZ|\cT]^2)$, as desired.
\end{proof}

\subsection{Wild overlaps}\label{Sec_wild}
The aim in this section is to prove 

\begin{lemma}\label{Lemma_wild}
Assume that $(\omega,\gamma)$ fails to be tame but 
	\begin{equation}\label{eqLemma_wild}
	\sum_{t\in T}\pi_t\omega_t^{00}=\frac14+\tilde O_k(2^{-0.49k}).
	\end{equation}
Then there exists a tame $(\tilde\omega,\tilde\gamma)$ such that
	$F(\omega,\gamma)\leq F(\tilde\omega,\tilde\gamma)-\Omega(1)$.
\end{lemma}

\noindent
Throughout, we tacitly assume that $\omega$ satisfies~(\ref{eqLemma_wild}).
Moreover, we let
	$\cS(\omega)=\cbc{t\in T:|\omega_t^{00}-1/4|>k^{-99}}.$

\subsubsection{A rough bound}
To prove \Lem~\ref{Lemma_wild} we proceed in two steps.
First, we argue that the contribution of $(\omega,\gamma)$ that satisfy~(\ref{eqLemma_wild}) 
but for which $\omega_t^{00}$ differs significantly from $\frac14$ for a large share of types $t$ is negligible.
The proof of this is based on a rough upper bound on $F(\omega,\gamma)$.
Subsequently we are going to derive a more accurate bound on those $(\omega,\gamma)$ that fail to be tame but for which $\omega_t^{00}$ is close to $1/4$ for most $t$.

\begin{lemma}\label{Lemma_rough}
We have $\textstyle\sup\cbc{F(\omega,\gamma):\Vol(\cS(\omega))>\exp(-\sqrt k)}<0$.
\end{lemma}

\noindent
The proof of \Lem~\ref{Lemma_rough} is based on the following very rough upper bound on $F(\omega,\gamma)$.

\begin{claim}\label{Claim_rough0}
Let 
	$$\hat f(\omega)=\sum_{t\in T}\pi_t H(\omega_t^{z_1z_2})_{z_1,z_2\in\cbc{0,1,*}}+\frac{m}{n}\sum_{\ell\in T^*}\pi_\ell\ln\brk{1-2\prod_{j \in [k_\ell]}\ell_j^\y+
		\prod_{j \in [k_\ell]}\omega_{\ell_j}^{\y\y}}.$$
Then $\sup_\gamma F(\omega,\gamma)\leq\hat f(\omega)+o(1)$.
\end{claim}
\begin{proof}
Consider a random vector $\vec\chi=(\vec\chi_{\ell,j}(i))_{\ell\in T^*,j\in[k_\ell],i\in[m_\ell]}$ whose entries are independent
random variables with values in $\cbc{(\purpur,\purpur),(\purpur,\y),(\y,\purpur),(\y,\y)}$ such that
	\begin{align}\label{eqClaim_rough0_0}
	\pr\brk{\vec\chi_{\ell,j}(i)=(z_1,z_2)}&=\omega_{\ell_j}^{z_1z_2}.
	\end{align}
Let $S$ be the event that for all $\ell\in T^*$ and $i\in[m_\ell]$ there exist
$j_1,j_2\in[k_\ell]$ such that $\vec\chi_{\ell,j_1}(i)\in\{(\purpur,\purpur),(\purpur,\y)\}$ and 
$\vec\chi_{\ell,j_2}(i)\in\{(\purpur,\purpur),(\y,\purpur)\}$.
Furthermore, let 
	$Y_{\ell,j}^{z_1 z_2}=|\{i\in[m_\ell]:\vec\chi_{\ell,j}(i)=(z_1,z_2)\}|$ and set 
	$Y_t^{z_1z_2}=(d_tn_t)^{-1}\sum_{h\in[d_t]}\sum_{(\ell,j)\in\partial(t,h)}Y_{\ell,j}^{z_1z_2}$.
Let $B$ be the event that $Y_t^{z_1z_2}=\omega_t^{z_1z_2}$ for all $t\in T$ and any $z_1,z_2\in\{\y,\purpur\}$
and that $Y_{\ell,j}^{\y\nix},Y_{\ell,j}^{\nix\y}\doteq\ell_j^{\y}$ for all $\ell,j$.
Then by the construction of $F$,
	\begin{align}\label{eqClaim_rough0_1}
	\sup_\gamma F(\omega,\gamma)&\leq\sum_{t\in T}\pi_t H(\omega_t^{z_1z_2})_{z_1,z_2\in\cbc{0,1,*}}+\frac{1}{n}\ln\pr\brk{S|B}.
	\end{align}
As (\ref{eqClaim_rough0_0}) ensures that $\Erw Y_{\ell,j}^{z_1z_2}=\omega_{\ell_j}^{z_1z_2}m_\ell$,
\Lem~\ref{Lemma_LLT} implies that $\pr\brk{B}=\exp(o(n))$.
Hence, by (\ref{eqClaim_rough0_1}),
	\begin{align}\label{eqClaim_rough0_2}
	\sup_\gamma F(\omega,\gamma)&\leq\sum_{t\in T}\pi_t H(\omega_t^{z_1z_2})_{z_1,z_2\in\cbc{0,1,*}}+\frac{1}{n}\ln\pr\brk{S}+o(1).
	\end{align}
Furthermore, as $\omega_{\ell_j}^{\y\y}+\omega_{\ell_j}^{\y\purpur},\omega_{\ell_j}^{\y\y}+\omega_{\ell_j}^{\purpur\y}\doteq\ell_j^\y$ for all $\ell,j$ by Fact~\ref{Fact_affine},
we see that
	\begin{align}\label{eqClaim_rough0_3}
	\frac1{n}\ln\pr\brk{S}&=\frac{m}{n}\sum_{\ell\in T^*}\pi_\ell\ln\brk{1-2\prod_{j \in [k_\ell]}\ell_j^\y+\prod_{j \in [k_\ell]}\omega_{\ell_j}^{\y\y}}+o(1).
	\end{align}
Finally, the assertion follows from~(\ref{eqClaim_rough0_2}) and~(\ref{eqClaim_rough0_3}).
\end{proof}

\begin{claim}\label{Claim_rough1}
Let $T_0=T_0(\omega)$ denote the set of all all types $t\in T$ such that
$\min\{\omega_t^{\purpur\purpur},\omega_t^{\purpur\y},\omega_t^{\y\purpur},\omega_t^{\y\y}\}>0.01$.
Then $$\sup\{\hat f(\omega):\omega\mbox{ satisfies }\Vol(T_0)<0.01\}<0.$$
\end{claim}
\begin{proof}
Assume that $\Vol(T_0)<0.01$.
Because $t^0,t^1=\frac12+\tilde O_k(2^{-k/2})$ for all $t$, (\ref{eqLemma_wild}) implies that
	\begin{align}\label{eqrough1}
	\sum_{t\in T}\pi_t\omega_t^{z_1z_2}=\frac14+\tilde O_k(2^{-0.49k})\qquad\mbox{for all }z_1,z_2\in\cbc{\purpur,\y}.
	\end{align}
Set $\delta=0.01+1/k$ and let
	$T_1=\cbc{t\in T:\omega_t^{\y\y}<\delta},$
		$T_2=\cbc{t\in T:\omega_t^{\y\y}>1/2-\delta}.$
Since $\Vol(T_0)<0.01$, (\ref{eqrough1}) implies that
	$\Vol(T_1)\geq0.48$, $\Vol(T_2)\geq0.48$.
Now, let $\cM$ be the set of all clause types $\ell$ that feature at least $0.4k$ literals of type $T_1$ and at least $0.4k$ literals of type $T_2$.
Then for any $\ell\in\cM$ we have
	\begin{equation}\label{eqClaim_rough1_1}
	1-2\prod_{j \in [k_\ell]}\ell_j^\y+\prod_{j \in [k_\ell]}\omega_{\ell_j}^{\y\y}\leq1-2^{1-k_\ell}+\tilde O_k(2^{-3k/2}).
	\end{equation}
Furthermore, {\bf DISC2} (from \Lem~\ref{Lemma_expansion}) implies that $\Vol(\cM)\geq1-k^{-9}$ \whp\
Hence, (\ref{eqClaim_rough1_1}) yields
	\begin{equation}\label{eqClaim_rough1_2}
	\frac{m}{n}\sum_{\ell\in\cM}\pi_\ell\ln\brk{1-2\prod_{j \in [k_\ell]}\ell_j^\y+\prod_{j \in [k_\ell]}\omega_{\ell_j}^{\y\y}}
		\leq-2\ln2+o_k(1).
	\end{equation}
By comparison, since $\Vol(T_0)\leq0.01$, we find
	\begin{equation}\label{eqClaim_rough1_3}
	\sum_{t\in T}\pi_t H(\omega_t^{z_1z_2})_{z_1,z_2\in\cbc{0,1,*}}\leq1.9\ln2+o_k(1).
	\end{equation}
Combining~(\ref{eqClaim_rough1_2}) and~(\ref{eqClaim_rough1_3}), we conclude that $\hat f(\omega)<0$.
\end{proof}

\begin{proof}[Proof of \Lem~\ref{Lemma_rough}]
Let $\eps=k^{-99}$ and $\delta=\exp(-\sqrt k)$.
Let $T_2$ be the set of all types $t$ such that $|\omega_t^{\y\y}-1/4|>\eps$.
Assume that $\Vol(T_2)>\delta$.
By Claim~\ref{Claim_rough1} and {\bf DISC1}, we may assume that the set $\cM$ of all clause
types $\ell$ with $k_\ell=k$ that contain at least $0.01k$ literals from $T_0$ satisfies $\Vol(\cM)=1-\exp(-\Omega_k(k))$.
Furthermore,  for any $\ell\in\cM$,
	\begin{align}\label{eqroughProof1}
	1-2\prod_{j \in [k_\ell]}\ell_j^\y+\prod_{j \in [k_\ell]}\omega_{\ell_j}^{\y\y}\leq1-2^{1-k}(1+\exp(-\Omega_k(k))).
	\end{align}

Now, obtain $\hat\omega$ from $\omega$ by setting $\hat\omega_t^{z_1z_2}=t^{z_1}t^{z_2}$ for all $z_1,z_2\in\spins,t\in T_2$.
In particular, $\hat\omega_t^{\y\y}=t^{0}t^{0}$  for $t\in T_2$.
Hence, (\ref{eqroughProof1}) implies
	\begin{align}\nonumber
	\hat f(\hat\omega)-\hat f(\omega)
		&=\exp(-\Omega_k(k))+
		\sum_{t\in T_2}\pi_t\brk{H(\hat\omega_t^{z_1z_2})_{z_1,z_2\in\cbc{0,1,*}}-H(\omega_t^{z_1z_2})_{z_1,z_2\in\cbc{0,1,*}}}\\
		&\geq \exp(-\Omega_k(k))+\Vol(T_2)\Omega_k(\eps^2)\label{eqroughProof2}\geq \exp(-k^{0.51}).
	\end{align}
On the other hand, a direct calculation shows that $\hat f(\hat\omega)\leq \tilde O_k(2^{-k})$.
Hence, (\ref{eqroughProof2}) implies that $\hat f(\omega)<0$.
Finally, the assertion follows from Claim~\ref{Claim_rough0}.
\end{proof}

\subsubsection{Reducing the discrepancy}\label{Sec_wild_disc}
In the following we enhance the bound from \Lem~\ref{Lemma_rough} to prove \Lem~\ref{Lemma_wild}.
We begin with the following statement.

\begin{lemma}\label{Lemma_reduceDisc}
Assume that $(\omega,\gamma)$ is such that
$\Vol(\cS(\omega))\leq \exp(-\sqrt k)$ but  the following condition is violated.
	\begin{equation}\label{eqLemma_reduceDisc}
	\parbox{14.5cm}{
	For all $t\in T\setminus\cS(\omega)$, $h\in[d_t]$, $(\ell,j)\in\partial(t,h)$ 
	we have $|\omega_{\ell,j}^{\y\y}-\omega_{t}^{00}|\leq2^{-k/3}$.}
	 \end{equation}
Then there exists $\hat\omega$ such that $F(\hat\omega,\gamma)>F(\omega,\gamma)+\Omega(1)$.
\end{lemma}

The proof of \Lem~\ref{Lemma_reduceDisc} is based on a local variations argument.
Let $t\in T\setminus\cS(\omega)$, $h\in[d_t]$ and assume
that $|\omega_{\ell,j}^{\y\y}-\omega_{t}^{00}|>2^{-k/3}$ for some $(\ell,j)\in\partial(t,h)$.
Then there exists $(\ell',j')\in\partial(t,h)$ such that  $|\omega_{\ell',j'}^{\y\y}-\omega_{t}^{00}|\geq\Omega(1)$
and such that $\sign(\omega_{\ell',j'}^{\y\y}-\omega_{t}^{00})\neq\sign(\omega_{\ell,j}^{\y\y}-\omega_{t}^{00})$.
Now, pick a number $\delta$ with $\sign(\delta)=\sign(\omega_{\ell,j}^{\y\y}-\omega_{t}^{00})$ of sufficiently small absolute value and let $\delta'=\delta m_\ell/m_{\ell'}$.
Further, let $\hat\omega$  be such that
	$\hat\omega_{\ell,j}^{\y\y}\doteq\omega_{\ell,j}^{\y\y}-\delta,\hat\omega_{\ell',j'}^{\y\y}\doteq\omega_{\ell,j}^{\y\y}+\delta'$,
	 $\hat\omega_{\ell'',j''}^{\y\y}\doteq\omega_{\ell'',j''}^{\y\y}$ if $(\ell'',j'')\not\in\{(\ell,j),(\ell',j')\}$, $\omega_t=\hat\omega_t$ for all $t\in T$ and such that the affine
relations from Fact~\ref{Fact_affine} hold.
Then
	\begin{equation}\label{eqLemma_reduceDisc_ent_occ}
	\Fent(\hat\omega)=\Fent(\omega),\Focc(\hat\omega)=\Focc(\omega).
	\end{equation}
Moreover, differentiating the Kullback-Leibler divergence, we see that
	\begin{equation}\label{eqLemma_reduceDisc2}
	\Fdisc(\hat\omega)-\Fdisc(\omega)\geq
		\frac{\delta m_\ell}{n'}\Omega_k(2^{-k/3}).
	\end{equation}

\begin{claim}\label{eqLemma_reduceDisc1}
We have 
	\begin{align}\label{eqLemma_reduceDisc1_Part1}
	\Fvall(\hat\omega_\ell,\gamma_\ell)&\geq \Fvall(\omega_\ell,\gamma_\ell)+\frac{\delta m_\ell}{n'}\tilde O_k(2^{-k}),\\
	{F_{\mathrm{val},\ell'}}(\hat\omega_{\ell'},\gamma_{\ell'})&\geq
		{F_{\mathrm{val},\ell'}}(\omega_{\ell'},\gamma_{\ell'})+\frac{\delta' m_{\ell'}}{n'}\tilde O_k(2^{-k}).
			\label{eqLemma_reduceDisc1_Part2}
	\end{align}
\end{claim}
\begin{proof}
We prove~(\ref{eqLemma_reduceDisc1_Part1}) in detail; the very same argument yields~(\ref{eqLemma_reduceDisc1_Part2}). 
For $\alpha\in[0,1]$ we let
 $\omega_\ell(\alpha)$ be the vector obtained from $\omega_\ell$ by replacing $\omega_{\ell,j}$ by $(1-\alpha)\omega_{\ell,j}+\alpha\hat\omega_{\ell,j}$.
Using the notation from the definition of $\Fvall$ in \Sec~\ref{Sec_fval},
we are going to ``interpolate'' between the probability spaces $\cX_\ell(\omega_\ell(0))$ and $\cX_\ell(\omega_\ell(1))$.
Let $\vec X_\ell^{\alpha}$ denote a uniformly random element of $\cX_\ell(\omega_\ell(\alpha))$.

Let us fix disjoint sets $\cG_{\ell,h}^{z_1z_2},\cG_{\ell,h,h'}^{\y\y}\subset[m_\ell]$, $(z_1,z_2)\in\{(\rot,\rot),(\rot,\cyan),(\cyan,\rot),(\rot,\y),(\y,\rot)\}$, $h,h'\in[k_\ell]$, $h\neq h'$,
such that $|\cG_{\ell,h}^{z_1z_2}|=m_\ell\gamma_{\ell,h}^{z_1z_2}$ and $|\cG_{\ell,h,h'}^{\y\y}|=m_\ell\gamma_{\ell,h,h'}^{\y\y}$. 
Let $\cG$ denote the union of all of these sets.
Further, let $\cR(\alpha)\subset\cX_\ell(\omega_\ell(\alpha))$ be the event that 
	\begin{itemize}
	\item for all  $(z_1,z_2)$,  $h\in[k_\ell]$, $i\in\cG_{\ell,h}^{z_1z_2}$, $\vec X_{\ell}^\alpha(i)$ is a  $(z_1,z_2,j)$-clause,
	\item for all $h\neq h'$,  $i\in\cG_{\ell,h,h'}^{\y\y}$, $\vec X_{\ell}^\alpha(i)$ is a $(\y,\y,h,h')$-clause.
	\end{itemize}
Additionally, let $\cC(\alpha)$ be the event that  $\vec X_{\ell}^\alpha(i)$ is a $(\cyan,\cyan)$-clause for all $i\in[m_\ell]\setminus\cG$.
Because the distribution of the random vector $\vec X_{\ell}^\alpha(i)$ is invariant under permutations of the clause indices $i$, we see that
	\begin{equation}\label{eqLemma_reduceDisc1_1}
	\Fvall(\hat\omega_\ell,\gamma_\ell)-\Fvall(\omega_\ell,\gamma_\ell)
		=\frac1{m_\ell}\brk{\ln\pr\brk{\cR(1)\cap\cC(1)}-
			\ln\pr\brk{\cR(0)\cap\cC(0)}}.
	\end{equation}
To estimate the r.h.s.\ of~(\ref{eqLemma_reduceDisc1_1}), we are going to work out (roughly speaking) the derivative of
	$\pr\brk{\cR(\alpha)\cap\cC(\alpha)}$ for $\alpha\in[0,1]$.
To deal with the issue that $\cR(\alpha),\cC(\alpha)$ are dependent, we are going to identify an event $\cE(\alpha,u)$ such that
$\cR(\alpha),\cC(\alpha)$ are independent given $\cE(u)$.
More specifically,
if $u_\ell=(u_{\ell,h}^{z_1z_2})_{z_1,z_2\in\cbc{\purpur,\y},h\in[k_\ell]}$ is 
such that $(u_{\ell,h}^{z_1z_2})_{z_1,z_2\in\cbc{\purpur,\y}}$ is a probability distribution for each $h\in[k_\ell]$,  then we let $\cE(u_\ell)$ be the event that
	\begin{align*}
	\forall h\in[k_\ell]:
	\abs{\cbc{i\in\cG:\vec X^\alpha_{\ell,h}(i)=\bc{z_1,z_2}}}=u_{\ell,h}^{z_1z_2}|\cG|.
	\end{align*}
Then for any $u_\ell$ such that $\pr\brk{\cE(u_\ell)}>0$ we have
	\begin{align}\label{eqLemma_reduceDisc1_2}
	\pr\brk{\cR(\alpha)\cap\cC(\alpha)|\cE(u_\ell)}=\pr\brk{\cR(\alpha)|\cE(u_\ell)}\pr\brk{\cC(\alpha)|\cE(u_\ell)}\pr\brk{\cE(u_\ell)}.
	\end{align}
Thus, we need to get a handle on $\pr\brk{\cR(\alpha)|\cE(u_\ell)},\pr\brk{\cC(\alpha)|\cE(u_\ell)},\pr\brk{\cE(u_\ell)}$.

Because given $\cE(u_\ell)$ we know the precise statistics of the ``dominos'' placed in clauses with indices in $\cG$, we have
	\begin{align}\label{eqLemma_reduceDisc1_3}
	\pr\brk{\cR(\alpha)|\cE(u_\ell)}&=\pr\brk{\cR(0)|\cE(u_\ell)}.
	\end{align}
Further, letting 	$\tilde u_{\ell}=\bc{m_\ell-|\cG|}^{-1}\bc{m_\ell \omega_\ell(\alpha)-|\cG|u_\ell},$ we obtain from Fact~\ref{Fact_binomialLargeDev}
	\begin{align}\label{eqLemma_reduceDisc1_4}
	\frac1{m_\ell}\ln\pr\brk{\cE(u_\ell)}\sim
		-\sum_j\frac{|\cG|}{m_\ell}\KL{u_{\ell,j}}{\omega_{\ell,j}(\alpha)}+\bc{1-\frac{|\cG|}{m_\ell}}\KL{\tilde u_{\ell,j}}{\omega_{\ell,j}(\alpha)};
	\end{align}
here $j$ ranges over indices such that $\omega_{\ell,j}\neq\hat\omega_{\ell,j}$.
Differentiating~(\ref{eqLemma_reduceDisc1_4}) using Fact~\ref{Fact_affine}, we find that
	\begin{align}\nonumber
	-\frac{\partial}{\partial\alpha}\KL{u_{\ell,j}}{\omega_{\ell,j}(\alpha)}&
		=\sum_{z_1,z_2\in\{\purpur,\y\}}\frac{u_{\ell,j}^{z_1z_2}}{\omega_{\ell,j}^{z_1z_2}(\alpha)}\frac{\partial\omega_{\ell,j}^{z_1z_2}(\alpha)}{\partial\alpha}\\
		&=\delta\brk{\frac{u_{\ell,j}^{\purpur\y}}{\omega_{\ell,j}^{\purpur\y}(\alpha)}+\frac{u_{\ell,j}^{\y\purpur}}{\omega_{\ell,j}^{\y\purpur}(\alpha)}
			-\frac{u_{\ell,j}^{\y\y}}{\omega_{\ell,j}^{\y\y}(\alpha)}-\frac{u_{\ell,j}^{\purpur\purpur}}{\omega_{\ell,j}^{\purpur\purpur}(\alpha)}},
				\label{eqLemma_reduceDisc1_5}\\
	-\frac{\partial}{\partial\alpha}\KL{\tilde u_{\ell,j}}{\omega_{\ell,j}(\alpha)}&=
		\sum_{z_1,z_2\in\{\purpur,\y\}}\frac{\tilde u_{\ell,j}^{z_1z_2}}{\omega_{\ell,j}^{z_1z_2}(\alpha)}\frac{\partial\omega_{\ell,j}^{z_1z_2}(\alpha)}{\partial\alpha}
			-\frac{\partial\tilde u_{\ell,j}^{z_1z_2}}{\partial\alpha}\ln\frac{\tilde u_{\ell,j}^{z_1z_2}}{\omega_{\ell,j}^{z_1z_2}(\alpha)}\nonumber\\
		&=\delta\bigg[
			\frac{\tilde u_{\ell,j}^{\purpur\y}}{\omega_{\ell,j}^{\purpur\y}(\alpha)}+\frac{\tilde u_{\ell,j}^{\y\purpur}}{\omega_{\ell,j}^{\y\purpur}(\alpha)}
			-\frac{\tilde u_{\ell,j}^{\y\y}}{\omega_{\ell,j}^{\y\y}(\alpha)}-\frac{\tilde u_{\ell,j}^{\purpur\purpur}}{\omega_{\ell,j}^{\purpur\purpur}(\alpha)}	
				\bigg]
					\nonumber\\
		&\qquad+\frac{\delta m_\ell}{m_\ell-|\cG|}\bigg[\ln\frac{\tilde u_{\ell,j}^{\y\y}}{\omega_{\ell,j}^{\y\y}(\alpha)}+
			\ln\frac{\tilde u_{\ell,j}^{\purpur\purpur}}{\omega_{\ell,j}^{\purpur\purpur}(\alpha)}
			-\ln\frac{\tilde u_{\ell,j}^{\purpur\y}}{\omega_{\ell,j}^{\purpur\y}(\alpha)}-\ln\frac{\tilde u_{\ell,j}^{\y\purpur}}{\omega_{\ell,j}^{\y\purpur}(\alpha)}
			\bigg].
				\label{eqLemma_reduceDisc1_6}
	\end{align}
We claim that
	\begin{align}\label{eqLemma_reduceDisc1_9}
	\frac{\partial}{\partial\alpha}-\frac{|\cG|}{m_\ell}\KL{u_\ell}{\omega_\ell(\alpha)}-\bc{1-\frac{|\cG|}{m_\ell}}\KL{\tilde u_\ell}{\omega_\ell(\alpha)}
		\geq-|\delta|\tilde O_k(2^{-k}).
	\end{align}
Indeed, if $\omega_{\ell,j}^{z_1z_2}\geq1/k$, then the logarithmic terms from (\ref{eqLemma_reduceDisc1_6})
 contribute $|\delta|\tilde O_k(2^{-k})$  to (\ref{eqLemma_reduceDisc1_9}).
Hence, assume that $\omega_{\ell,j}^{z_1z_2}<1/k$.
Then $\delta<0$ if $z_1=z_2$ and $\delta>0$ if $z_1\neq z_2$.
Assume without loss that $z_1=z_2$.
If $\tilde u_{\ell,j}^{z_1z_2}\leq\omega_{\ell,j}^{z_1z_2}(\alpha)$, then the contribution of the logarithmic terms from (\ref{eqLemma_reduceDisc1_6}) is non-negative.
Otherwise the definition ensures that $\tilde u_{\ell,j}^{z_1z_2}\leq(1+\tilde O_k(2^{-k}))\omega_\ell^{z_1z_2}(\alpha)$,
whence the contribution of the logarithmic term is $|\delta|\tilde O_k(2^{-k})$.
Further, the contribution of the non-logarithmic terms from (\ref{eqLemma_reduceDisc1_5})--(\ref{eqLemma_reduceDisc1_6}) to (\ref{eqLemma_reduceDisc1_9}) comes to
	$$(-1)^{\vecone\{z_1\neq z_2\}}
		\frac{\delta}{\omega_{\ell,j}^{z_1z_2}(\alpha)}\brk{\frac{|\cG|}{m_\ell}u_{\ell,j}^{z_1z_2}+\frac{m_\ell-|\cG|}{m_\ell}\tilde u_{\ell,j}^{z_1z_2}}=
			(-1)^{\vecone\{z_1\neq z_2\}}\delta.$$
Summing over $z_1,z_2$ yields (\ref{eqLemma_reduceDisc1_9}).
	
As a next step, we calculate the derivative of $Q(\alpha,u_\ell)=\frac1{m_\ell}\ln\pr\brk{\cC(\alpha)|\cE(u_\ell)}$.
This is via a similar argument as in the proof of \Lem~\ref{Lemma_smValid}.
More specifically, we are going to calculate the derivative of
	\begin{align*}
	g_{\ell}^{\cyan\cyan}&=
	1-\prod_{j \in [k_\ell]}q^{\y\nix}_{\ell,j}-\sum_{j \in [k_\ell]}q^{\purpur\nix}_{\ell,j}\prod_{j'\neq j}q^{\y\nix}_{\ell,j'}
					-\prod_{j \in [k_\ell]}q^{\nix\y}_{\ell,j}-\sum_{j \in [k_\ell]}q^{\nix\purpur}_{\ell,j}\prod_{j'\neq j}q^{\nix\y}_{\ell,j'}\\
			&\qquad\qquad\qquad	+\prod_{j \in [k_\ell]}q_{\ell,j}^{\y\y}+\sum_{j \in [k_\ell]}(1-q_{\ell,j}^{\y\y})\prod_{j'\neq j}q_{\ell,j'}^{\y\y}
				+\sum_{j_1\neq j_2}q_{\ell,j_1}^{\purpur\y}q_{\ell,j_2}^{\y\purpur}\prod_{j\neq j_1,j_2}q_{\ell,j}^{\y\y}
	\end{align*}
for an appropriately defined $q_{\ell,j}=q_{\ell,j}(\alpha,u_\ell)$.
To determine $q_{\ell,j}$, we let
	\begin{align*}
\hat e_{\ell,j}^{\purpur\purpur}&=
	\frac{q_{\ell,j}^{\purpur\purpur}}{g_\ell^{\cyan\cyan}}
		\brk{1-\prod_{j'\neq j}q^{\nix\y}_{\ell, j'}-\prod_{j'\neq j}q^{\y\nix}_{\ell, j'}+\prod_{j'\neq j}q^{\y\y}_{\ell,j'}},\\
\hat e_{\ell,j}^{\purpur\y}&=
	\frac{q_{\ell,j}^{\purpur\y}}{g_{\ell}^{\cyan\cyan}}
		\bigg[1-\prod_{j'\neq j}q^{\nix\y}_{\ell, j'}-\prod_{j'\neq j}q^{\y\nix}_{\ell, j'}
			-\sum_{j'\neq j}q^{\nix\purpur}_{\ell, j'}\prod_{j''\neq j,j'}q^{\nix\y}_{\ell, j''}+\prod_{j'\neq j}q^{\y\y}_{\ell, j'}
					+\sum_{j'\neq j}q^{\y\purpur}_{\ell, j'}\prod_{j''\neq j,j'}q^{\y\y}_{\ell, j''}
					\bigg],\\
\hat e_{\ell,j}^{\y\purpur}&=
	\frac{q_{\ell,j}^{\y\purpur}}{g_{\ell}^{\cyan\cyan}}
		\bigg[1-\prod_{j'\neq j}q^{\y\nix}_{\ell, j'}-\prod_{j'\neq j}q^{\nix\y}_{\ell, j'}
			-\sum_{j'\neq j}q^{\purpur\nix}_{\ell, j'}\prod_{j''\neq j,j'}q^{\y\nix}_{\ell, j''}+\prod_{j'\neq j}q^{\y\y}_{\ell, j'}
					+\sum_{j'\neq j}q^{\purpur\y}_{\ell, j'}\prod_{j''\neq j,j'}q^{\y\y}_{\ell, j''}
					\bigg].
	\end{align*}
For any $q=(q_{\ell,j})_{j\in[k_\ell]}$ such that $q_{\ell,j}^{\y\nix},q_{\ell,j}^{\nix\y}=\frac12+O_k(k^{-2})$ we find
	$D\hat e_{\ell,j}=\id+\cM_{\ell,j},$
where $\cM_{\ell,j}$ is a matrix all of whose entries are $\tilde O_k(2^{-k})$.
Hence, by the inverse function theorem there exists $q_\ell=q_\ell(\alpha,u_\ell)$ such that
	$e_{\ell,j}^{z_1z_2}=\tilde u_{\ell,j}^{z_1z_2}$.
With this choice of $q_\ell$, we have
	\begin{equation}\label{eqLemma_reduceDisc1_10}
	\frac1{m_\ell}\ln\pr\brk{\cC(\alpha)|\cE(u_\ell)}\sim\bc{1-\frac{|\cG|}{m_\ell}}\brk{\ln g_\ell^{\cyan\cyan}+\sum_{h\in[k_\ell]}\KL{\tilde u_{\ell,h}}{q_{\ell,h}}}.
	\end{equation}
Once more by the inverse function theorem, we have
	$Dq_{\ell,j}=\id+\cN_{\ell,j},$
where $\cN_{\ell,j}$ is another matrix all of whose entries are $\tilde O_k(2^{-k})$.
Using this estimate to differentiate the r.h.s.\ of~(\ref{eqLemma_reduceDisc1_10}), we see that 
	\begin{equation}\label{eqLemma_reduceDisc1_11}
	\frac{\partial}{\partial\alpha}\ln g_\ell^{\cyan\cyan}+\sum_{h\in[k_\ell]}\KL{\tilde u_{\ell,h}}{q_{\ell,h}}=
	\delta\tilde O_k(2^{-k})+\sum_{h,z_1,z_2}
		\frac{\partial\tilde u_{\ell,h}^{z_1z_2}}{\partial\alpha}\ln\frac{\hat e_{\ell,h}^{z_1z_2}}{q_{\ell,h}^{z_1z_2}}
			-\frac{\partial q_{\ell,h}^{z_1z_2}}{\partial\alpha}\frac{\hat e_{\ell,h}^{z_1z_2}}{q_{\ell,h}^{z_1z_2}}
		=\delta\tilde O_k(2^{-k}).
	\end{equation}

Finally, combining  (\ref{eqLemma_reduceDisc1_2})--(\ref{eqLemma_reduceDisc1_4}) and (\ref{eqLemma_reduceDisc1_9})-- (\ref{eqLemma_reduceDisc1_11})
	and integrating over $\alpha\in[0,1]$, we conclude that 	
	$$\frac1{m_\ell}\brk{\ln\pr\brk{\cR(1)\cap\cC(1)}-
			\ln\pr\brk{\cR(0)\cap\cC(0)}}\geq\delta\tilde O_k(2^{-k}).$$
Thus, the claim follows from~(\ref{eqLemma_reduceDisc1_1}).
\end{proof}

\begin{proof}[Proof of \Lem~\ref{Lemma_reduceDisc}]
The assertion is immediate from Claim~\ref{eqLemma_reduceDisc1} and equations (\ref{eqLemma_reduceDisc_ent_occ}), (\ref{eqLemma_reduceDisc2}).
\end{proof}

\subsubsection{Increasing the entropy}\label{Sec_wild_ent}
Assume that $(\omega,\gamma)$ is such that $\cS(\omega)\neq\emptyset$.
Let $\cS'(\omega)$ be the set of all pairs $(\ell,j)$ such that there exist $t\in\cS(\omega)$ and $h\in[d_t]$ such that $(\ell,j)\in\partial(t,h)$
and let $\tilde\cS(\omega)$ be the set of all $\ell$ such that $(\ell\times[k_\ell])\cap\cS'(\omega)\neq\emptyset$.
Moreover, let $\tilde\omega$ be such that
$\tilde\omega_t=\bar\omega_t$ for all $t\in\cS(\omega)$ and $\tilde\omega_{\ell,j}\doteq\bar\omega_{\ell,j}$ 
for all $(\ell,j)\in\cS'(\omega)$, while $\tilde\omega_t=\omega_t$ for all $t\not\in\cS(\omega)$
and $\tilde\omega_{\ell,j}\doteq\omega_{\ell,j}$ for all $(\ell,j)\not\in\cS'(\omega)$.
Further, let  $\tilde\gamma_{\ell,j}\doteq\gamma_{\ell,j}$ for all $\ell\not\in\tilde\cS(\omega)$, $j\in[k_\ell]$
and let 
	$\tilde\gamma_{\ell,j}$ for 
	$\ell\in\tilde\cS(\omega)$, $j\in[k_\ell]$ be such that $\Fvall(\tilde\omega_\ell,\tilde\gamma_\ell)$ is maximum
	subject to the affine relations from Fact~\ref{Fact_affine}.

\begin{lemma}\label{Lemma_wild_ent}
Assume that $(\omega,\gamma)$ is such that~(\ref{eqLemma_wild}) and~(\ref{eqLemma_reduceDisc}) are satisfied and $\Vol(\cS(\omega))\leq\exp(-\sqrt k)$.
Then
	$$F(\tilde\omega,\tilde\gamma)-\Focc(\tilde\omega)\geq F(\omega,\gamma)-\Focc(\omega)+\tilde\Omega_k(1)\Vol(\tilde\cS(\omega)).$$
\end{lemma}

\noindent
The rest of this section is devoted to the proof of \Lem~\ref{Lemma_wild_ent}.
We begin with the following statement.
Let $\cS''(\omega)$ be the set of all clause types $\ell\in T^*$ such that $|(\cbc\ell\times[k_\ell])\cap\cS'(\omega)|\geq0.9k$.

\begin{claim}\label{Claim_ILikeExpansion}
We have $\sum_{\ell\in\cS''(\omega)}\frac{m_\ell}{n}\Fvall(\tilde\omega_\ell,\tilde\gamma_\ell)=\tilde O_k(2^{-k})\Vol(\tilde\cS(\omega)).$
\end{claim}
\begin{proof}
Since $\Vol(\cS(\omega))\leq \exp(-\sqrt k)$, {\bf DISC3} implies that
	\begin{equation}\label{eqILikeExpansion}
	\sum_{\ell\in\cS''(\omega)}m_\ell/n\leq \tilde O_k(1)\Vol\bc{\cS(\omega)}.
	\end{equation}
Let $\ell\in\cS''(\omega)$.
Since $\tilde\omega_\ell,\tilde\gamma_\ell$ satisfy the assumptions of \Lem~\ref{Lemma_implicit2},
we obtain $q_\ell(\tilde\omega_\ell,\tilde\gamma_\ell)$ such that
$\Fvall(\tilde\omega_\ell,\tilde\gamma_\ell)=\fvall(\tilde\omega_\ell,\tilde\gamma_\ell,q_\ell(\tilde\omega_\ell,\tilde\gamma_\ell))+o(1)$.
Furthermore, \Lem~\ref{Lemma_implicit2} entails that $|q_{\ell,j}^{z_1z_2}-1/4|\leq k^{-2}$ for all $j\in[k_\ell]$, $z_1,z_2\in\cbc{\purpur,\y}$.
Therefore, we verify that
	$\fvall(\tilde\omega_\ell,\tilde\gamma_\ell,q_\ell(\tilde\omega_\ell,\tilde\gamma_\ell))=\tilde O_k(2^{-k})$.
Hence, (\ref{eqILikeExpansion}) implies
	\begin{equation}\label{eqILikeExpansion2}
	\sum_{\ell\in\cS''(\omega)}\frac{m_\ell}{n}\Fvall(\tilde\omega_\ell,\tilde\gamma_\ell)=\tilde O_k(2^{-k})\Vol(\cS(\omega)).
	\end{equation}
Because $d_t=\Theta_k(k2^k)$ for all $t\in T$, the assertion follows from~(\ref{eqILikeExpansion2}).
\end{proof}

\begin{claim}\label{Claim_WithoutExpansion}
We have
$\sum_{\ell\not\in\cS''(\omega)}\frac{m_\ell}{n}\brk{\Fvall(\tilde\omega_\ell,\tilde\gamma_\ell)-\Fvall(\omega_\ell,\gamma_\ell)}
		\leq \exp\bc{-\Omega_k(k)}\Vol(\tilde\cS(\omega)).$
\end{claim}
\begin{proof}
Fix $\ell\in\tilde\cS\bc\omega\setminus\cS''\bc\omega$.
To compare $\Fvall(\tilde\omega_\ell,\tilde\gamma_\ell)$ and $\Fvall(\omega_\ell,\gamma_\ell)$, we proceed in five steps.
Let $\varphi_0$ be such that $\frac{m_\ell}{n}\varphi_0=\Fvall(\omega_\ell,\gamma_\ell)$.
Moreover, let
	$\varphi_1=-\KL{\gamma_\ell}{g_\ell(\omega_\ell)}.$
Then 
	$\varphi_0\leq\varphi_1+o(1)$ by \Lem~\ref{Lemma_smValid_rough}.
Further, let 
	$$\varphi_2=\ln\brk{1-\prod_{j\in[k_\ell]}\omega_{\ell,j}^{\y\nix}-\prod_{j\in[k_\ell]}\omega_{\ell,j}^{\nix\y}+\prod_{j\in[k_\ell]}\omega_{\ell,j}^{\y\y}}.$$
Since the sum
of all entries of $g_\ell(\omega_\ell)$ is no greater than $1-\prod_{j\in[k_\ell]}\omega_{\ell,j}^{\y\nix}-\prod_{j\in[k_\ell]}\omega_{\ell,j}^{\nix\y}+\prod_{j\in[k_\ell]}\omega_{\ell,j}^{\y\y}$,
we see that  $\varphi_2\geq\varphi_1$. Moreover, let 
	$$\varphi_3=
		\ln\brk{1-\prod_{j\in[k_\ell]}\tilde\omega_{\ell,j}^{\y\nix}-\prod_{j\in[k_\ell]}\tilde\omega_{\ell,j}^{\nix\y}+\prod_{j\in[k_\ell]}\tilde\omega_{\ell,j}^{\y\y}}.$$
To compare $\varphi_3$ and $\varphi_2$, we note that
by Fact~\ref{Fact_affine} and the construction of $\tilde\omega$ we have
	$$\omega_{\ell,j}^{\y\nix},\tilde\omega_{\ell,j}^{\y\nix},\omega_{\ell,j}^{\nix\y},\tilde\omega_{\ell,j}^{\nix\y}\doteq\ell_j^\y.$$
Furthermore, because $\ell\not\in\cS''(\omega)$ we have
	$\prod_{j\in[k_\ell]}\omega_{\ell,j}^{\y\y}\leq 2^{-\Omega_k(k)}\prod_{j\in[k_\ell]}\ell_j^{\y}.$
Additionally, the construction of $\tilde\omega$ ensures that 
	$\prod_{j\in[k_\ell]}\tilde\omega_{\ell,j}^{\y\y}\leq 2^{-k-\Omega_k(k)}$.
Consequently, there exists a fixed number $c_1<1/2$ such that $\varphi_3\geq\varphi_2-c_1^k$. To proceed, 
we observe that \Lem~\ref{Lemma_implicit2} applies to $(\tilde\omega,\tilde\gamma)$;
let $q_\ell=q_\ell(\tilde\omega_\ell,\tilde\gamma_\ell)$  be the vector produced by \Lem~\ref{Lemma_implicit2}
and set
	$$\varphi_4=\ln\brk{1-\prod_{j\in[k_\ell]}q_{\ell,j}^{\y\nix}-\prod_{j\in[k_\ell]}q_{\ell,j}^{\nix\y}+\prod_{j\in[k_\ell]}q_{\ell,j}^{\y\y}}.$$
Because 
\Lem~\ref{Lemma_implicit2} guarantees that
	$$|q_{\ell,j}^{z_1z_2}-\tilde\omega_{\ell,j}^{z_1z_2}|=O_k(2^{-k})\quad\mbox{ for all $j\in[k_\ell],z_1,z_2\in\{\purpur,\y\}$},$$
we conclude that
	 $\varphi_4\geq\varphi_3-c_2^k$ for some fixed $c_2<1/2$.
Finally, let $\varphi_5$ be such that $\fvall(\tilde\omega_\ell,\tilde\gamma_\ell,q_\ell)=\frac{m_\ell}{n}\varphi_5$.
Then 
$\Fvall(\tilde\omega_\ell,\tilde\gamma_\ell)=\frac{m_\ell}{n}\varphi_5+o(1)$.
Moreover, the choice of $\tilde\gamma_\ell$ ensures the existence of $c_3<1/2$ such that $\varphi_5\geq\varphi_4-c_3^k$.
Combining all of the above estimates, we obtain
	\begin{equation}\label{eqClaim_WithoutExpansion2}
	\Fvall(\omega_\ell,\gamma_\ell)=\frac{m_\ell}{n}\varphi_0\leq\frac{m_\ell}{n}\brk{\varphi_5+2^{-k-\Omega_k(k)}}.
	\end{equation}
Summing~(\ref{eqClaim_WithoutExpansion2}) over $\ell\not\in\cS''(\omega)$ and recalling that $m/n\leq 2^k$ completes the proof.
\end{proof}

\begin{proof}[Proof of \Lem~\ref{Lemma_wild_ent}]
By direct inspection,
	\begin{equation}\label{eqLemma_wild_ent_1}
	\Fent(\tilde\omega,\tilde\gamma)+\Fdisc(\tilde\omega,\tilde\gamma)\geq
		\Fent(\omega,\gamma)+\Fdisc(\omega,\gamma)+\tilde\Omega_k(1)\Vol(\cS(\omega))\geq\tilde\Omega_k(1)\Vol(\tilde\cS(\omega)).
	\end{equation}
The assertion follows by combining~(\ref{eqLemma_wild_ent_1}) with Claims~\ref{Claim_ILikeExpansion}--\ref{Claim_WithoutExpansion}.
\end{proof}

\subsubsection{The occupancy probability}\label{Sec_Lemma_wild}

Let $(\tilde\omega,\tilde\gamma)$ be as in \Sec~\ref{Sec_wild_ent} and
let $\hat\cS(\omega,\gamma)$ be the set of all $\ell\in T^*\setminus\tilde\cS(\omega,\gamma)$ for which there is $j\in[k_\ell]$ such that
	\begin{equation}\label{eqhatSomegagamma}
	\max\{|\gamma_{\ell,j}^{\rot\cyan}-\bar\gamma_{\ell,j}^{\rot\cyan}|,|\gamma_{\ell,j}^{\cyan\rot}-\bar\gamma_{\ell,j}^{\cyan\rot}|,
		|\gamma_{\ell,j}^{\rot\y}-\bar\gamma_{\ell,j}^{\rot\y}|,|\gamma_{\ell,j}^{\y\rot}-\bar\gamma_{\ell,j}^{\y\rot}|\}>k^{-25}2^{-k}.
	\end{equation}
Moreover, let $(\hat\omega,\hat\gamma)$ be such that for all $t\in T$, $h\in[d_t]$, $\ell\in T^*$ and $j\in[k_\ell]$
	 \begin{align*}
	 \hat\gamma_\ell&\doteq\begin{cases}\bar\gamma_\ell&\mbox{ if }\ell\in\hat\cS(\omega),\\
	 		\tilde\gamma_\ell&\mbox{ if }\ell\not\in\hat\cS(\omega),
	 		\end{cases}
	 	\qquad
	\hat\omega_t^{z_1z_2}=\tilde\omega_t^{z_1z_2}\mbox{ for }z_1,z_2\in\spins,\qquad
		\hat\omega_{\ell,j}^{z_1z_2}=\tilde\omega_{\ell,j}^{z_1z_2}\mbox{ for }z_1,z_2\in\{\purpur,\y\},\\
	\hat\omega_{t,h}^{z_1z_2}&=\sum_{(\ell',j')\in\partial(t,h)}\frac{m_{\ell'}}{n_t}\hat\gamma_{\ell',j'}^{z_1z_2}
		\qquad\mbox{ for all }(z_1,z_2)\in\{(\rot,\rot),(\rot,\cyan),(\cyan,\rot)\},\\
	\hat\omega_{t,h}^{\rot\y}&=\sum_{(\ell',j')\in\partial(t,h)}\frac{m_{\ell'}}{n_t}\brk{\hat\gamma_{\ell',j'}^{\rot\y}+\sum_{j''\neq j'}\hat\gamma_{\ell,j',j''}^{\y\y}},
		\qquad
		\hat\omega_{t,h}^{\y\rot}=\sum_{(\ell',j')\in\partial(t,h)}\frac{m_{\ell'}}{n_t}\brk{\hat\gamma_{\ell',j'}^{\y\rot}+\sum_{j''\neq j'}\hat\gamma_{\ell',j'',j'}^{\y\y}}.
	 \end{align*}
In this section we prove

\begin{lemma}\label{Claim_fixRedCyan}
If $(\omega,\gamma)$ is such that~(\ref{eqLemma_wild}) and~(\ref{eqLemma_reduceDisc}) hold and
$\Vol(\cS(\omega))\leq \exp(-\sqrt k)$ but $\hat\cS(\omega,\gamma)\neq\emptyset$,
then
	$F(\hat\omega,\hat\gamma)\geq F(\omega,\gamma)+\Omega(1)$.
\end{lemma}

For $t\in T$ let $\cY_t$ be the set of all $\ell\in\hat\cS(\omega,\gamma)\cup\tilde\cS(\omega,\gamma)$ such that $(\ell,j)\in\partial(t,h)$ for some $h\in[d_t]$, $j\in[k_\ell]$.
Let $Y_t=\frac{m}{n_t}\Vol(\cY_t)$.

\begin{claim}\label{Claim_fixRedCyan_val}
For all $\ell\in\hat\cS(\omega,\gamma)$ we have
	$\Fvall(\hat\omega_\ell,\hat\gamma_\ell)-\Fvall(\tilde\omega_\ell,\tilde\gamma_\ell)\geq
		\tilde\Omega_k(2^{-k})$.
\end{claim}
\begin{proof}
Let $\ell\in\hat\cS(\omega,\gamma)$.
For $\alpha\in[0,1]$ let $\gamma_\ell(\alpha)=\alpha\hat\gamma_\ell+(1-\alpha)\tilde\gamma_\ell$.
\Lem~\ref{Lemma_implicit2} implies that
 there
exists $q_\ell(\alpha)=q_\ell(\tilde\omega_\ell,\gamma_\ell(\alpha))$
such that $\Fvall(\tilde\omega_\ell,\gamma_\ell(\alpha))=\fvall(\tilde\omega_\ell,\gamma_\ell(\alpha),q_\ell(\alpha))+o(1)$.
In particular,
	\begin{equation}\label{eqClaim_fixRedCyan666}
	\Fvall(\hat\omega_\ell,\hat\gamma_\ell)-\Fvall(\tilde\omega_\ell,\tilde\gamma_\ell)=
	\fvall(\tilde\omega_\ell,\gamma_\ell(1),q_\ell(1))-\fvall(\tilde\omega_\ell,\gamma_\ell(0),q_\ell(0)).
	\end{equation}
Estimating the differentials of the implicit parameter $q_\ell(\alpha)$ via \Lem~\ref{Lemma_implicit2}, we find
	\begin{equation}\label{eqClaim_fixRedCyan2}
	\frac{\partial}{\partial\alpha}\fvall(\tilde\omega_\ell,\gamma_\ell(\alpha),q_\ell(\alpha))=\tilde\Omega_k(2^{-k}).
	\end{equation}
The claim follows by combining~(\ref{eqClaim_fixRedCyan666}) and~(\ref{eqClaim_fixRedCyan2}).
\end{proof}

\begin{claim}\label{Claim_fixRedCyan_small}
Assume that $t\in T$ is such that $Y_t<2^{k/4}$.
Then $\Foct(\hat\omega,\hat\gamma)-\Foct(\omega,\gamma)\leq Y_t\tilde O_k(4^{-k}).$
\end{claim}
\begin{proof}
For $(z_1,z_2)\in\{(\rot,\cyan),(\cyan,\rot),(\rot,\y),(\y,\rot)\}$ let
$I_t^{z_1z_2}$ be the set of all $i\in[d_t]$ such that $|\omega_{t,i}^{z_1z_2}-\bar\omega_{t,i}^{z_1z_2}|\geq k^{-20}$.
Then
	\begin{align*}
	k^{-20}|I_t^{z_1z_2}|&\leq\sum_{i\in[d_t]}|\omega_{t,i}^{z_1z_2}-\bar\omega_{t,i}^{z_1z_2}|
		=\sum_{i\in[d_t]}\abs{\bar\omega_{t,i}^{z_1z_2}-\sum_{(\ell,j)\in\partial(t,i)}\frac{m_\ell}{n_t}\omega_{\ell,j}^{z_1z_2}}\\
		&\leq\sum_{i\in[d_t]}\sum_{(\ell,j)\in\partial(t,i)}\frac{m_\ell}{n_t}\abs{\bar\omega_{t,i}^{z_1z_2}-\omega_{\ell,j}^{z_1z_2}}\leq\frac{d_t}{k^{25}}+2^{2-k}Y_t
				\leq 2d_t/k^{25}.
	\end{align*}
Hence, the set $I_t=I_t^{\rot\cyan}\cup I_t^{\rot\y}\cup I_t^{\cyan\rot}\cup I_t^{\y\rot}$ has size $|I_t|\leq d_t/k^{4}$.
Moreover, we have $|\omega_t^{z_1z_2}-\frac14|<1/k$ for all $z_1,z_2\in\{0,1\}$ because otherwise $\ell\in\tilde\cS(\omega)$
for all $\ell$ that feature a literal of type $t$ and thus $Y_t>2^{k/4}$.
Therefore, \Lem~\ref{Lemma_occImplicit} guarantees that for any $\alpha\in[0,1]$ there exists $q_t(\alpha)$ 
such that $\Foct((1-\alpha)\omega+\alpha\hat\omega)=\focct((1-\alpha)\omega+\alpha\hat\omega,q_t(\alpha))+o(1)$.
Further, since $|\hat\omega_{t,h}^{z_1z_2}-\omega_{t,h}^{z_1z_2}|=O_k(2^{-k})$ for all $z_1,z_2\in\{(\rot,\rot),(\rot,\cyan),(\cyan,\rot),(\rot,\y),(\y,\rot)\}$,
 a direct calculation based on Fact~\ref{Fact_affine} and the estimates of the derivatives of $q_t(\alpha)$ provided by \Lem~\ref{Lemma_occImplicit} yields
	$$\frac{\partial}{\partial\alpha}\focct((1-\alpha)\omega+\alpha\hat\omega,q_t(\alpha))\leq Y_t \tilde O_k(4^{-k}).$$
Hence,
	$\Foct(\hat\omega,\hat\gamma)-\Foct(\omega,\gamma)
		=\focct(\hat\omega_t,q_t(\hat\omega_t))-\focct(\omega_t,q_t(\omega_t))+o(1)
		\leq Y_t \tilde O_k(4^{-k}).$
\end{proof}

\begin{proof}[Proof of \Lem~\ref{Claim_fixRedCyan}]
\Lem~\ref{Lemma_occImplicit} implies that for any $t\in T$ there exists a vector $q_t(\hat\omega,\hat\gamma)$ such that
	$\Foct(\hat\omega)=\focct(\hat\omega,q_t(\hat\omega,\hat\gamma))+o(1).$
Furthermore, the construction of $\hat\omega,\hat\gamma$ ensures that $\focct(\hat\omega,q_t(\hat\omega,\hat\gamma))=O_k(k2^{-k})$ for all $t\in T$.
Hence,
	\begin{equation}\label{eqClaim_fixRedCyan1}
	\Foct(\hat\omega,\hat\gamma)=O_k(k2^{-k})\qquad\mbox{for all }t\in T.
	\end{equation}
Let $T_0$ be the set of all $t\in T$ such that $Y_t<2^{k/4}$ and let $T_1=T\setminus T_0$.
Combining Claim~\ref{Claim_fixRedCyan_small} and~(\ref{eqClaim_fixRedCyan1}), we find
	\begin{align}\nonumber
	\Focc(\hat\omega)-\Focc(\omega)
	&=\sum_{t\in T}\pi_t[\Foct(\hat\omega,\hat\gamma)-\Foct(\omega,\gamma)]
		\leq
		\tilde O_k(2^{-1.1k})
		\sum_{t\in T_0}\pi_tY_t +O_k(k2^{-k})\Vol(T_1)\nonumber\\
		&\leq \tilde O_k(2^{-1.1k})
		\sum_{t\in T}\pi_tY_t\leq\tilde O_k(2^{-1.1k})\frac{m}{n}\Vol(\hat\cS(\omega,\gamma)\cup\tilde\cS(\omega,\gamma)).
			\label{eqClaim_fixRedCyan11}
	\end{align}
On the other hand, 
let $\Delta=F(\hat\omega,\hat\gamma)-\Focc(\hat\omega)-(F(\omega,\gamma)-\Focc(\omega))$.
\Lem~\ref{Lemma_wild_ent} and Claim~\ref{Claim_fixRedCyan_val} imply that
	\begin{align}	\label{eqClaim_fixRedCyan12}
	\Delta&\geq\frac{m}{n}\tilde\Omega_k(2^{-k})\Vol(\hat\cS(\omega,\gamma)\cup\tilde\cS(\omega,\gamma)).
	\end{align}
Combining~(\ref{eqClaim_fixRedCyan11}) and~(\ref{eqClaim_fixRedCyan12}) completes the proof.
\end{proof}

\begin{proof}[Proof of \Lem~\ref{Lemma_wild}]
Assume that $(\omega,\gamma)$ is a wild overlap such that $F(\omega,\gamma)$ is maximum. 
Then \Lem s~\ref{Lemma_rough}, \ref{Lemma_reduceDisc}, \ref{Lemma_wild_ent} and~\ref{Claim_fixRedCyan}
	imply that either $\cS(\omega)\cup\tilde\cS(\omega,\gamma)\cup\hat\cS(\omega,\gamma)=\emptyset$,
	or there exists a tame overlap $(\dot\omega,\dot\gamma)$ such that $F(\dot\omega,\dot\gamma)\geq F(\omega,\gamma)+\Omega(1)$.
In the latter case we are done.
Hence, let us assume that $\cS(\omega)\cup\tilde\cS(\omega,\gamma)\cup\hat\cS(\omega,\gamma)=\emptyset$.
Then $(\omega,\gamma)$ satisfies conditions {\bf TM1--TM2} from the definition of tame and violates either {\bf TM3} or {\bf TM4}.

If $\cS(\omega)\cup\tilde\cS(\omega,\gamma)\cup\hat\cS(\omega,\gamma)=\emptyset$,
then \Lem s~\ref{Lemma_implicit2} and~\ref{Lemma_occImplicit} show that there exist $q_t,q_\ell$ for each $t\in T$, $\ell\in T^*$
such that $\Fvall(\omega_\ell,\gamma_\ell)=\fvall(\omega_\ell,\gamma_\ell,q_\ell)+o(1)$, $\Foct(\omega_t)=\focct(\omega_t,q_t)+o(1)$.
Therefore, if {\bf TM3} is violated for $(\ell,j,j')$, we obtain
	\begin{align*}
	\abs{\bc{\frac{\partial}{\partial \gamma_{\ell,j,j'}^{\y\y}}-\frac{\partial}{\partial \gamma_{\ell,j}^{\rot\y}}-\frac{\partial}{\partial \gamma_{\ell,j'}^{\y\rot}}}
		\fvall(\omega_\ell,\gamma_\ell,q_\ell)}=\tilde\Omega_k(1).
	\end{align*}
Similarly, if $(\ell,j)$ is a pair for which {\bf TM4} is violated, then subject to the affine relations from Fact~\ref{Fact_affine},
	\begin{align*}
	\abs{\bc{\frac{\partial}{\partial \gamma_{\ell,j}^{\rot\rot}}-\frac{\partial}{\partial \gamma_{\ell,j}^{\rot\cyan}}-\frac{\partial}{\partial \gamma_{\ell,j}^{\cyan\rot}}}
		\fvall(\omega_\ell,\gamma_\ell,q_\ell)}
		&=\tilde\Omega_k(1),&
	\bc{\frac{\partial}{\partial \gamma_{\ell,j}^{\rot\rot}}
		-\frac{\partial}{\partial \gamma_{\ell,j}^{\rot\cyan}}-\frac{\partial}{\partial \gamma_{\ell,j}^{\cyan\rot}}}\focct(\omega_t,q_t)
		&=\frac{m_\ell}{n_t}\tilde O_k(4^{-k}).
	\end{align*}
Hence, in either case there exists an overlap $(\omega',\gamma')$ such that
	$F(\omega',\gamma')\geq F(\omega,\gamma)+\Omega(1)$.
\end{proof}

\begin{proof}[Proof of \Prop~\ref{Prop_secondMoment}]
Because the total number of wild overlaps $(\omega,\gamma)$ is bounded by a polynomial in $n$,
the assertion is immediate from Fact~\ref{Fact_F} and \Lem s~\ref{Lemma_tame} and~\ref{Lemma_wild}.
\end{proof}

\bigskip
\noindent{\bf Acknowledgment.}
This work has benefited from conversations with Dimitris Achlioptas, Florent Krzakala, Guilhem Semerjian and Lenka Zdeborov\'a.
We are also grateful to Victor Bapst, Charilaos Efthymiou, Samuel Hetterich and Felicia Ra\ss mann for their comments on a draft version of this paper.


\begin{appendix}

\section{Symmetric and asymmetric problems}\label{Apx_sym}

\noindent
There is a relatively general and natural way of defining the notion of a symmetric problem.
As asymmetry generally poses a substantial difficulty in random constraint satisfaction problems, and particularly so in random $k$-SAT, we discuss this concept here in a bit of detail.
Suppose that we are given a sequence $(\vec F_N)_N$ of distributions over instances of a constraint satisfaction problem.
For instance, think of $\vec F_N$ as a random $k$-CNF on $N$ variables with a fixed density $r\sim M/N$.
Suppose that the set of variables in the problem instance $\vec F_N$ is a set $\cV_N$ of size $N$,
and assume that each of these variables can take a value from a finite set $\cX$ of possible ``spins'' (in $k$-SAT, this would be $\cX=\cbc{0,1}$).
Let $\cS(\vec F_N)$ be the set of solutions of the random problem instance $\vec F_N$, i.e., the set of assignments $\sigma:\cV_n\ra\cbc\cX$ under
which all the constraints are satisfied.

Suppose that we fix a problem instance $F=\vec F_N$ such that $\cS(\vec F_N)\neq\emptyset$.
Then we can define the {\em marginal distribution $\mu_{x,F}$} of a variable $x\in\cV_N$ by letting
	$$\mu_{x,F}(c)=\frac{\abs{\cbc{\sigma\in\cS(\vec F_N):\sigma(x)=c}}}{|\cS(\vec F_N)|}\qquad(c\in\cX).$$
Thus, $\mu_{x,F}$ is a probability distribution over $\cX$.

Formally, we could call $(\vec F_N)_N$ {\em symmetric} if there is a fixed probability distribution $p$ on $\cX$ such that
	\begin{equation}\label{eqSymmetric1}
	\lim_{N\ra\infty}\frac1N\sum_{x\in\cV_N}\Erw\brk{\norm{\mu_{x,\vec F_N}-p}_{\mathrm{TV}}|\cS(\vec F_N)\neq\emptyset}=0.
	\end{equation}
Here $\norm{\,\cdot\,}_{\mathrm{TV}}$ denotes the total variation distance (although any other norm would do, because $\cX$ is finite).
In words, (\ref{eqSymmetric1}) means that the marginal distribution $\mu_{x,\vec F_N}$ is independent of the variable $x$, at least asymptotically in the limit of large $N$.
Of course, problems such as random graph coloring or random $k$-NAESAT satisfy~(\ref{eqSymmetric1}), with $p$ the uniform distribution over the set $\cX$ of ``spins''.
In addition, also the random $k$-XORSAT problem satisfies~(\ref{eqSymmetric1}) (up to the threshold for the existence of solutions).
By contrast, in the uniformly random $k$-CNF $\PHI$ (\ref{eqSymmetric1}) does not hold.

While~(\ref{eqSymmetric1}) refers to the plain set of solutions, it is also natural to ask if there is symmetry
with respect to {\em covers}.
Of course, the appropriate definition of ``cover'' varies from one CSP to another, as does the notion of what a solution is.
But there are natural ways of defining this term in many problems.
The problem of finding a cover of $\vec F_n$ can then itself be viewed as a random constraint satisfaction problem, where the joker value $*$ is added to the set $\cX$ of spins.
The notion of symmetry can thus be extended to covers.

Interestingly, some problems that are symmetric at the levels of solutions fail to be symmetric at the level of covers.\footnote{This was brought to our attention by Florent Krzakala.}
This is because the marginal probability of being unfrozen (i.e., the probability mass assigned to $*$) may vary from variable to variable.
An example of this seems to be the graph coloring problem on the \Erdos-\Renyi\ random graph $G(n,m)$ (see~\cite{pnas,MM} and the references therein).
By contrast, the random graph coloring problem on random {\em regular} graph is conjectured to be symmetric both on the level of covers and solutions.
Similarly, the problem of finding a cover in random $k$-NAESAT is asymmetric in uniformly random formulas but
symmetric in random regular formulas~\cite{Catching,DSS1}.
The symmetry on the level of solutions is what greatly simplifies the proof in~\cite{Catching} by comparison to the present work.
In addition, the independent set problem on random graphs $G(n,m)$ is asymmetric in the sense of ``solutions'' as well as in the sense of covers.
By contrast, it is symmetric in terms of covers on random regular graphs~\cite{DSS2}.

There is a relatively natural symmetric version of the random $k$-SAT problem.
Namely, let $\PHI_{k,d-\mathrm{reg}}$ denote a $k$-CNF on the variables $V=\cbc{x_1,\ldots,x_N}$
in which each of the $2N$ literals $x_1,\neg x_1,\ldots,x_N,\neg x_N$ occurs exactly $d$ times, chosen uniformly at random among all such formulas.
Hence, $D_{x_i}=D_{\neg x_i}=d$ for all $i$.
In this model, there is no drift towards the (trivial) majority vote assignment.

In effect, it is possible to obtain a ``sharp'' result in this case.
More precisely, the cavity method predicts that near the $k$-SAT threshold all clusters correspond to covers with no more than $2^{-k}N$ variables set to $*$.
Thus, let $\Sigma'(\PHI_{k,d-\mathrm{reg}})$ be the number of covers of  the random formula $\PHI_{k,d-\mathrm{reg}}$ with at most $2^{-k}N$ variables assigned $*$, and let
	$$\textstyle\Xi(k,d)=\lim_{N\ra\infty}\frac1N\ln\Erw[\Sigma'(\PHI_{k,d-\mathrm{reg}})].$$
The arguments that we used to prove \Prop~\ref{Prop_firstMomentFormula} imply that the limit exists.
Furthermore, it is possible to perform a second moment argument along the lines of \Sec~\ref{Sec_smmFull}.
(Actually, both the first and the second moment argument greatly simplify because there is only a {\em single} type.)
The result of this analysis is

\begin{theorem}\label{Thm_reg}
There is a constant $k_0\geq 3$ such that the following is true for all $k\geq k_0$. 
\begin{enumerate}
\item If $d$ is such that $\Xi(k,d)\geq 0$, then $\PHI_{k,d-\mathrm{reg}}$ has an assignment $\sigma:V\ra\cbc{0,1}$
		that satisfies all but $o(n)$ clauses \whp
\item If $d$ is such that $\Xi(k,d)<0$, then \whp\ under any assignment assignment $\sigma:V\ra\cbc{0,1}$ at least $\Omega(n)$ clauses are unsatisfied.
\end{enumerate}
\end{theorem}

\noindent
The random regular $k$-SAT problem was previously studied via the ``vanilla'' second moment method by Rathi, Aurell, Rasmussen and Skoglund~\cite{Rathi}.
In terms of the degree $d$, \Thm~\ref{Thm_reg} improves the bounds that they obtained by an additive constant.

\begin{remark}
In the first part of \Thm~\ref{Thm_reg}, we obtain an assignment
that satisfies a $1-o(1)$-fraction of all clauses rather than an actual satisfying assignment.
This is because there is no counterpart to \Lem~\ref{Lemma_Ehud}  in random regular formulas.
However, we expect that $\PHI_{k,d-\mathrm{reg}}$ has an actual satisfying assignment \whp\ if $d$ such that $\Xi(k,d)>0$.
\end{remark}

\section{Proof of \Prop~\ref{Prop_pruning}}\label{Sec_Prop_pruning}

\noindent
The proof follows arguments developed in~\cite{Barriers,Lenka}.
We continue to  let $\vec D'=(D_l')_{l\in L}$ be a family of independent Poisson variables with mean $\Erw[D_l']=kr/2$ for all $l$.
We recall the following well-known fact.

\begin{lemma}\label{Lemma_Poissonize}
There is a number $C=C(k)>0$ such that for any sequence $(y_l)_{l\in L}$ of integers we have
	$$\textstyle
		\pr\brk{\forall l\in L:D_l=y_l}=\pr\brk{\forall l\in L:y_l=D_l' \bigg| \sum_{l\in L}D_l'=kM}\leq C\sqrt N\cdot\pr\brk{\forall l\in L:D_l'=y_l}.$$
\end{lemma}
\begin{proof}
The first equality is immediate.
The second one follows because $\sum_{l\in L}D_l'$ is Poisson with mean $kM$.
\end{proof}

\begin{lemma}\label{Lemma_PR1}
Let $U_1'$ be the set of all variables 
$x$ such that
$\max\cbc{|D_x-kr/2|,|D_{\neg x}-kr/2|}>k^32^{k/2-1}$.
Then $|U_1'|\leq\exp(-k^{3.9})n$  \whp\
\end{lemma}
\begin{proof}
Let	$U_1'' = \cbc{l\in L:|D_l'-kr/2|>t},$  $t = k^32^{k/2-1}$.
Since the $(D_l')_{l\in L}$ are independent, $|U_1''|$ is a binomial random variable.
Its mean is bounded by
\[
	\Erw|U_1''|
	\le N\pr\brk{|\Po(kr/2)-kr/2|>t}
	\le N \exp\big(t - (kr/2+t)\ln(1 + 2t/kr)\big)
	\le N \exp(-k^4).
\]
Consequently, applying the Chernoff bound to $|U_1''|$, we obtain 
$\pr\brk{|U_1''|>\exp(-k^{3.9})N }\leq\exp(-\Omega(N))$.
Thus, the assertion follows from \Lem~\ref{Lemma_Poissonize}.
\end{proof}

\begin{lemma}\label{Lemma_PR2}
\Whp\ the set $U$ of variables removed by {\bf PR1--PR2} satisfies  $|U|\leq\exp(-k^3)N$.
\end{lemma}
\begin{proof}
Let us consider a modified process in which step {\bf PR1} is replaced by
\begin{description}
\item[PR1'] Initially, let $U=U_1'$ be the set from \Lem~\ref{Lemma_PR1}.
\end{description}
Clearly, the set $U$ of variables removed by {\bf PR1--PR2} is contained in the set $U'$
of variables removed by executing {\bf PR1'} and then {\bf PR2}.

Hence, assume that $|U'|>\exp(-k^3)N$ and let $U_2'\subset U'\setminus U_1'$ contain the first $\exp(-k^3)N$ variables that get removed by {\bf PR2}.
Set $\alpha=\exp(-k^3)$ and $\beta=k^32^{-1+k/2}$. 
By construction, each $x\in U_2'$ occurs in at least $\beta$ clauses that each feature three or more variables from $U_1'\cup U_2'$.
Hence, there are at least $\alpha\beta N/k$ such clauses.
Since by Lemma~\ref{Lemma_PR1} we know that \whp\ $|U_1'| \le \alpha N$, it suffices to prove the following statement.
	\begin{equation}\label{eqLemma_PR2666}
	\parbox{12cm}{\Whp\ the random formula $\PHI$ does not admit a set $Y\subset V$ of size $y\leq 2\alpha N$ and at least
			$y\beta/(2k)$ clauses contain at least three variables from $Y$.}
	\end{equation}

To prove~(\ref{eqLemma_PR2666}), we note that
there are $\bink{N}{y}$  ways of choosing $y$ variables
and $\bink{M}{y\beta/(2k)}$ ways of choosing $y\beta/(2k)$ clauses.
Further, the probability that a random clause contains at least three variables from $Y$ is bounded by $\bink k3(y/N)^3$.
Thus, by the union bound, the independence of the clauses, and  our choice of $\alpha,\beta$, we obtain
	\begin{eqnarray*}
	\pr\brk{\mbox{there is $Y$ as in~(\ref{eqLemma_PR2666})}}&\leq&\sum_{y\leq 2\alpha N}
		\bink{N}{y}\bink{M}{y\beta/(2k)}\brk{\bink k3(y/N)^3}^{y\beta/(2k)}\\
		&\leq&\sum_{y\leq 2\alpha N}\brk{\bcfr{\eul}{2}^2\bcfr{\eul k^4r}{\beta}^{\beta/k}\bcfr yN^{\frac{2\beta}k-2}}^{y/2}=o(1),
	\end{eqnarray*}
thereby proving~(\ref{eqLemma_PR2666}).
\end{proof}

\begin{corollary}\label{Lemma_degSumPR}
Let $U$ be the set of variables removed by {\bf PR1--PR3}.
Then $\sum_{x\in U}D_x+D_{\neg x}\leq\exp(-k^2)N$   \whp
\end{corollary}
\begin{proof}
Let
	$S=\sum_{l\in L}D_l\vecone_{D_l>4^k}$.
Moreover,  let
	$S'=\sum_{l\in L}D_l'\vecone_{D_l'>4^k}.$
The Chernoff bound shows that $\Erw[S']\leq\exp(-k^4)N$ (with room to spare).
Moreover, since $S'$ is a sum of independent random variables with $\Erw[S'] = \Theta_k(N)$ and $\Var[S'] = \Theta_k(N)$, Chebyshev's inequality yields
	$\pr\brk{S'>\exp(-k^3)n} \le O_k(N^{-1}).$
Therefore, by Lemma~\ref{Lemma_Poissonize}
\begin{eqnarray}\nonumber
	\pr\brk{S > \exp(-k^3)N}\leq C \sqrt{N} \cdot \pr\brk{S'>\exp(-k^3)N}=o(1).
	\label{eqLemma_degSumPR1}
\end{eqnarray}
Since \whp\ $|U|\leq2\exp(-k^3)N$ by \Lem s~\ref{Lemma_PR1} and~\ref{Lemma_PR2}, we see that \whp
	\begin{eqnarray*}
	\sum_{x\in U}d_x+d_{\neg x}&\leq& S+\sum_{x\in U}\vecone_{d_x\leq 4^k}d_x+\vecone_{d_{\neg x}\leq 4^k}d_{\neg x}
		\leq S+4^k|U|\leq\exp(-k^2)N,
	\end{eqnarray*}
as desired.
\end{proof}

\begin{lemma}\label{Lemma_pruning_Omega}
If $d^+,d^-$ are such that $|d^\pm-kr/2|\leq k^32^{k/2}$, then
	$\abs{\cbc{l\in L':d_l=d^+,d_{\neg l}=d^-}}=\Omega(N)$.
\end{lemma}
\begin{proof}
Let $\cX$ be the set of variables $x$ with $D_x=d^+$, $D_{\neg x}=d^{-}$.
Combining \Lem~\ref{Lemma_Poissonize} with the Chernoff bound, we see that $\abs\cX=\Omega(N)$ \whp\
Further, with $U$ the set of variables removed by {\bf PR1--PR3}, let $\cX'$ be the set of all $x\in\cX\setminus U$ with $d_x=D_x$, $d_{\neg x}=D_{\neg x}$.
Thus, $\cX'$ contains all $x\in\cX$ that remain unscathed by the process {\bf PR1--PR3}.

Think of {\bf PR2} as removing one clause (that contains at least three variables from $U$) at a time.
By the principle of deferred decisions, at the time when that clause is removed its remaining literals
are random subject to the degree distribution of the literals $x,\neg x$ ($x\in V\setminus U$).
Therefore, \Cor~\ref{Lemma_degSumPR} implies that $\Erw|\cX'|=\Omega(N)$.
Finally, a standard martingale argument implies that $|\cX|=\Erw|\cX'|+o(n)$ \whp
\end{proof}

\begin{lemma}\label{Lemma_pruning_extend}
\Whp\ any satisfying assignment of $\PHI'$ extends to a satisfying assignment of $\PHI$.
\end{lemma}
\begin{proof}
We begin by proving the following fact.
	\begin{equation}\label{eqProofPropPruning1}
	\parbox{14cm}{	\Whp\ there are no sets $I\subset [M]$ and $S\subset V$ such that $|I| = |S| = \alpha N$ with $0 < \alpha\leq\exp(-k^2)N$
		and each clause $\PHI_i$, $i\in I$, contains at least three variables from $S$.}
	\end{equation}
Indeed, by the union bound for any $0 < \alpha \le \exp(-k^2)$ the probability that there exist $I,S$ as above is bounded by
	\begin{eqnarray*}
	\bink{N}{\alpha N}\bink{M}{\alpha N}\brk{\bink k3\alpha^3}^{\alpha N}&\leq&
		\brk{\frac{\eul}\alpha\cdot\frac{\eul r}{\alpha}\cdot(k\alpha)^3}^{\alpha N}
			\leq\brk{\eul^2 kr\alpha}^{\alpha N}.
	\end{eqnarray*}
Summing over $\alpha=i/N\leq\exp(-k^2)$, we obtain~(\ref{eqProofPropPruning1}).

To complete the proof let $I\subset\brk M$ be the set of all indices of clauses that {\bf PR2} removes.
By \Cor~\ref{Lemma_degSumPR} we have \whp\ $|I|\leq\exp(-k^2)N$.
Moreover, each clause $\PHI_i$, $i\in I$, contains at least three variables from $U$.
Hence, (\ref{eqProofPropPruning1}) implies together with the marriage theorem that we can match each clause $\PHI_i$, $i\in I$, to a variable in $U$.
This variable can be set such that $\PHI_i$ is satisfied; we conclude that
any satisfying assignment of $\PHI'$ can be extended to a satisfying assignment of $\PHI$.
\end{proof}

\begin{proof}[Proof of \Prop~\ref{Prop_pruning}]
The first assertion follows from \Lem~\ref{Lemma_pruning_extend} and \Cor~\ref{Lemma_degSumPR} implies the second part of \Prop~\ref{Prop_pruning}.
The third claim follows from \Lem~\ref{Lemma_pruning_Omega} and \Cor~\ref{Lemma_degSumPR}.
\end{proof}

\section{Proof of \Lem~\ref{Lemma_expansion}}\label{Sec_Lemma_expansion}

\noindent
Because \whp\ $\hat\PHI$ is obtained from $\PHI$ by removing no more than $8^{-k}N$ vertices and $8^{-k}M$ edges,
it suffices to establish certain expansion properties for the random formula $\PHI$.
More specifically, to obtain \Lem~\ref{Lemma_expansion} it suffices to prove that $\PHI$ enjoys the following three (stronger) properties \whp

\begin{enumerate}[(i)]
\item Assume that $A\subset L$ is a set of literals such that $|A|\geq0.01N$.
	Let $\cM$ be the set of all clause indices $i\in[M]$ such that $\PHI_i$ contains at least 
	$0.002k$ literals from $A$.
	Then $|\cM|/M\geq1-\exp(-\Omega_k(k))$.
\item Assume that $A,B\subset L$ are disjoint sets of literals such that $|A|,|B|\geq0.93N$.
	Let $\cM$ be the set of all $i\in[m]$ such that $\PHI_i$ contains
	at least $0.41k$ literals from $A$ and at least $0.41k$ literals from $B$.
	Then $|\cM|/M\geq1-k^{-10}$.
\item Assume that $A\subset L$ has size $|A|\leq k^{-8}N$.
	Let $\cM$ be the set of all $i\in[M]$ such that $\PHI_i$ contains at least $0.9k$ literals from $A$.
	Then $|\cM|\leq |A|$.
\end{enumerate}

To prove (i), let $a=0.01$. By the Chernoff bound there exists $\gamma > 0$ such that $\pr\brk{\Bin(k,a/2)<0.002k}\leq\exp(-\gamma k)$. We may assume that, say, $\gamma \le 0.1$. Let $\beta=\exp(-\gamma k/2)$. The probability that (i) is violated can be bounded as follows. There are $\bink{2N}{2aN}$ ways to choose a set $A$ of $2aN$ literals and
	$\bink{M}{\beta M}$ ways to choose $\beta=\exp(-\Omega_k(k))$ clauses.
Moreover, the probability that none of these $\beta M$ clauses contains $0.002k$ literals from $A$ is bounded by $\exp(-\gamma k \cdot \beta M)$,
because the literals are chosen independently and uniformly at random. So, the probability that (i) is violated is at most
	$$p=\bink{2N}{2aN}\bink{M}{\beta M}\exp(-\gamma k \beta M).$$
By Fact~\ref{Fact_entropyFunction} and the inequality $H(x) \le x(1-\ln x)$ we obtain
	\begin{align*}
	\frac{\ln p}N&\sim 2H(a)+\frac MN\bc{H(\beta)-\beta\gamma k}\leq 2+\frac{\beta M}N\bc{1-\ln\beta-\gamma k}.
	\end{align*}
However, the last expression is negative whenever $k$ is sufficiently large because $M/N=\Omega_k(2^k)$ and, say,  $\beta\geq 2^{-k/2}$.
Thus, the probability that (i) is violated is bounded by $\exp(-\Omega(N))$.

With respect to (ii), fix two sets $A,B$.
Then by the Chernoff bound the probability that a random clause fails to contain at least $0.41k$ literals from either $A$ or $B$
is bounded by $2\pr\brk{\Bin(k,0.465)<0.41k}\leq\exp(-\gamma k)$ for some constant $\gamma>0$.
Hence, the total number $X(A,B)$ of clauses with this property is a binomial random variable with mean $\Erw[X(A,B)]\leq\exp(-\gamma k)M$.
Consequently, once more by the Chernoff bound and because $M/N=\Omega_k(2^k)$
	\begin{align*}
	\pr\brk{X(A,B)\geq M/k^{10}}\leq\exp\brk{-M/k^{10}}\leq 5^{-N}.
	\end{align*}
Since the total number of ways of choosing $A,B$ is bounded by $4^N$, (ii) holds \whp

To establish (iii), fix a set $A$ of size $|A|=2aN$ with $0<a\leq k^{-8}$.
Let $X(A)$ be the number of clauses with at least $0.9k$ literals from $A$.
Then $X(A)$ has distribution $\Bin(M,q)$ with $q=\pr\brk{\Bin(k,a)\geq0.9k}$.
The Chernoff bound guarantees  that $q\leq a^{0.8k}$ whenever $k$ is sufficiently large.
Therefore, applying Chernoff once more, we find
	\begin{align*}
	\pr\brk{X(A)\geq aN}\leq\exp\brk{-aN \, \ln\frac{aN}{\eul qM}}\leq\exp(0.7ak\ln a\cdot N).
	\end{align*}
Since the number of possible sets $A$ is bounded by $\bink{2N}{2aN}\leq\exp(2a(1-\ln a)N)$, the assertion follows from the union bound.

\section{Notation index} \label{sec:notation}

\noindent{\bf Random formulas:}

\noindent
\begin{tabular}{lll}
\parbox{1.9cm}{\em Symbol}&\parbox{11.5cm}{\em  Description}&\parbox{5cm}{\em  Definition}\\
$\PHI$&random $K$-SAT formula with $N$ variables and $M$ clauses&\Sec~\ref{Sec_intro}\\
$\PHI'$&pruned random formula&\Sec~\ref{Sec_pruning}\\
$n$&number of variables of $\PHI'$&\Sec~\ref{Sec_pruning}\\
$m$&number of variables clauses $\PHI'$&\Sec~\ref{Sec_pruning}\\
$d_l$&degree of literal $l$ in $\PHI'$&\Sec~\ref{Sec_pruning}\\
$\cL'$&set of literal clones, $\cL'=\bigcup_{l\in L'}\cbc{l}\times[d_l]$&\Sec~\ref{Sec_pruning}\\
$\hat\PHI$&random formula (configuration model)&Eq.\ (\ref{eqRandomBijection})\\
\end{tabular}

\medskip

\noindent{\bf Colors:}

\noindent
\begin{tabular}{lll}
 \parbox{1.9cm}{$\rot$} & \parbox{11.5cm}{red, representing a true and ``blocking'' literal occurrence}&\Sec~\ref{Sec_colorCode}\\
$\blau$&blue, representing a true but ``non-blocking'' literal occurrence&\Sec~\ref{Sec_colorCode}\\
$1 = \{\rot, \blau\}$ &  represents a true literal occurrence&\Sec~\ref{Sec_colorCode}\\
$\grun$&green, an occurrence of a literal set to the joker value $*$&\Sec~\ref{Sec_colorCode}\\
$\y$&yellow, an occurrence of a false literal&\Sec~\ref{Sec_colorCode}\\
$\cyan = \{\blau, \grun\}$&cyan: either blue or green&\Sec~\ref{Sec_colorCode}\\
$\purpur = \{\rot, \blau, \grun\}$&purple: either red, blue or green&\Sec~\ref{Sec_colorCode}\\
\end{tabular}

\medskip
\noindent{\bf Types:}

\noindent
\begin{tabular}{lll}
$\theta_l$& type of literal $l$, comprising of $d_l, d_{\neg l}$ and distributions $(\theta_{l,j})_{j \in [d_l]}$, $(\theta_{\neg l,j})_{j \in [d_{\neg l}]}$ & Section~\ref{Sec_Types} \\
\parbox{1.9cm}{$T$}&\parbox{11.5cm}{set of literal types}&\Def~\ref{Def_typeAssignement}\\
$[T]$&set of pairs $\{t,\neg t\}$, $t\in T$&\Sec~\ref{Sec_observations}\\
$T^*$&set of clause types, consisting of all litelat types in the clause&\Sec~\ref{Sec_Types}\\
$k_\ell$&length of a clause of type $\ell$; $k_\ell\in\{k-2,k-1,k\}$&\Sec~\ref{Sec_Types}\\
$t^z$&probability of color $z$ under $t\in T$&eq.\ (\ref{eqDefMyTheta})\\
$\ell_j^z$&probability of color $z$ under $\ell_j$, $\ell\in T^*$&eq.\ (\ref{Lemma_Lambda1})--(\ref{Lemma_Lambda2})\\
$\partial(t,h)$&set of ``clause slots'' $(\ell,j)$ where the $h$th clone of a type $t$ literal may occur&\Sec~\ref{Sec_Types}
\end{tabular}

\medskip
\noindent{\bf First moment computation:}

\noindent
\begin{tabular}{lll}
\parbox{1.9cm}{$q_{\ell,j}^{\purpur},q_{\ell,j}^{\y}$}&\parbox{11.5cm}{auxiliary parameters associated with clause type $\ell$ and $j\in[k_\ell]$}&Prop.~\ref{Prop_firstMomentFormula}\\
$q_{t,h}^{\rot}$&auxiliary parameters associated with literal type $t$ and $h\in[d_t]$&Prop.~\ref{Prop_firstMomentFormula}\\
$s_t$&probability that a literal of type $t$ set to $1$ is ``blocked'' &Figure~\ref{Fig_firstMomentFormula}\\
$g_{\ell}^{\cyan}$&probability that a clause of type $\ell$ contains two cyan literals&Figure~\ref{Fig_firstMomentFormula}\\
$g_{\ell,j}^{\rot}$&probability of containing a red literal in position $j$ and yellow ones elsewhere&Figure~\ref{Fig_firstMomentFormula}\\
\end{tabular}

\medskip
\noindent{\bf Second moment computation:}

\noindent
\begin{tabular}{lll}
\parbox{1.9cm}{$\omega,\gamma$} & \parbox{11.5cm}{overlaps}&\Sec~\ref{Sec_theOverlap}\\
$\bar\omega,\bar\gamma$&average overlaps&\Sec~\ref{Sec_theOverlap}\\
$q_{\ell,j}^{z_1z_2}$&auxiliary parameters associated with clause type $\ell$, $j\in[k_\ell]$ and $z_1,z_2\in\{\purpur,\y\}$&\Lem~\ref{Lemma_implicit2}\\
$g_{\ell,j}^{z_1z_2}$&success probability for the validity problem&Figure~\ref{Fig_gell}\\
$g_{\ell,j,j'}^{\y\y}$&success probability for the validity problem&Figure~\ref{Fig_gell}\\
$q_{t,h}^{z_1z_2}$&auxiliary parameters associated with literal type $t$, $h\in[d_t]$ and colors $z_1,z_2$&\Lem~\ref{Lemma_occImplicit}\\
$s_t^{z_1z_2}$&success probabilities for the occupancy problem&Figure~\ref{Fig_occs}\\ 
\end{tabular}

\end{appendix}

\end{document}